\documentclass[11pt,reqno]{amsart}

\newtheorem{thm}{Theorem}[section]
\newtheorem*{thm*}{Theorem}
\newtheorem{lem}[thm]{Lemma}
\newtheorem*{lem*}{Lemma}

\newtheorem{cor}[thm]{Corollary}
\newtheorem{claim}[thm]{Claim}
\newtheorem{prop}[thm]{Proposition}

\theoremstyle{definition}

\newtheorem{assump}[thm]{Assumption}
\newtheorem*{case*}{Case}

\newtheorem{defn}[thm]{Definition}
\newtheorem*{defn*}{Definition}

\newtheorem*{exmp*}{Example}

\renewcommand{\thestep}{}

\theoremstyle{remark}

\renewcommand{\thecase}{}

\newtheorem{rmk}[thm]{Remark}
\newtheorem*{rmk*}{Remark}

\makeatletter
\def\alphenumi{
  \def\theenumi{\alph{enumi}}
  \def\p@enumi{\theenumi}
  \def\labelenumi{(\@alph\c@enumi)}}
\makeatother

\makeatletter
\def\thecase{\@arabic\c@case}
\makeatother

\makeatletter
\def\thestep{\@arabic\c@step}
\makeatother

\newcount\hh
\newcount\mm
\mm=\time
\hh=\time
\divide\hh by 60
\divide\mm by 60
\multiply\mm by 60
\mm=-\mm
\advance\mm by \time
\def\hhmm{\number\hh:\ifnum\mm<10{}0\fi\number\mm}

\let\oldmarginpar\marginpar
\renewcommand\marginpar[1]{\-\oldmarginpar[\raggedleft\footnotesize #1]%
{\raggedright\footnotesize #1}}

\newcommand\HH{\mathbb{H}}

\newcommand\NN{\mathbb{N}}

\newcommand\RR{\mathbb{R}}

\newcommand\fa{{\mathfrak{a}}}

\newcommand\fw{{\mathfrak{w}}}

\newcommand\sA{{\mathscr{A}}}
\newcommand\sB{{\mathscr{B}}}

\newcommand\sH{{\mathscr{H}}}

\newcommand\sO{{\mathscr{O}}}

\newcommand\eps{\varepsilon}

\newcommand\less{\setminus}

\newcommand\diam{\operatorname{diam}}

\newcommand\dist{\operatorname{dist}}

\DeclareMathOperator{\height}{height}

\newcommand\loc{\operatorname{loc}}

\newcommand\supp{\operatorname{supp}}

\newcommand\vol{\operatorname{vol}}

\numberwithin{equation}{section}

\usepackage{amssymb}
\usepackage{graphicx}
\usepackage{hyperref}
\usepackage{mathrsfs}
\usepackage[usenames]{color}
\usepackage{url}
\usepackage{verbatim}

\hypersetup{pdftitle={Degenerate-elliptic operators in mathematical finance and higher-order regularity for solutions to variational equations}}
\hypersetup{pdfauthor={Paul M. N. Feehan and Camelia A. Pop}}

\begin{document}

\title[Higher-order regularity for solutions to variational equations]
{Degenerate-elliptic operators in mathematical finance and higher-order regularity for solutions to variational equations}

\author[P. M. N. Feehan]{Paul M. N. Feehan}
\address{Department of Mathematics, Rutgers, The State University of New Jersey, 110 Frelinghuysen Road, Piscataway, NJ 08854-8019}
\email{feehan@math.rutgers.edu}

\author[C. A. Pop]{Camelia A. Pop}
\address{Department of Mathematics, University of Pennsylvania, 209 South 33rd Street, Philadelphia, PA 19104-6395}
\email{cpop@math.upenn.edu}
\date{November 13, 2014. To appear in Advances in Differential Equations. Incorporates final galley proof corrections corresponding to published version}

\begin{abstract}
We establish higher-order weighted Sobolev and H\"older regularity for solutions to variational equations defined by the elliptic Heston operator, a linear second-order degenerate-elliptic operator arising in mathematical finance \cite{Heston1993}. Furthermore,  given $C^\infty$-smooth data, we prove $C^\infty$-regularity of solutions up to the portion of the boundary where the operator is degenerate. In mathematical finance, solutions to obstacle problems for the elliptic Heston operator correspond to value functions for perpetual American-style options on the underlying asset.
\end{abstract}

%

\subjclass[2010]{Primary 35J70, 49J40, 35R45; Secondary 60J60}

\keywords{Campanato space, degenerate-elliptic differential operator, degenerate diffusion process, Heston stochastic volatility process, H\"older regularity, mathematical finance, Schauder a priori estimate, Sobolev regularity, variational equation, weighted Sobolev space}

\thanks{The first author was partially supported by NSF grant DMS-1059206 and the Max Planck Institut f\"ur Mathematik in der Naturwissenschaft.
}

\maketitle
\tableofcontents
\listoffigures

\section{Introduction}
\label{sec:Introduction}
Suppose $\sO\subseteqq\HH$ is a domain (possibly unbounded) in the open upper half-space $\HH := \RR^{d-1}\times\RR_+$ (where $d\geq 2$ and $\RR_+ := (0,\infty)$), and $\partial_1\sO := \partial\sO\cap\HH$ is the portion of the boundary $\partial\sO$ of $\sO$ which lies in $\HH$, and $\partial_0\sO $ is the interior of $\partial\HH\cap\partial\sO$, where $\partial\HH = \RR^{d-1}\times\{0\}$ is the boundary of $\bar\HH := \RR^{d-1}\times\bar\RR_+$ and $\bar\RR_+ := [0,\infty)$. We allow $\partial_0\sO$ to be non-empty and consider a second-order, linear elliptic differential operator, $A$, on $\sO$ which is degenerate along $\partial_0\sO$. In this article, when $d=2$ and the operator $A$ is given by \eqref{eq:OperatorHestonIntro}, we prove higher-order regularity up to the boundary portion, $\partial_0\sO$ --- as measured by certain weighted Sobolev spaces, $\sH^{k+2}(\sO,\fw)$ (Definition \ref{defn:HkWeightedSobolevSpaceNormPowery}),
and weighted H\"older spaces, $C^{k,2+\alpha}_s(\underline{\sO})$ (Definition \ref{defn:DH2spaces})
--- for suitably defined \emph{weak} solutions, $u \in H^1(\sO,\fw)$ (see \eqref{eq:H1WeightedSobolevSpace} for its definition),
to the elliptic boundary value problem,
\begin{align}
\label{eq:IntroBoundaryValueProblem}
Au &= f \quad \hbox{(a.e.) on }\sO,
\\
\label{eq:IntroBoundaryValueProblemBC}
u &= g \quad \hbox{on } \partial_1\sO,
\end{align}
where $f:\sO\to\RR$ is a source function and the function $g:\partial_1\sO\to\RR$ prescribes a Dirichlet boundary condition. We denote $\underline{\sO} := \sO\cup\partial_0\sO$ throughout our article. Furthermore, when $f \in C^\infty(\underline\sO)$, we will also show that $u\in C^\infty(\underline\sO)$ (see Corollary \ref{cor:CinftyGlobal}). Since $A$ becomes degenerate along $\partial_0\sO$, such regularity results do not follow from the standard theory for strictly elliptic differential operators \cite{GilbargTrudinger, Krylov_LecturesHolder}.

Because $\kappa\theta>0$ (see Assumption \ref{assump:HestonCoefficients} below), \emph{no boundary condition} is prescribed for the equation \eqref{eq:IntroBoundaryValueProblem} along $\partial_0\sO$. Indeed, we recall from \cite{Daskalopoulos_Feehan_statvarineqheston} that the problem \eqref{eq:IntroBoundaryValueProblem}, \eqref{eq:IntroBoundaryValueProblemBC} is well-posed, given $f\in L^2(\sO,\fw)$ and $g\in H^1(\sO,\fw)$ obeying mild pointwise growth conditions, when we seek weak solutions in $H^1(\sO,\fw)$ or strong solutions in $H^2(\sO,\fw)$. The \emph{elliptic Heston operator} is defined by
\begin{equation}
\label{eq:OperatorHestonIntro}
Av := -\frac{y}{2}\left(v_{xx} + 2\varrho\sigma v_{xy} + \sigma^2 v_{yy}\right) - \left(c_0-q-\frac{y}{2}\right)v_x - \kappa(\theta-y)v_y + c_0v, \quad v\in C^\infty(\HH),
\end{equation}
and $-A$ is the generator of the two-dimensional Heston stochastic volatility process with killing \cite{Heston1993}, a degenerate diffusion process well known in mathematical finance and a paradigm for a broad class of degenerate Markov processes, driven by $d$-dimensional Brownian motion, and corresponding generators which are degenerate-elliptic integro-differential operators. The coefficients of $A$ are required to satisfy the

\begin{assump}[Ellipticity condition for the coefficients of the Heston operator]
\label{assump:HestonCoefficients}
The coefficients defining $A$ in \eqref{eq:OperatorHestonIntro} are constants obeying
\begin{equation}
\label{eq:EllipticHeston}
\sigma \neq 0, \quad -1< \varrho < 1,
\end{equation}
and $\kappa>0$, $\theta>0$, $c_0\geq 0$, and\footnote{Although $q$ has a financial interpretation as a dividend yield, which is non-negative, our analysis allows $q\in\RR$.} $q \in \RR$.
\end{assump}

In \cite{Feehan_Pop_regularityweaksoln}, we proved that a weak solution, $u \in H^1(\sO,\fw)$, to \eqref{eq:IntroBoundaryValueProblem}, \eqref{eq:IntroBoundaryValueProblemBC} is H\"older continuous up to $\partial\sO$ in the sense that $u\in C^\alpha_{s,\loc}(\bar\sO)$
(Definition \ref{defn:Calphas}), while in \cite{Daskalopoulos_Feehan_statvarineqheston}, we proved that $u \in H^2(\sO,\fw)$, for suitable $f$ and $g$ in both cases. Before describing our main results, we provide in \S \ref{subsec:Motivations} some motivations for our article. In \S \ref{subsec:Summary}, we state the main results of our article and set them in context in \S \ref{subsec:Survey}, where we provide a survey of previous related research by other authors. We point out some of the mathematical difficulties and issues of broader interest in \S \ref{subsec:Highlights}. The results of this article may be generalized to a broader class of degenerate-elliptic operators and expected extensions of our results to such a class are discussed in \S \ref{subsec:Extensions}. We provide a guide in \S \ref{subsec:Guide} to the remainder of this article.  We refer the reader to \S \ref{subsec:Notation} for our notational conventions.

\subsection{Motivations}
\label{subsec:Motivations}
In \cite{Feehan_classical_perron_elliptic, Feehan_Pop_elliptichestonschauder}, we give a purely Schauder-theoretic approach to some of the regularity and existence results achieved by variational methods in this article and its companion articles \cite{Daskalopoulos_Feehan_statvarineqheston, Feehan_Pop_regularityweaksoln}. However, these results complement each other in some important respects. As noted in \S \ref{subsec:Highlights} and also explained in \cite{Feehan_classical_perron_elliptic, Feehan_Pop_elliptichestonschauder}, it appears difficult to 
adapt traditional Schauder methods such as those in \cite[\S 6.3 and \S 6.4]{GilbargTrudinger} to prove that the solutions to \eqref{eq:IntroBoundaryValueProblem}, \eqref{eq:IntroBoundaryValueProblemBC} are continuous up to the domain corner points,
$\overline{\partial_0\sO}\cap\overline{\partial_1\sO}$, let alone H\"older continuous as we show in \cite{Feehan_Pop_regularityweaksoln}, or even smooth. (The question of regularity up to domain corner points does not arise in \cite{DaskalHamilton1998, Koch}.) The variational approach using weighted Sobolev spaces \cite{Daskalopoulos_Feehan_statvarineqheston, Feehan_Pop_regularityweaksoln}, which we develop further in this article, may provide further insight into how a Schauder-theoretic approach could be strengthened to give better regularity near the corner points, whether using a blend of variational, weighted Sobolev space and Schauder methods via Campanato spaces \cite{Troianiello} and ideas of Caccioppoli \cite{Miranda} or more creative choices of barrier functions than those used for strictly elliptic operators \cite[\S 6.3]{GilbargTrudinger}.

The regularity and existence results developed in this article are applied in the approach of Daskalopoulos and Feehan \cite{Daskalopoulos_Feehan_statvarineqheston, Daskalopoulos_Feehan_optimalregstatheston} in their proofs of regularity of a variational (that is, weak) solution in a weighted Sobolev space, $H^1(\sO,\fw)$, to an obstacle problem,
\begin{equation}
\label{eq:Elliptic_obstacle_problem}
\min\{Au-f, \ u-\psi\} = 0 \quad \hbox{a.e. on }\sO,
\end{equation}
with partial Dirichlet boundary condition \eqref{eq:IntroBoundaryValueProblemBC}, given a suitably regular obstacle function, $\psi:\sO\cup\partial_1\sO\to\RR$, which is compatible with $g$ in the sense that
\begin{equation}
\label{eq:Boundarydata_obstacle_compatibility}
\psi\leq g \quad\hbox{on } \partial_1\sO.
\end{equation}
In \cite{Daskalopoulos_Feehan_statvarineqheston}, Daskalopoulos and Feehan apply the regularity results in this article and the penalization method \cite{Bensoussan_Lions, Friedman_1982, Rodrigues_1987} to prove existence of strong solutions to the obstacle problem \eqref{eq:Elliptic_obstacle_problem}, \eqref{eq:IntroBoundaryValueProblemBC}, namely, a solution belonging to a weighted Sobolev space, $H^2(\sO,\fw)$, while in \cite{Daskalopoulos_Feehan_optimalregstatheston} they combine the existence and regularity results in this article and \cite{Feehan_Pop_elliptichestonschauder} with ideas of Caffarelli \cite{Caffarelli_jfa_1998} and Daskalopoulos and Hamilton \cite{DaskalHamilton1998} to prove that a strong solution to this obstacle problem has the optimal regularity, $C^{1,1}_s(\underline\sO)$. We recall from \cite[Definition 2.2]{Daskalopoulos_Feehan_optimalregstatheston} that $u \in C^{1,1}_s(\bar \sO)$ if $u$ belongs to  $C^{1,1}(\sO) \cap C^1(\bar \sO)$ and
$$
\|u\|_{C^{1,1}_s(\bar \sO)} := \|yD^2 u\|_{L^\infty(\sO)} +  \|Du\|_{C(\bar\sO)} +  \|u\|_{C(\bar\sO)}  < \infty.
$$
We say that $u \in C^{1,1}_s(\underline\sO)$ if $u \in C^{1,1}_s(\bar U)$ for every precompact subdomain $U \Subset \underline\sO$. In mathematical finance, solutions to obstacle problems for the elliptic Heston operator correspond to value functions for perpetual American-style options on the underlying asset \cite[\S 8.3]{Shreve2}.

\subsection{Summary of main results}
\label{subsec:Summary}
We summarize our main results concerning interior higher-order Sobolev regularity in \S \ref{subsubsec:HigherOrderSobolevRegularity}, while our results on interior higher-order H\"older regularity are given in \S \ref{subsubsec:HigherOrderHolderRegularity}.  Here, our use of the term ``interior'' is in the sense intended by \cite{DaskalHamilton1998}, for example, $U\subset\sO$ is an \emph{interior subdomain} of a domain $\sO\subseteqq\HH$ if $\bar U \subset \underline{\sO}$ and by ``interior regularity'' of a function $u$ on $\sO$, we mean regularity of $u$ up to $\partial_0\sO$ --- see Figure \ref{fig:higher_order_heston_regularity_regions}.

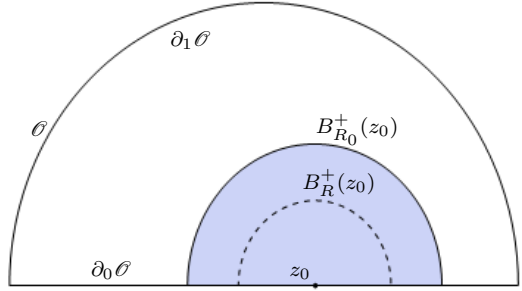
\begin{figure}[htbp]
\centering
\begin{picture}(200,200)(0,0)
\put(0,0){\includegraphics[height=200pt]{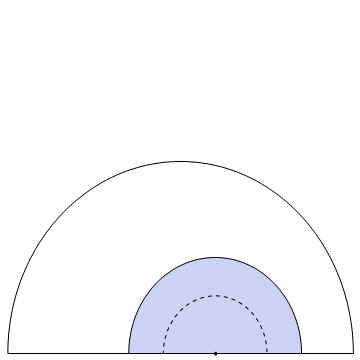}}
\put(12,60){$\scriptstyle \sO$}
\put(35,7){$\scriptstyle \partial_0\sO$}
\put(65,94){$\scriptstyle \partial_1\sO$}
\put(120,61){$\scriptstyle B_{R_0}^+(z_0)$}
\put(115,40){$\scriptstyle B_R^+(z_0)$}
\put(110,7){$\scriptstyle z_0$}
\end{picture}
\caption{Boundaries and regions in the statement of Theorem \ref{thm:HkSobolevRegularityInterior}.}
\label{fig:higher_order_heston_regularity_regions}
\end{figure}

The conditions \eqref{eq:EllipticHeston} ensure that $y^{-1}A$ is strictly elliptic on $\HH$,
that is there is a positive constant, $\nu_0=\nu_0(\sigma, \rho)$, such that
\begin{equation}
\label{eq:HestonModulusEllipticity}
\frac{y}{2}(\xi_1^2 + 2\varrho\sigma\xi_1\xi_2 + \sigma^2\xi_2^2) \geq \nu_0 y(\xi_1^2 + \xi_2^2), \quad\forall\, (\xi_1,\xi_2) \in \RR^2,
\end{equation}
It is convenient to denote the sum of the absolute values of the coefficients of the Heston operator, $A$, in \eqref{eq:OperatorHestonIntro} by
\begin{equation}
\label{eq:DefnXi}
\Lambda := 1 + 2|\varrho\sigma| + \sigma^2 + \kappa\theta + |r-q| + r,
\end{equation}
keeping in mind potential generalizations of our results in this article from the Heston operator to operators of the form \eqref{eq:OperatorHestonIntroHigherDimension}.

\subsubsection{Higher-order interior Sobolev regularity}
\label{subsubsec:HigherOrderSobolevRegularity}
We explain in \S \ref{subsec:HolderRegularityPreliminaries} how solutions, $u \in H^1(\sO,\fw)$, to the variational equation \eqref{eq:HestonVariationalEquation} defined by the operator, $A$, may be interpreted as weak solutions to \eqref{eq:IntroBoundaryValueProblem}.
We review our definitions of weighted Sobolev spaces from \cite[Definition 2.20]{Daskalopoulos_Feehan_statvarineqheston}. For $1\leq p<\infty$, let
\begin{align}
\label{eq:LqWeightedSpace}
L^p(\sO,\fw) &:= \{u\in L^1_{\loc}(\sO): \ \|u\|_{L^p(\sO,\fw)} < \infty\},
\\
\label{eq:H1WeightedSobolevSpace}
H^1(\sO,\fw) &:= \{u \in W^{1,2}_{\loc}(\sO): \ \|u\|_{H^1(\sO,\fw)} < \infty\},
\end{align}
where
\begin{align}
\label{eq:LqNormHeston}
\|u\|_{L^p(\sO,\fw)}^p &:= \int_\sO |u|^p\fw\,dx\,dy,
\\
\label{eq:H1NormHeston}
\|u\|_{H^1(\sO,\fw)}^2 &:= \int_\sO\left(y|Du|^2 + (1+y)u^2\right)\fw\,dx\,dy,
\end{align}
with weight function $\fw:\HH\to(0,\infty)$ given by\footnote{In \cite{Daskalopoulos_Feehan_statvarineqheston}, we used the equivalent factor, $|x|$, but we use $\sqrt{1+x^2}$ here since the resulting weight is in $C^\infty(\HH)$.}
\begin{equation}
\label{eq:HestonWeight}
\fw(x,y) := y^{\beta-1}e^{-\gamma\sqrt{1+x^2}-\mu y}, \quad (x,y) \in \HH,
\end{equation}
where
\begin{equation}
\label{eq:DefnBetaMu}
\beta := \frac{2\kappa\theta}{\sigma^2} \quad\hbox{and}\quad \mu := \frac{2\kappa}{\sigma^2},
\end{equation}
and $0<\gamma<\gamma_0(A)$, where $\gamma_0$ depends only on the constant coefficients of $A$ in \eqref{eq:OperatorHestonIntro}. We denote $H^0(\sO,\fw) = L^2(\sO,\fw)$.

\begin{rmk}[Role of $\gamma$]
\label{rmk:Rolegamma}
In \cite{Daskalopoulos_Feehan_statvarineqheston} and \cite{Feehan_maximumprinciple}, we require the constant, $\gamma$, in \eqref{eq:HestonWeight} to be positive for the purpose of proving existence and uniqueness, respectively, for solutions to \eqref{eq:HestonVariationalEquation} when $\sO$ is unbounded. However, while we shall continue to assume $\gamma>0$ in this article for consistency, this constant plays no role in regularity arguments or when $\sO$ is bounded and, for the latter purposes, one could set $\gamma=0$.
\end{rmk}

See Definitions \ref{defn:HkWeightedSobolevSpaceNormPowery} and \ref{defn:HkWeightedSobolevSpaceNormSingleWeight} for the technical descriptions of the weighted Sobolev spaces, $\sH^{k+2}(\sO,\fw)$ and $W^{k,p}(\sO,\fw)$, respectively.

We let $\NN:=\left\{0,1,2,3,\ldots\right\}$ denote the set of non-negative integers. For $r>0$ and $P_0=(x_0,y_0)\in\RR^2$, we let $B_r(P_0) := \{(x,y)\in\RR^2: (x-x_0)^2+(y-y_0)^2<r^2\}$ denote the open ball with center $P_0$ and radius $r$ and, given a domain $\sO\subset\RR^2$, we denote $B_r^+(P_0) := \sO\cap B_r(P_0)$, when the domain $\sO$ is understood from the context.

\begin{thm}[Interior $\sH^{k+2}$ regularity on half-balls]
\label{thm:HkSobolevRegularityInterior}
Let $R_0>R$ be positive constants and let $k\geq 0$ be an integer. Then there is a positive constant, $C=C(\Lambda,\nu_0,k,R,R_0)$, such that the following holds. Let $\sO\subseteqq\HH$ be a domain and let $z_0 \in \partial_0\sO$ be such that
$
\HH\cap B_{R_0}(z_0) \subset \sO.
$
Suppose that $f\in L^2(\sO,\fw)$ and that $u \in H^1(\sO,\fw)$ is a solution to the variational equation \eqref{eq:HestonVariationalEquation}.
If $f \in W^{k,2}(B_{R_0}^+(z_0),\fw)$, then
$
u \in \sH^{k+2}(B_R^+(z_0),\fw),
$
and $u$ solves \eqref{eq:IntroBoundaryValueProblem} on $B_R^+(z_0)$ and
\begin{equation}
\label{eq:HkSobolevRegularity}
\|u\|_{\sH^{k+2}(B_R^+(z_0),\fw)} \leq C\left(\|f\|_{W^{k,2}(B_{R_0}^+(z_0),\fw)} + \|u\|_{L^2(B_{R_0}^+(z_0),\fw)}\right).
\end{equation}
\end{thm}

We also have the following analogues of \cite[Theorem 8.10]{GilbargTrudinger}.

\begin{thm}[Interior $\sH^{k+2}$ regularity on domains]
\label{thm:HkSobolevRegularityDomain}
Let $k\geq 0$ be an integer and let $\sO\subseteqq\HH$ be a domain. Suppose that $f\in L^2(\sO,\fw)$ and $u \in H^1(\sO,\fw)$ is a solution to the variational equation \eqref{eq:HestonVariationalEquation}. If $f \in W^{k,2}_{\loc}(\underline\sO,\fw)$, then
$
u \in \sH^{k+2}_{\loc}(\underline\sO,\fw),
$
and $u$ solves \eqref{eq:IntroBoundaryValueProblem}. Moreover, for positive constants $d_1<\Upsilon$ and each pair of subdomains, $\sO'\subset\sO''\subset\sO$ with $\sO'\Subset\underline\sO''$ and $\dist(\partial_1\sO',\partial_1\sO'')\geq d_1$ and $\height(\sO'') \leq \Upsilon$, there is a positive constant, $C(\Lambda,\nu_0,d_1,k,\Upsilon)$,
such that $u$ obeys
\begin{equation}
\label{eq:HkSobolevRegularityDomain}
\|u\|_{\sH^{k+2}(\sO',\fw)} \leq C\left(\|f\|_{W^{k,2}(\sO'',\fw)} + \|u\|_{L^2(\sO'',\fw)}\right).
\end{equation}
\end{thm}

\begin{rmk}[Regularity up to the ``non-degenerate boundary'']
Standard elliptic regularity results for linear, second-order, strictly elliptic operators \cite[Theorem 8.13]{GilbargTrudinger} also imply that $u \in W^{k+2,2}_{\loc}(\sO\cup\partial_1\sO)$ if $f \in W^{k,2}_{\loc}(\sO\cup\partial_1\sO)$, and $g \in W^{k+2,2}_{\loc}(\sO\cup\partial_1\sO)\cap H^1(\sO,\fw)$, and $\partial_1\sO$ is $C^{k+2}$, and $u-g \in H^1_0(\underline\sO,\fw)$. We shall omit further mention of similarly straightforward generalizations.
\end{rmk}

Finally, we have an analogue of \cite[Theorem 8.9]{GilbargTrudinger}.

\begin{thm}[Existence and uniqueness of solutions with interior $\sH^{k+2}$ regularity]
\label{thm:ExistUniqueHkSobolevRegularityDomain}
Let $k\geq 0$ be an integer and let $\sO\subseteqq\HH$. Suppose that\footnote{The hypotheses on $f$ and $g$ are relaxed in \cite{Daskalopoulos_Feehan_statvarineqheston, Feehan_maximumprinciple} to allow for unbounded $f$ and $g$ with suitable growth properties.} $f\in L^\infty(\sO)\cap W^{k,2}_{\loc}(\underline\sO,\fw)$, and $(1+y)g\in W^{2,\infty}(\sO)$, and the constant, $c_0$, in \eqref{eq:OperatorHestonIntro} obeys $c_0>0$. Then there exists a unique solution, $u \in H^1(\sO,\fw)\cap \sH^{k+2}_{\loc}(\underline{\sO},\fw)$, to the variational equation \eqref{eq:HestonVariationalEquation} with boundary condition $u-g\in H^1_0(\underline{\sO},\fw)$. Moreover, $u$ solves \eqref{eq:IntroBoundaryValueProblem}, \eqref{eq:IntroBoundaryValueProblemBC} and, for positive constants $d_1<\Upsilon$ and each pair of subdomains, $\sO'\subset\sO''\subset\sO$ with $\sO'\Subset\underline\sO''$ and $\dist(\partial_1\sO',\partial_1\sO'')\geq d_1$ and $\height(\sO'') \leq \Upsilon$, there is a positive constant,
$C=C(\Lambda,\nu_0,d_1,k,\Upsilon)$, such that the estimate \eqref{eq:HkSobolevRegularityDomain} holds.
\end{thm}

\subsubsection{Higher-order interior H\"older regularity}
\label{subsubsec:HigherOrderHolderRegularity}
See Definitions \ref{defn:Calphas} and \ref{defn:DHspaces} for descriptions of the Daskalopoulos-Hamilton family of $C^{k,\alpha}_s$ H\"older norms. We have the following analogue of Theorem \ref{thm:HkSobolevRegularityInterior}.

\begin{thm} [Interior $C^{k,\alpha}_s$ regularity for a solution to the variational equation]
\label{thm:HolderContinuity_Dku_Interior}
Let $k\geq 0$ be an integer, let $p>\max\{4, 2+k+\beta\}$, and let $R_0$ be a positive constant. Then there are positive constants
$R_1=R_1(k,R_0) < R_0$, and $C=C(\Lambda,\nu_0,k,p,R_0)$, and $\alpha=\alpha(\Lambda,\nu_0,k,p,R_0)\in(0,1)$ such that the following holds. Let $\sO\subseteqq\HH$ be a domain. If $f \in L^2(\sO,\fw)$, and $u \in H^1(\sO,\fw)$ is a solution to the variational equation \eqref{eq:HestonVariationalEquation}, and $z_0 \in \partial_0\sO$ is such that
$
\HH\cap B_{R_0}(z_0) \subset \sO,
$
and $f \in W^{2k,p}(B_{R_0}^+(z_0),y^{\beta-1})$, then
$
u \in C^{k,\alpha}_s(\bar B_{R_1}^+(z_0)),
$
and $u$ solves \eqref{eq:IntroBoundaryValueProblem} on $B_{R_1}^+(z_0)$. Moreover, $u$ obeys
\begin{equation}
\label{eq:HolderSobolevContinuity_Dku}
\|u\|_{C^{k,\alpha}_s(\bar B_{R_1}^+(z_0))} \leq C\left(\|f\|_{W^{2k,p}(B^+_{R_0}(z_0),y^{\beta-1})}
+ \|u\|_{L^2(B^+_{R_0}(z_0),y^{\beta-1})}\right).
\end{equation}
\end{thm}

Given Theorem \ref{thm:HolderContinuity_Dku_Interior}, one easily obtains --- but via purely Sobolev space and Moser iteration methods --- the following degenerate-elliptic analogue of the $C^\infty$-regularity result for the degenerate-parabolic model for the linearization of the porous medium equation \cite[Theorem I.1.1]{DaskalHamilton1998}.

\begin{cor} [Interior $C^\infty$-regularity]
\label{cor:CinftyGlobal}
Let $\sO\subseteqq\HH$ be a domain. If $f \in L^2(\sO,\fw)$ and $u \in H^1(\sO,\fw)$ is a solution to the variational equation \eqref{eq:HestonVariationalEquation}, and $f \in C^\infty(\underline\sO)$, then $u \in C^\infty(\underline\sO)$ and $u$ solves \eqref{eq:IntroBoundaryValueProblem}, \eqref{eq:IntroBoundaryValueProblemBC}.
\end{cor}

We also have an analogue of Theorem \ref{thm:HkSobolevRegularityDomain} and of \cite[Theorem 6.17]{GilbargTrudinger}.

\begin{thm} [Interior $C^{k,\alpha}_s$ regularity on domains]
\label{thm:CkalphasHolderContinuityDomain}
Let $k\geq 0$ be an integer and let $p> \max\{4, 3+k+\beta\}$. Then there is a positive constant $\alpha=\alpha(\Lambda,\nu_0,k,p)\in(0,1)$ such that the following holds. Let $\sO\subseteqq\HH$ be a domain. If $f \in L^2(\sO,\fw)$ and $u \in H^1(\sO,\fw)$ is a solution to the variational equation \eqref{eq:HestonVariationalEquation}, and $f \in W^{2k+2,p}_{\loc}(\underline\sO,\fw)$, then
$
u \in C^{k,\alpha}_s(\underline\sO).
$
Moreover, $u$ solves \eqref{eq:IntroBoundaryValueProblem} and, for positive constants $d_1<\Upsilon$ and each pair of subdomains, $\sO'\subset\sO''\subset\sO$ with $\sO'\Subset\underline\sO''$ and $\dist(\partial_1\sO',\partial_1\sO'')\geq d_1$ and \footnote{While the equation \eqref{eq:IntroBoundaryValueProblem} is translation-invariant in the $x$-direction, the estimate \eqref{eq:CkalphasHolderContinuityDomain} is not when $\gamma\neq 0$ in the definition \eqref{eq:HestonWeight} of the weight, $\fw$.}
$\sO''\subset(-\Upsilon,\Upsilon)\times(0,\Upsilon)$, there is a positive constant, $C=C(\Lambda,\nu_0,d_1,k,p,\Upsilon)$,
such that
\begin{equation}
\label{eq:CkalphasHolderContinuityDomain}
\|u\|_{C^{k,\alpha}_s(\bar\sO')} \leq C\left(\|f\|_{W^{2k+2,p}(\sO'',\fw)} + \|u\|_{L^2(\sO'',\fw)}\right).
\end{equation}
\end{thm}

\begin{cor} [Interior a priori $C^{k,\alpha}_s$ estimate on domains of finite height]
\label{cor:CkalphasHolderContinuityDomainStrip}
If in addition to the hypotheses of Theorem \ref{thm:CkalphasHolderContinuityDomain}, the hypothesis on $f$ is strengthened to $f \in W^{2k+2,p}(\sO, y^{\beta-1})$, then for positive constants $d_1<\Upsilon$ and each pair of subdomains, $\sO'\subset\sO''\subset\sO$ with $\sO'\Subset\underline\sO''$ and $\dist(\partial_1\sO',\partial_1\sO'')\geq d_1$ and $\height(\sO'')\leq\Upsilon$, there is a positive constant,
$C=C(\Lambda,\nu_0,d_1,k,p,\Upsilon)$,
such that
\begin{equation}
\label{eq:CkalphasHolderContinuityDomainStrip}
\|u\|_{C^{k,\alpha}_s(\bar\sO')} \leq C\left(\|f\|_{W^{2k+2,p}(\sO'', y^{\beta-1})} + \|u\|_{L^2(\sO'', y^{\beta-1})}\right).
\end{equation}
\end{cor}

Lastly, we give an analogue of the existence and interior Schauder a priori estimate results \cite[Theorems I.1.1, I.1.2, and I.12.2]{DaskalHamilton1998} for the initial value problem for a degenerate-parabolic model for the linearization of the porous medium equation on a half-space, and of \cite[Theorems 6.13 and 6.19]{GilbargTrudinger}, in the case of boundary value problems for strictly elliptic operators. We recall the description of the weighted H\"older space, $C^{k,2+\alpha}_s(\underline\sO)$, due to Daskalopoulos and Hamilton \cite{DaskalHamilton1998}, in Definition \ref{defn:DH2spaces}. While the interior a priori estimate \eqref{eq:ExistUniqueCk2+alphasHolderContinuityDomain} is stated in Theorem \ref{thm:ExistUniqueCk2+alphasHolderContinuityDomain} for bounded subdomains, $\sO''\subset\HH$, for the sake of clarity, the estimate should easily extend to unbounded domains using the family of H\"older spaces and H\"older norms defined in \cite{Feehan_Pop_mimickingdegen_pde}.

\begin{thm} [Existence and uniqueness of solutions with interior $C^{k,2+\alpha}_s$ regularity]
\label{thm:ExistUniqueCk2+alphasHolderContinuityDomain}
Let $k\geq 0$ be an integer and let $K$ be a finite right-circular cone. Then there is a positive constant $\alpha=\alpha(\Lambda,\nu_0,k,K)\in(0,1)$ such that the following holds. Let $\sO\subseteqq\HH$ be a domain obeying a uniform exterior cone condition along $\partial_1\sO$ with cone $K$. If $f \in C^{2k+6,\alpha}_s(\underline\sO)\cap C(\bar \sO)$ and $g \in C(\bar\sO)$ with $(1+y)g \in  C^2(\bar\sO)$,
and the constant, $c_0$, in \eqref{eq:OperatorHestonIntro} obeys $c_0>0$, then there is a unique solution,
$
u \in C^{k,2+\alpha}_s(\underline\sO)\cap C^\alpha(\sO\cup\partial_1\sO)\cap L^\infty(\sO),
$
to the boundary value problem, \eqref{eq:IntroBoundaryValueProblem}, \eqref{eq:IntroBoundaryValueProblemBC}. Moreover, for positive constants $d_1<\Upsilon$ and each pair of subdomains, $\sO'\subset\sO''\subset\sO$ with $\sO'\Subset\underline\sO''$ and $\dist(\partial_1\sO',\partial_1\sO'')\geq d_1$ and $\diam(\sO'')\leq\Upsilon$, there is a positive constant, $C=C(\Lambda,\nu_0,d_1,k,p,\Upsilon)$,
such that
\begin{equation}
\label{eq:ExistUniqueCk2+alphasHolderContinuityDomain}
\|u\|_{C^{k,2+\alpha}_s(\bar\sO')} \leq C\left(\|f\|_{C^{2k+6,\alpha}_s(\bar\sO'')} + \|u\|_{C(\bar\sO'')}\right).
\end{equation}
\end{thm}

\begin{rmk}[Schauder a priori estimates and approach to existence of solutions]
As we explain in \cite{Feehan_Pop_elliptichestonschauder}, the proof of existence of solutions, $u \in C^{k,2+\alpha}_s(\underline\sO)\cap C(\bar\sO)$, to the boundary value problem, \eqref{eq:IntroBoundaryValueProblem}, \eqref{eq:IntroBoundaryValueProblemBC}, given $f\in C^{k,\alpha}_s(\underline\sO)$ and $g \in C(\bar\sO)$, is considerably more difficult when $\partial_0\sO$ is non-empty because, unlike in \cite{DaskalHamilton1998}, one must consider the impact of the ``corner'' points, $\overline{\partial_0\sO}\cap\overline{\partial_1\sO}$, of the subdomain, $\sO\subset\HH$, where the ``non-degenerate boundary'', $\partial_1\sO$, intersects the ``degenerate boundary'', $\partial\HH$.
Note that when $d=2$ and $\sO$ is a bounded domain
such that $\sO$ obeys uniform exterior and interior cone conditions along $\overline{\partial_0\sO}\cap\overline{\partial_1\sO}$, in the sense of
Definition \ref{defn:RegularDomain} in the sequel,
then the set
$\overline{\partial_0\sO}\cap\overline{\partial_1\sO}$
consists of finitely many points.
\end{rmk}

\begin{rmk}[Refinements of Theorem \ref{thm:ExistUniqueCk2+alphasHolderContinuityDomain}]
Our existence result and Schauder a priori estimate in Theorem \ref{thm:ExistUniqueCk2+alphasHolderContinuityDomain} may appear far from optimal because of the strong hypothesis that $f\in C^{2k+6,\alpha}_s(\underline\sO)$, the fact that the H\"older exponent, $\alpha\in(0,1)$, is not arbitrary, and the presence of the cone condition hypothesis. However, the regularity hypothesis for $f$ in Theorem \ref{thm:ExistUniqueCk2+alphasHolderContinuityDomain} may be relaxed to $f\in C^{k,\alpha}_s(\underline\sO)$, with $\alpha\in(0,1)$ arbitrary, and the cone conditions on $\sO$ removed, using an interior Schauder a priori estimate which we develop by quite different methods in \cite{Feehan_Pop_elliptichestonschauder}.
\end{rmk}

\begin{rmk}[Campanato spaces]
In the context of non-degenerate elliptic equations, Campanato spaces \cite{Troianiello} provide a natural bridge between Sobolev spaces and H\"older spaces and allow one to prove Schauder a priori estimates and H\"older regularity using Sobolev space methods. It would be interesting to explore whether the conclusions of Theorems \ref{thm:HolderContinuity_Dku_Interior} and \ref{thm:CkalphasHolderContinuityDomain}, and thus Theorem \ref{thm:ExistUniqueCk2+alphasHolderContinuityDomain} in particular, could be sharpened with the aid of a suitable version of Campanato spaces adapted to the weights appearing in our definitions of weighted Sobolev and H\"older spaces.
\end{rmk}

Given an additional geometric hypothesis on $\sO$ near points in $\overline{\partial_0\sO}\cap\overline{\partial_1\sO}$, the property $u \in C^\alpha(\sO\cup\partial_1\sO)\cap L^\infty(\sO)$ simplifies to $u \in C^\alpha_{s,\loc}(\bar\sO)\cap C(\bar\sO)$.

We say that a domain, $U\subset\HH$, obeys an \emph{exterior cone condition relative to $\HH$ at a point $z_0\in \partial U$} if there exists a finite, right circular cone, $K = K_{z_0}\subset \bar\HH$, with vertex $z_0$ such that $\bar U\cap K_{z_0} = \{z_0\}$ (compare \cite[p. 203]{GilbargTrudinger}).  A domain, $U$, obeys a \emph{uniform exterior cone condition relative to $\HH$ on $T\subset\partial U$} if $U$ satisfies an exterior cone condition relative to $\HH$ at every point $z_0\in T$ and the cones $K_{z_0}$ are all congruent to some fixed finite cone, $K$ (compare with \cite[p. 205]{GilbargTrudinger}).

\begin{defn}[Interior and exterior cone conditions]
\label{defn:RegularDomain}
Let $K$ be a finite, right circular cone. We say that $\sO$ obeys \emph{interior and exterior cone conditions at $z_0\in \overline{\partial_0\sO}\cap\overline{\partial_1\sO}$ with cone $K$} if the domains $\sO$ and $\HH\less\bar\sO$ obey exterior cone conditions
relative to $\HH$ at $z_0$ with cones congruent to $K$. We say that $\sO$ obeys \emph{uniform interior and exterior cone conditions on $\overline{\partial_0\sO}\cap\overline{\partial_1\sO}$ with cone $K$} if the domains $\sO$ and $\HH\less\bar\sO$ obey exterior cone
conditions relative to $\HH$ at each point $z_0\in \overline{\partial_0\sO}\cap\overline{\partial_1\sO}$ with cones congruent to $K$.
\end{defn}

\begin{cor} [Existence and uniqueness of globally continuous $C^{k,2+\alpha}_s(\underline\sO)$ solutions]
\label{cor:ExistUniqueCk2+alphasHolderContinuityDomain}
Suppose, in addition to the hypotheses of Theorem \ref{thm:ExistUniqueCk2+alphasHolderContinuityDomain}, that the domain, $\sO$, satisfies a uniform exterior and interior cone condition on $\overline{\partial_0\sO}\cap\overline{\partial_1\sO}$ with cone $K$ in the sense of
Definition \ref{defn:RegularDomain}.
Then the solution, $u$, obeys
$$
u \in C^{k,2+\alpha}_s(\underline\sO)\cap C^\alpha_{s,\loc}(\bar\sO)\cap C(\bar\sO),
$$
and, if $\sO$ is bounded, then $u \in C^{k,2+\alpha}_s(\underline\sO)\cap C^\alpha_s(\bar\sO)$.
\end{cor}

In a different direction, given additional hypotheses on $f$, we easily obtain

\begin{cor} [Interior a priori $C^{k,2+\alpha}_s$ estimate on domains of finite height]
\label{cor:Ck2+alphasHolderContinuityDomainStrip}
If in addition to the hypotheses of Theorem \ref{thm:ExistUniqueCk2+alphasHolderContinuityDomain}, the hypothesis on $f$ is strengthened to $f \in C^{2k+6,\alpha}_s(\bar\sO)$ then, for positive constants $d_1<\Upsilon$ and each pair of subdomains, $\sO'\subset\sO''\subset\sO$ with $\sO'\Subset\underline\sO''$ and $\dist(\partial_1\sO',\partial_1\sO'')\geq d_1$ and $\height(\sO'')\leq\Upsilon$, there is a positive constant, $C=C(\Lambda,\nu_0,d_1,k,p,\Upsilon)$,
such that
\begin{equation}
\label{eq:Ck2+alphasHolderContinuityDomainStrip}
\|u\|_{C^{k,2+\alpha}_s(\bar\sO')} \leq C\left(\|f\|_{C^{2k+6,\alpha}_s(\bar\sO'')} + \|u\|_{C(\bar\sO'')}\right).
\end{equation}
\end{cor}

\subsection{Connections with previous research}
\label{subsec:Survey}

The literature on degenerate elliptic and parabolic equations is vast, with the well-known articles of Fabes, Kenig, and Serapioni \cite{Fabes_1982, Fabes_Kenig_Serapioni_1982a}, Fichera \cite{Fichera_1956, Fichera_1960}, Kohn and Nirenberg \cite{Kohn_Nirenberg_1967}, Murthy and Stampacchia \cite{Murthy_Stampacchia_1968} and the monographs of Levendorski{\u\i} \cite{LevendorskiDegenElliptic} and Ole{\u\i}nik and Radkevi{\v{c}} \cite{Oleinik_Radkevic, Radkevich_2009a, Radkevich_2009b}, being merely the tip of the iceberg. However, there has been far less research on higher-order regularity of solutions up to the portion of the domain boundary where the operator becomes degenerate. In this context, we recall the work of Daskalopoulos and her collaborators \cite{DaskalHamilton1998, Daskalopoulos_Rhee_2003}, Koch \cite{Koch}, Epstein and Mazzeo \cite{Epstein_Mazzeo_2010, Epstein_Mazzeo_annmathstudies, Epstein_Mazzeo_cont_est, Epstein_Mazzeo_cont_est_diag} and the second author \cite{Pop_2013b}.

Our work is most closely related to that of Koch \cite{Koch}. However, while Koch uses Sobolev weights which are comparable to ours, his methods --- which use Moser iterations and pointwise estimates for fundamental solutions --- are very different from those we employ in \cite{Feehan_Pop_regularityweaksoln}, which use Moser iterations and the abstract John-Nirenberg inequality, and so we avoid
any need to consider pointwise estimates for the fundamental solution for the Heston operator \eqref{eq:OperatorHestonIntro}. Moreover, the structure of the lower-order terms in the linear operators is simpler in \cite{Koch}, whereas the new $u_x$ term present in \eqref{eq:OperatorHestonIntro} causes considerable difficulty. Finally, Koch does not consider the case where $\partial\sO = \partial_0\sO\cup\overline{\partial_1\sO}$, where $A$ is degenerate along $\partial_0\sO$ but non-degenerate along $\partial_1\sO$.

\subsection{Some mathematical highlights of this article}
\label{subsec:Highlights}
Our approach in \S \ref{sec:H2Regularity} and \S \ref{sec:SobolevRegularity} to the higher-order Sobolev regularity theory for weak solutions to equation \eqref{eq:IntroBoundaryValueProblem} may appear to proceed by adapting a traditional strategy \cite[\S 8.3 and \S 8.4]{GilbargTrudinger}, but the degeneracy of the operator, $A$, in \eqref{eq:OperatorHestonIntro} makes this strategy far more complicated than one might expect.

The main difficulty is to improve the $L^2(\sO,\fw)$-estimate for $y^{1/2}Du$ implicit in the a priori $H^1(\sO,\fw)$-estimate for a solution to \eqref{eq:HestonVariationalEquation} to an $L^2(\sO,\fw)$-estimate for the gradient, $Du$. For this purpose, we use a trick due to Koch \cite[Lemma 4.6.1]{Koch} which yields an interior $L^2(\sO',\fw)$-estimate for $Du$ on a subdomain $\sO'\subset\sO\subsetneqq\HH$ with $\bar\sO'\subset\underline\sO$ and $\dist(\partial_1\sO', \partial_1\sO)>0$. It is unclear how to justify a global $L^2(\sO,\fw)$-estimate for $Du$ on a subdomain $\sO\subsetneqq\HH$ without a priori knowledge of the existence of smooth solutions, $u \in C^\infty(\bar\sO)$, to the boundary value problem \eqref{eq:IntroBoundaryValueProblem}, \eqref{eq:IntroBoundaryValueProblemBC}.

The essential idea in the development of higher-order Sobolev regularity is to take derivatives of the equation \eqref{eq:IntroBoundaryValueProblem} and estimate $k+2$ derivatives of $u$ in terms of $k$ derivatives of $f$. Such an approach is complicated by the presence of the degeneracy factor, $y$, because differentiating \eqref{eq:IntroBoundaryValueProblem} once with respect to $y$ yields unweighted second-order derivative terms, $u_{xx}, u_{xy}, u_{yy}$, \emph{without} the degeneracy factor, $y$, and these are even harder to estimate, precisely because the operator, $A$, in \eqref{eq:OperatorHestonIntro} is degenerate elliptic. It is this feature which partly accounts for the complexity of our Definition \ref{defn:HkWeightedSobolevSpaceNormPowery} of the higher-order weighted Sobolev spaces, $\sH^{k+2}(\sO, \fw)$.

Naturally, the same difficulty arises in \S \ref{sec:HolderRegularity} when we consider higher-order H\"older regularity, $C^{k,\alpha}_s(\underline\sO)$, with $k\geq 1$, and $C^{k,2+\alpha}_s(\underline\sO)$, of a solution $u \in H^1(\sO, \fw)$. However, at this stage, the difficulties have largely been overcome in \S \ref{sec:SobolevRegularity}.
The challenges to prove existence of solutions, $u$ --- whether in $C^{k,2+\alpha}_s(\underline\sO)\cap C(\bar\sO)$ or $C^{k,2+\alpha}_s(\bar\sO)$ --- to the elliptic boundary value problem \eqref{eq:IntroBoundaryValueProblem}, \eqref{eq:IntroBoundaryValueProblemBC} entirely within a Schauder framework is explained in more detail in \cite{Feehan_Pop_elliptichestonschauder}.

\subsection{Extensions}
\label{subsec:Extensions}
The Heston stochastic volatility process and its associated generator serve as paradigms for degenerate Markov processes and their degenerate-elliptic generators which appear widely in mathematical finance, so we briefly comment on two directions for extending our work in this article.

First, the two-dimensional Heston process has natural $d$-dimensional analogues \cite{DaFonseca_Grasselli_Tebaldi_2008} defined, for example, by coupling non-degenerate $(d-1)$-diffusion processes with degenerate one-dimensional processes
\cite{Cherny_Engelbert_2005, Mandl, Zettl}. We expect that all of the main results of this article should extend to the case of a degenerate-elliptic operator on a subdomain $\sO$ of a half-space $\HH=\RR^{d-1} \times \RR_+$,
\begin{equation}
\label{eq:OperatorHestonIntroHigherDimension}
Av :=  -x_d a^{ij} v_{x_ix_j} - b^iv_{x_i} + cv, \quad v\in C^\infty(\sO),
\end{equation}
under the assumptions that the matrix $(a^{ij})$ is strictly elliptic,  $b^d \geq \nu >0$,  for some constant $\nu >0$, and $c \geq 0$ and the coefficients $(a^{ij})$, $(b^i)$, and $c$ have suitable growth and regularity properties. See \cite{Feehan_Pop_mimickingdegen_pde} for an analysis with applications to probability theory based on a parabolic version of this type of elliptic operator as well as \cite{Feehan_maximumprinciple} for weak maximum principles for a general class of degenerate-elliptic operators.

Second, the important question of higher-order regularity of solutions up to the corner points $\overline{\partial_0\sO}\cap\overline{\partial_1\sO}$ is deferred to a subsequent article.

\subsection{Outline of the article}
\label{subsec:Guide}
For the convenience of the reader, we provide a brief outline of the article. In \S \ref{sec:ReviewHolderContinuity},
we review the variational inequality associated to the Heston operator \eqref{eq:OperatorHestonIntro} and we give the definition of the H\"older spaces used in our article.
In \S \ref{sec:H2Regularity}, we establish the $H^2(\underline\sO,\fw)$-regularity for solutions, $u\in H^1(\sO,\fw)$, concluding with Theorem \ref{thm:H2BoundSolutionHestonVarEqnSubdomainInterior}. In \S \ref{sec:SobolevRegularity}, we establish the $\sH^{k+2}(\underline\sO,\fw)$-regularity for solutions, $u\in H^1(\sO,\fw)$, for all $k\geq 0$, with proofs of Theorems \ref{thm:HkSobolevRegularityInterior} and \ref{thm:HkSobolevRegularityDomain}, together with Theorem \ref{thm:ExistUniqueHkSobolevRegularityDomain} and Corollary \ref{cor:CinftyGlobal}. Section \ref{sec:HolderRegularity} contains our proofs of $C^{k,\alpha}_s(\underline\sO)$-
regularity of solutions, $u\in H^1(\sO,\fw)$, in the form of Theorems \ref{thm:HolderContinuity_Dku_Interior} and \ref{thm:CkalphasHolderContinuityDomain}, together with proofs of $C^{k,2+\alpha}_s(\underline\sO)$-regularity and a Schauder a priori estimate, as part of Theorem \ref{thm:ExistUniqueCk2+alphasHolderContinuityDomain}, and Corollary \ref{cor:ExistUniqueCk2+alphasHolderContinuityDomain}.
Appendix \ref{sec:Appendix} collects some technical definitions and observations used in the proofs.

\subsection{Notation and conventions}
\label{subsec:Notation}
In the definition and naming of function spaces, including spaces of continuous functions, H\"older spaces, or Sobolev spaces, we follow Adams \cite{Adams_1975} and alert the reader to occasional differences in definitions between \cite{Adams_1975} and standard references such as Gilbarg and Trudinger \cite{GilbargTrudinger} or Krylov \cite{Krylov_LecturesHolder, Krylov_LecturesSobolev}.

If $V\subset U\subset \RR^d$ are open subsets, we write $V\Subset U$ when $U$ is bounded with closure $\bar U \subset V$. By $\supp\zeta$, for any $\zeta\in C(\RR^2)$, we mean the \emph{closure} in $\RR^2$ of the set of points where $\zeta\neq 0$.

We use $C=C(*,\ldots,*)$ to denote a constant which depends at most on the quantities appearing on the parentheses. In a given context, a constant denoted by $C$ may have different values depending on the same set of arguments and may increase from one inequality to the next.

\section{Preliminaries}
\label{sec:ReviewHolderContinuity}
In \S \ref{subsec:HolderRegularityPreliminaries}, we review our definition of the variational equation \eqref{eq:HestonVariationalEquation} corresponding to \eqref{eq:IntroBoundaryValueProblemBC}, in our framework of weighted Sobolev spaces. In \S \ref{subsec:LocalSupremumBounds}, we recall our local supremum estimates for solutions, $u\in H^1(\sO,\fw)$, to the variational equation \eqref{eq:HestonVariationalEquation} which we proved in \cite{Feehan_Pop_regularityweaksoln}, while in \S \ref{subsec:HolderContinuity}, we review our H\"older continuity results for solutions, $u\in H^1(\sO,\fw)$, which we also proved in \cite{Feehan_Pop_regularityweaksoln}. Finally, in \S \ref{subsec:DHspaces}, we give the definitions of higher-order weighted H\"older spaces due to Daskalopoulos
and Hamilton \cite{DaskalHamilton1998}.

\subsection{Variational equality}
\label{subsec:HolderRegularityPreliminaries}
We recall that \cite[Definition 2.22]{Daskalopoulos_Feehan_statvarineqheston}
\begin{equation}
\label{eq:HestonWithKillingBilinearForm}
\begin{aligned}
\fa(u,v) &:= \frac{1}{2}\int_\sO\left(u_xv_x + \varrho\sigma u_yv_x
+ \varrho\sigma u_xv_y + \sigma^2u_yv_y\right)y\,\fw\,dx\,dy
\\
&\quad - \frac{\gamma}{2}\int_\sO\left(u_x + \varrho\sigma u_y\right)v \,\frac{x}{\sqrt{1+x^2}}y\,\fw\,dx\,dy
\\
&\quad - \int_\sO(a_1y + b_1)u_xv\,\fw\,dx\,dy + \int_\sO c_0uv\,\fw\,dx\,dy, \quad\forall u, v \in H^1(\sO,\fw),
\end{aligned}
\end{equation}
is the \emph{bilinear form associated with the Heston operator}, $A$, in \eqref{eq:OperatorHestonIntro}, where
\begin{equation}
\label{eq:DefinitionA1B1}
a_1 := \frac{\kappa\varrho}{\sigma}-\frac{1}{2} \quad\hbox{and}\quad b_1 := c_0-q-\frac{\kappa\theta\varrho}{\sigma}.
\end{equation}
We shall also avail of the

\begin{assump}[Condition on the coefficients of the Heston operator]
\label{assump:HestonCoefficientb1}
The coefficients defining $A$ in \eqref{eq:OperatorHestonIntro} have the property that $b_1=0$ in \eqref{eq:DefinitionA1B1}.
\end{assump}

Assumption \ref{assump:HestonCoefficientb1} involves no loss of generality because, using a simple affine changes of variables on $\RR^2$ which map $(\HH,\partial\HH)$ onto $(\HH,\partial\HH)$ (see \cite[Lemma 2.2]{Daskalopoulos_Feehan_statvarineqheston}), we can arrange that $b_1 = 0$.

We recall the definition of the weighted Sobolev space $H^1(\sO, \fw)$ in \eqref{eq:H1WeightedSobolevSpace} and \eqref{eq:H1NormHeston}. Given a subset $T\subset\partial\sO$, we let $H^1_0(\sO\cup T,\fw)$ be the closure\footnote{Note that $H^1_0(\sO\cup \bar T,\fw)=H^1_0(\sO\cup T,\fw)=H^1_0(\sO\cup \mathring{T},\fw)$, since  $C^\infty_0(\sO\cup \bar T) =  C^\infty_0(\sO\cup T) =  C^\infty_0(\sO\cup \mathring{T})$, where $\mathring{T}$ and $\bar T$ denote the interior and closure, respectively, of $T$ in $\partial\sO$.} in $H^1(\sO,\fw)$ of $C^\infty_0(\sO\cup T)$. Given a source function $f\in L^2(\sO,\fw)$ and recalling that $\underline{\sO}=\sO\cup\partial_0\sO$, we call a function $u\in H^1(\sO,\fw)$ a \emph{solution to the variational equation} for the Heston operator if
\begin{equation}
\label{eq:HestonVariationalEquation}
\fa(u,v) = (f,v)_{L^2(\sO,\fw)}, \quad \forall v \in H^1_0(\underline{\sO},\fw).
\end{equation}
Given a subset $T\subset\partial\sO$ and $g\in H^1(\sO,\fw)$, we say that $u\in H^1(\sO,\fw)$ obeys \emph{$u=g$ on $T\subset\partial\sO$ in the sense of $H^1$} if
\begin{equation}
\label{eq:HestonVariationalEquation_BC}
u-g \in H^1_0(\sO\cup T^c,\fw),
\end{equation}
where $T^c := \partial\sO\less T$. In our application, we shall only consider $T\subseteqq\partial_1\sO$.

Recall from \cite[Definition 2.20]{Daskalopoulos_Feehan_statvarineqheston} that
\begin{equation}
\label{eq:H2WeightedSobolevSpace}
H^2(\sO,\fw) := \{u \in W^{2,2}_{\loc}(\sO): \|u\|_{H^2(\sO,\fw)} < \infty\},
\end{equation}
where
\begin{equation}
\label{eq:H2NormHeston}
\|u\|_{H^2(\sO,\fw)}^2 := \int_\sO\left( y^2|D^2u|^2 + (1+y)^2|Du|^2 + (1+y)u^2 \right)\,\fw\,dx\,dy.
\end{equation}
If $u \in H^2(\sO,\fw)$ and $g \in H^1(\sO,\fw)$, we recall from \cite[Lemma 2.29]{Daskalopoulos_Feehan_statvarineqheston}
that $u$ is a solution to \eqref{eq:IntroBoundaryValueProblem} (a.e. on $\sO$) and \eqref{eq:IntroBoundaryValueProblemBC} (in the sense of $H^1$) if and only if $u-g\in H^1_0(\underline{\sO},\fw)$ and $u$ is a solution to the variational equation \eqref{eq:HestonVariationalEquation}.


\subsection{Local supremum bounds near the degenerate boundary}
\label{subsec:LocalSupremumBounds}
If $S \subset \bar{\HH}$ is a Borel measurable subset, we let $|S|_{\beta}$ denote the volume of $S$ with respect to the measure $y^\beta \,dx\,dy$, and $|S|_{\fw}$ denote the volume of $S$ with respect to the measure $\fw \,dx\,dy$.
We recall from \cite{Feehan_Pop_regularityweaksoln} the following analogues of \cite[Proposition 4.5.1]{Koch} and \cite[Theorem 8.15]{GilbargTrudinger}.

\begin{thm}[Supremum estimates near points in $\partial_0\sO$]
\label{thm:MainSupremumEstimatesInterior}
\cite[Theorem 1.5]{Feehan_Pop_regularityweaksoln}
Let $p>2+\beta$ and let $R_0$ be a positive constant. Then there are positive constants, $C=C(\Lambda,\nu_0,R_0,p)$ and $R_1=R_1(R_0)<R_0$, such that the following holds. Let $\sO\subseteqq\HH$ be a domain. If $u \in H^1(\sO,\fw)$ satisfies the variational equation \eqref{eq:HestonVariationalEquation} with source function $f \in L^2(\sO,\fw)$, and $z_0 \in \partial_0\sO$ is such that
$$
\HH\cap B_{R_0}(z_0) \subset \sO,
$$
and $f$ obeys
\begin{equation}
\label{eq:Lsfcondition_ball}
f \in L^p(B_{R_0}^+(z_0),y^{\beta-1}),
\end{equation}
then $u \in L^\infty(B_{R_1}^+(z_0))$, and
\begin{equation}
\label{eq:MainSupremumEstimates2}
\|u\|_{L^\infty(B_{R_1}^+(z_0))}
\leq  C\left( \|f\|_{L^p(B_{R_0}^+(z_0),y^{\beta-1})} + \|u\|_{L^2(B_{R_0}^+(z_0),y^{\beta-1})} \right).
\end{equation}
\end{thm}

\begin{thm}[Supremum estimates near points in $\overline{\partial_0\sO}\cap\overline{\partial_1\sO}$]
\label{thm:MainSupremumEstimatesBoundary}
\cite[Theorem 1.6]{Feehan_Pop_regularityweaksoln}
Let $K$ be a finite right circular cone, let $p>2+\beta$, and let $R_0>0$ be a positive constant. Then there are positive constants, $C=C(K,\Lambda,\nu_0,R_0,p)$ and $R_1=R_1(K,\Lambda,\nu_0, R_0)<R_0$ such that the following holds. Let $\sO\subsetneqq\HH$ be a domain. If $u \in H^1(\sO,\fw)$ satisfies the variational equation \eqref{eq:HestonVariationalEquation} with source function $f \in L^2(\sO,\fw)$ and $z_0 \in \overline{\partial_0\sO}\cap\overline{\partial_1\sO}$ is such that $\sO$ obeys an interior cone condition at $z_0$ with cone $K$, and
$$
u = 0 \hbox{ on } \partial_1\sO\cap B_{R_0}(z_0) \quad\hbox{(in the sense of $H^1$)},
$$
and $f$ obeys \eqref{eq:Lsfcondition_ball}, then $u \in L^\infty(B_{R_1}^+(z_0))$ and $u$ satisfies \eqref{eq:MainSupremumEstimates2}.
\end{thm}

\subsection{H\"older continuity up to the degenerate boundary for solutions to the variational equation}
\label{subsec:HolderContinuity}
We recall the definition of the \emph{Koch distance function}, $s(\cdot,\cdot)$, on $\HH$ introduced by Koch in \cite[p.~11]{Koch},
\begin{equation}
\label{eq:IntroKochDistance}
\begin{aligned}
s(z,z_0) := \frac{|z-z_0|}{\sqrt{y + y_0 + |z-z_0|}}, \quad \forall z=(x,y), z_0=(x_0,y_0)\in \bar\HH,
\end{aligned}
\end{equation}
where $|z-z_0|^2 = (x-x_0)^2 + (y-y_0)^2$. The Koch distance function is equivalent to the \emph{cycloidal distance function} introduced by Daskalopoulos and Hamilton in \cite[p.~901]{DaskalHamilton1998} for the study of the porous medium equation.

Following \cite[\S 1.26]{Adams_1975}, for a domain $U\subset\HH$, we let $C(U)$ denote the vector space of continuous functions on $U$ and let $C(\bar U)$ denote the Banach space of functions in $C(U)$ which are bounded and uniformly continuous on $U$, and thus have unique bounded, continuous extensions to $\bar U$, with norm
$$
\|u\|_{C(\bar U)} := \sup_{U}|u|.
$$
Noting that $U$ may be \emph{unbounded}, we let $C_{\loc}(\bar U)$ denote the linear subspace of functions $u\in C(U)$ such that $u\in C(\bar V)$ for every precompact open subset $V\Subset \bar U$. Daskalopoulos and Hamilton provide the

\begin{defn}[$C^\alpha_s$ norm and Banach space]
\label{defn:Calphas}
\cite[p. 901]{DaskalHamilton1998}
Given $\alpha \in (0,1)$ and a domain $U\subset\HH$, we say that $u\in C^\alpha_s(\bar U)$ if $u\in C(\bar U)$ and
\begin{equation}
\label{eq:CalphasNorm}
\|u\|_{C^\alpha_s(\bar U)} := [u]_{C^\alpha_s(\bar U)} + \|u\|_{C(\bar U)} < \infty,
\end{equation}
where
\begin{equation}
\label{eq:CalphasSeminorm}
[u]_{C^\alpha_s(\bar U)} := \sup_{\stackrel{z_1,z_2\in U}{z_1\neq z_2}}\frac{|u(z_1)-u(z_2)|}{s(z_1,z_2)^\alpha}.
\end{equation}
We say that $u\in C^\alpha_s(\underline{U})$ if $u\in C^\alpha_s(\bar V)$ for all precompact open subsets $V\Subset \underline{U}$, recalling that $\underline{U} := U\cup\partial_0 U$. We let $C^\alpha_{s,\loc}(\bar U)$ denote the linear subspace of functions $u \in C^\alpha_s(U)$ such that $u\in C^\alpha_s(\bar V)$ for every precompact open subset $V\Subset \bar U$.
\end{defn}

It is known that $C^\alpha_s(\bar U)$ is a Banach space \cite[\S I.1]{DaskalHamilton1998} with respect to the norm \eqref{eq:CalphasNorm}. We recall the following analogues of \cite[Theorem 8.27 and 8.29]{GilbargTrudinger} and \cite[Theorem 4.5.5 and 4.5.6]{Koch}.

\begin{thm} [H\"older continuity near points in $\partial_0\sO$ for solutions to the variational equation]
\label{thm:MainContinuityInteriorL2RHSu}
\cite[Theorem 1.11]{Feehan_Pop_regularityweaksoln}
Let $p > \max\{4,2+\beta\}$ and let $R_0$ be a positive constant. Then there are positive constants, $R_1 = R_1(R_0) < R_0$, and $C=C(\Lambda,\nu_0,R_0,p)$, and $\alpha = \alpha(\Lambda,\nu_0,R_0,p) \in (0,1)$
such that the following holds. Let $\sO\subseteqq\HH$ be a domain. If $u \in H^1(\sO,\fw)$ satisfies the variational equation \eqref{eq:HestonVariationalEquation} with source function $f \in L^2(\sO,\fw)$ and $z_0 \in \partial_0\sO$ is such that
$$
\HH\cap B_{R_0}(z_0) \subset \sO,
$$
and $f$ obeys \eqref{eq:Lsfcondition_ball}, then $u \in C^\alpha_s(\bar B_{R_1}^+(z_0))$ and
\begin{equation}
\label{eq:MainContinuity3L2RHSu}
\|u\|_{C^\alpha_s(\bar B_{R_1}^+(z_0))} \leq C\left(\|f\|_{L^p(B_{R_0}^+(z_0),y^{\beta-1})} + \|u\|_{L^2(B_{R_0}^+(z_0))}\right).
\end{equation}
\end{thm}

\begin{thm} [H\"older continuity near points in $\overline{\partial_0\sO}\cap\overline{\partial_1\sO}$ for solutions to the variational equation]
\label{thm:MainContinuityBoundaryL2RHSu}
\cite[Theorem 1.13]{Feehan_Pop_regularityweaksoln}
Let $K$ be a finite, right circular cone, let $p > \max\{4,2+\beta\}$, and let $R_0$ be a positive constant. Then there are positive constants, $R_1 = R_1(K,\Lambda,\nu_0,R_0) < R_0$, and $C=C(K,\Lambda,\nu_0,R_0,p)$, and $\alpha = \alpha(K,\Lambda,\nu_0,R_0,p) \in (0,1)$ such that the following holds. Let $\sO\subsetneqq\HH$ be a domain. If $u \in H^1(\sO,\fw)$ satisfies the variational equation \eqref{eq:HestonVariationalEquation} with source function $f \in L^2(\sO,\fw)$ and $z_0 \in \overline{\partial_0\sO}\cap\overline{\partial_1\sO}$ is such that $f$ obeys \eqref{eq:Lsfcondition_ball}, and
$$
u = 0 \hbox{ on }\partial_1\sO\cap B_{R_0}(z_0) \quad\hbox{(in the sense of $H^1$)},
$$
and $\sO$ obeys an interior and exterior cone condition with cone $K$ at $z_0$ and a uniform exterior cone condition with cone $K$ along $\overline{\partial_1\sO} \cap B_{R_0}(z_0)$, then $u \in C^\alpha_s(\bar B_{R_1}^+(z_0))$ and satisfies \eqref{eq:MainContinuity3L2RHSu}.
\end{thm}

\subsection{Higher-order Daskalopoulos-Hamilton H\"older spaces}
\label{subsec:DHspaces}
We shall need the following higher-order weighted H\"older $C^{k,\alpha}_s$ and $C^{k,2+\alpha}_s$ norms and Banach spaces pioneered by Daskalopoulos and Hamilton \cite{DaskalHamilton1998}. We record their definition here for later reference.

\begin{defn}[$C^{k,\alpha}_s$ norms and Banach spaces]
\label{defn:DHspaces}
\cite[p. 902]{DaskalHamilton1998}
Given an integer $k\geq 0$, $\alpha \in (0,1)$, and a domain $U\subset\HH$, we say that $u\in C^{k,\alpha}_s(\bar U)$ if
$u \in C^k(\bar U)$ and
\begin{equation}
\label{eq:CkalphasNorm}
\|u\|_{C^{k,\alpha}_s(\bar U)} := \sum_{j=0}^k\|D^ju\|_{C^\alpha_s(\bar U)} < \infty.
\end{equation}
When $k=0$, we denote $C^{0,\alpha}_s(\bar U) = C^\alpha_s(\bar U)$.
\end{defn}

\begin{defn}[$C^{k,2+\alpha}_s$ norms and Banach spaces]
\label{defn:DH2spaces}
\cite[pp. 901--902]{DaskalHamilton1998}
Given an integer $k\geq 0$, $\alpha \in (0,1)$, and a domain $U\subset\HH$, we say that $u \in C^{k,2+\alpha}_s(\bar U)$ if $u\in C^{k+1,\alpha}_s(\bar U)$,
the derivatives, $D_x^{k+2-m}D_y^m$, $0\leq m\leq k+2$, of order $k+2$ are continuous on $U$, and the functions, $yD_x^{k+2-m}D_y^m$, $0\leq m\leq k+2$, extend continuously up to the boundary, $\partial U$, and those extensions belong to $C^\alpha_s(\bar U)$. We define
$$
\| u \|_{C^{k, 2+\alpha}_s(\bar U)} :=  \| u \|_{C^{k+1,\alpha}_s(\bar U)} + \| yD^2 u \|_{C^{k,\alpha}_s(\bar U)}.
$$
We say that\footnote{In \cite[p. 901]{DaskalHamilton1998}, when defining the spaces $C^{k,\alpha}_s(\sA)$ and $C^{k,2+\alpha}_s(\sA)$, it is assumed that $\sA$ is a compact subset of the \emph{closed} half-plane, $\{y\geq 0\}$.} $u\in C^{k,2+\alpha}_s(\underline{U})$ if $u\in C^{k,2+\alpha}_s(\bar V)$ for all precompact open subsets $V\Subset \underline{U}$. When $k=0$, we denote $C^{0,2+\alpha}_s(\bar U) = C^{2+\alpha}_s(\bar U)$.
\end{defn}

For any non-negative integer $k$, we let $C^k_0(\underline{U})$ denote the linear subspace of functions $u\in C^k(U)$ such that $u\in C^k(\bar V)$ for every precompact open subset $V\Subset \underline{U}$ and define $C^{\infty}_0(\underline{U}) := \cap_{k\geq 0} C^k_0(\underline{U})$. Note that we also have $C^{\infty}_0(\underline{U}) = \cap_{k\geq 0}C^{k,\alpha}_s(\underline{U}) = \cap_{k\geq 0}C^{k,2+\alpha}_s(\underline{U})$.

\section{$H^2$ regularity for solutions to the variational equation}
\label{sec:H2Regularity}
In this section, we develop $H^2_{\loc}(\underline\sO,\fw)$-regularity results for a solution, $u\in H^1(\sO,\fw)$, to the variational equation \eqref{eq:HestonVariationalEquation}, whose existence was established in \cite{Daskalopoulos_Feehan_statvarineqheston}. In \S \ref{subsec:Koch}, we give a self-contained proof (that is, independent of results in \cite{Daskalopoulos_Feehan_statvarineqheston}) of an important and powerful interior a priori estimate (Proposition \ref{prop:InteriorKochEstimate}) for solutions, $u \in H^1(\sO,\fw)$, to the variational equation \eqref{eq:HestonVariationalEquation} by exploiting an idea of Koch \cite{Koch}. In \S \ref{subsec:InteriorH2Regularity}, we prove $H^2_{\loc}(\underline\sO, \fw)$-regularity (Theorem \ref{thm:H2BoundSolutionHestonVarEqnSubdomainInterior}) for a solution, $u \in H^1(\sO,\fw)$, using finite-difference methods independent of results in \cite{Daskalopoulos_Feehan_statvarineqheston}).

\subsection{Interior Koch estimate and interior $W^{1,2}$ regularity}
\label{subsec:Koch}
We first recall the following elementary a priori estimate for a solution to the variational equation \eqref{eq:HestonVariationalEquation}.

\begin{lem}[A priori estimate for solutions to the variational equation]
\label{lem:ExistenceUniquenessEllipticHestonAprioriEstimate}
\cite[Lemma 3.20]{Daskalopoulos_Feehan_statvarineqheston}
Let $\sO\subseteqq\HH$ be a domain. Then there is a positive constant, $C=C(\Lambda,\nu_0)$, such that the following holds. If $f \in L^2(\sO,\fw)$ and $u \in H^1(\sO,\fw)$ is a solution to the variational equation \eqref{eq:HestonVariationalEquation}, then\footnote{The result trivially holds  if $(1+y)u\notin L^2(\sO,\fw)$.}
\begin{equation}
\label{eq:VariationalEqualityHestonBoundH}
\|u\|_{H^1(\sO, \fw)} \leq C\left(\|f\|_{L^2(\sO, \fw)} + \|(1+y)u\|_{L^2(\sO, \fw)}\right).
\end{equation}
\end{lem}

As in \cite[Definition 3.1]{Daskalopoulos_Feehan_statvarineqheston}, we let $A_\lambda := A+\lambda(1+y)$, for any constant $\lambda\geq 0$. We recall

\begin{thm}[Existence of smooth solutions on the half-plane]
\label{thm:HestonExistenceSmoothSolutionsHalfSpace}
\cite{Feehan_Pop_elliptichestonschauder}
Let $f\in C^\infty_0(\underline{\HH})$ and $\lambda\geq 0$. Then there is a function $u\in C^\infty(\bar\HH)$ such that
$$
A_\lambda u = f \quad\hbox{on }\HH.
$$
\end{thm}

\begin{prop}[Koch estimate on the half-plane]
\label{prop:KochEstimate}
There is a positive constant, $C=C(\Lambda,\nu_0)$, such that the following holds. If $f \in L^2(\HH,\fw)$ and $u \in H^1(\HH,\fw)$ is a solution to the variational equation \eqref{eq:HestonVariationalEquation} with $\sO=\HH$, then\footnotemark[\value{footnote}] 
$$
\|Du\|_{L^2(\HH, \fw)} \leq C\left(\|f\|_{L^2(\HH, \fw)} + \|(1+y)u\|_{L^2(\HH, \fw)}\right).
$$
\end{prop}

\begin{rmk}[Proof of the Koch estimate on the half-plane]
\label{rmk:KochEstimate}
Proposition \ref{prop:KochEstimate} is proved as \cite[Proposition 5.8]{Daskalopoulos_Feehan_statvarineqheston} when $\sO=\HH$ with the aid of Theorem \ref{thm:HestonExistenceSmoothSolutionsHalfSpace} (this is \cite[Theorem 5.2]{Daskalopoulos_Feehan_statvarineqheston} when $\sO=\HH$). However, the hypothesis in \cite[Proposition 5.8]{Daskalopoulos_Feehan_statvarineqheston} that \cite[Theorem 5.2]{Daskalopoulos_Feehan_statvarineqheston} holds for $\sO\subsetneqq\HH$ and $u\in H^1_0(\underline\sO, \fw)$ appears difficult to verify.
\end{rmk}

In order to prove an interior version of the Koch estimate on subdomains of the half-plane, we shall need the following commutator identity.

\begin{lem}[Heston bilinear map commutator identity]
\label{lem:SimpleCommutatorInnerProduct}
\cite[Corollary 2.46]{Daskalopoulos_Feehan_statvarineqheston}
Let $u, v \in H^1(\sO,\fw)$ and let $\zeta \in C^\infty(\bar\sO)$ be such that $\supp\zeta\subset\underline\sO$. Then\footnote{This a simpler version of \cite[Corollary 2.46]{Daskalopoulos_Feehan_statvarineqheston}.}
\begin{equation}
\label{eq:SimpleBilinearFormSquaredFunction}
\fa(\zeta u, v) = \fa(u, \zeta v) + ([A,\zeta]u, v)_{L^2(\sO,\fw)},
\end{equation}
where $[A,\zeta]$ is defined by identity \eqref{eq:ACommutator}.
\end{lem}

\begin{proof}
For $u\in C^\infty(\bar\sO)$ and $v\in C^\infty(\underline\sO)$, then
\begin{align*}
\fa(\zeta u, v) &= (A(\zeta u), v)_{L^2(\sO,\fw)} \quad\hbox{(by
Lemma \ref{lem:HestonIntegrationByParts}
)}
\\
&= (\zeta Au, v)_{L^2(\sO,\fw)} + ([A,\zeta]u, v)_{L^2(\sO,\fw)}
\\
&= (Au, \zeta v)_{L^2(\sO,\fw)} + ([A,\zeta]u, v)_{L^2(\sO,\fw)}
\\
&= \fa(u, \zeta v) + ([A,\zeta]u, v)_{L^2(\sO,\fw)}.
\end{align*}
By approximation, the result continues to hold for $u \in H^1(\sO,\fw)$ and $v \in H^1(\sO,\fw)$.
\end{proof}

Hence, if $u \in H^1(\sO,\fw)$ is a solution to the variational equation \eqref{eq:HestonVariationalEquation} and $\zeta \in C^\infty(\bar\sO)$ is such that $\supp\zeta\subset\underline\sO$, then $\zeta u \in H^1(\HH,\fw)$ obeys, for all $v \in H^1_0(\underline{\sO},\fw)$,
\begin{align*}
\fa(\zeta u, v) &= \fa(u, \zeta v) + ([A,\zeta]u, v)_{L^2(\sO,\fw)}
\\
&= (f, \zeta v)_{L^2(\sO,\fw)} + ([A,\zeta]u, v)_{L^2(\sO,\fw)}
\\
&= (\zeta f + [A,\zeta]u, v)_{L^2(\HH,\fw)}.
\end{align*}
Because $\supp\zeta\subset\underline\sO$, the preceding variational equation remains unchanged when the space of test functions, $H^1_0(\underline{\sO},\fw)$, is replaced by $H^1(\HH,\fw)$ and so $\zeta u \in H^1(\HH,\fw)$ obeys the variational equation,
\begin{equation}
\label{eq:LocalizedHestonVariationalEquation}
\fa(\zeta u, v) = (f_{\zeta,u}, v)_{L^2(\HH,\fw)}, \quad \forall v \in H^1(\HH,\fw),
\end{equation}
where, if $\height(\supp\zeta)<\infty$,
\begin{equation}
\label{eq:LocalizedSource}
f_{\zeta, u} := \zeta f + [A,\zeta]u \in L^2(\HH, \fw).
\end{equation}
We recall from \cite[Equation (2.33)]{Daskalopoulos_Feehan_statvarineqheston}) that
\begin{equation}
\label{eq:ACommutator}
\begin{aligned}
{}[A,\zeta]v &= -y\left((\zeta_x + \varrho\sigma\zeta_y)v_x + (\varrho\sigma\zeta_x + \sigma^2\zeta_y)v_y\right)
\\
&\quad - \frac{y}{2}\left(\zeta_{xx} + 2\varrho\sigma\zeta_{xy} + \sigma^2\zeta_{yy}\right)v
\\
&\quad - \left(r-q-\frac{y}{2}\right)\zeta_x v - \kappa(\theta-y)\zeta_y v.
\end{aligned}
\end{equation}
Noting that the derivatives $v_x$ and $v_y$ in \eqref{eq:ACommutator} are multiplied by the factor $y$, we immediately obtain the

\begin{lem}[$L^p$ commutator estimate]
\label{lem:CommutatorEstimate}
Let $\sO\subseteqq\HH$ be a domain and let $M$ be a positive constant. Then there is a positive constant, $C=C(\Lambda,\nu_0,M)$, such that the following holds. If $\zeta \in C^\infty(\bar\sO)$ is such that $\|\zeta\|_{C^2(\bar\sO)}\leq M$ and $v \in W^{1,p}_{\loc}(\sO)$, for $1\leq p\leq\infty$, then\footnote{The result trivially holds if $yDv\notin L^p(\sO,\fw)$ or $(1+y)v\notin L^p(\sO,\fw)$.}
$$
\|[A,\zeta]v\|_{L^p(\sO,\fw)} \leq C\left(\|yDv\|_{L^p(\sO,\fw)} + \|(1+y)v\|_{L^p(\sO,\fw)}\right).
$$
Moreover, if $\height(\supp\zeta)\leq\Upsilon<\infty$ and $p=2$, then there is a positive constant, $C=C(\Lambda,\nu_0,M,\Upsilon)$, such that
$$
\|[A,\zeta]v\|_{L^2(\sO,\fw)} \leq C\|v\|_{H^1(\sO,\fw)}.
$$
\end{lem}

Recall from \cite[Proposition 5.1]{Daskalopoulos_Feehan_statvarineqheston} that one has the

\begin{lem}[Weighted a priori first-order derivative estimate for a solution to the variational equation]
\label{lem:AuxiliaryWeightedH1Estimate}
\cite[Proposition 5.1 (1)]{Daskalopoulos_Feehan_statvarineqheston}
There is a positive constant, $C=C(\Lambda,\nu_0)$, such that the following holds. Let $\sO\subseteqq\HH$ be a domain, $f \in L^2(\sO,\fw)$, and $u \in H^1(\sO,\fw)$ be a solution to the variational equation \eqref{eq:HestonVariationalEquation}. Then\footnote{The result trivially holds if $(1+y)^{1/2}f \notin L^2(\sO,\fw)$ or $(1+y)u \notin L^2(\sO,\fw)$.}
\begin{equation}
\label{eq:AuxiliaryWeightedH1Estimate}
\|yDu\|_{L^2(\sO,\fw)} \leq C\left(\|y^{1/2}f\|_{L^2(\sO,\fw)} + \|(1+y)u\|_{L^2(\sO,\fw)}\right).
\end{equation}
\end{lem}

We have the following interior version of Proposition \ref{prop:KochEstimate} for a solution, $u \in H^1(\sO,\fw)$, to the variational equation \eqref{eq:HestonVariationalEquation}, given $f \in L^2(\sO,\fw)$.

\begin{prop}[Interior Koch estimate]
\label{prop:InteriorKochEstimate}
Let $\sO\subseteqq\HH$ be a domain and let $d_1>0$. Then there is a constant $C=C(\Lambda,\nu_0,d_1)$ such that the following holds. Let $f\in L^2(\sO,\fw)$ and suppose that $u\in H^1(\sO,\fw)$ satisfies the variational equation \eqref{eq:HestonVariationalEquation}. If $\sO'\subset\sO$ is a subdomain such that $\bar\sO'\subset\underline{\sO}$ and $\dist(\partial_1\sO',\partial_1\sO)\geq d_1$, then\footnotemark[\value{footnote}]
\begin{equation}
\label{eq:InteriorKochEstimate}
\|Du\|_{L^2(\sO', \fw)} \leq C\left(\|(1+y)^{1/2}f\|_{L^2(\sO, \fw)} + \|(1+y)u\|_{L^2(\sO, \fw)}\right).
\end{equation}
\end{prop}

\begin{proof}
Choose a cutoff function, $\zeta \in C^\infty(\bar\sO)$, such that $0\leq\zeta\leq 1$ on $\sO$ and $\zeta=1$ on $\sO'$ and $\supp\zeta\subset\underline\sO''$, for a subdomain $\sO''\subset\sO$ such that $\bar\sO''\subset\underline{\sO}$ and $\dist(\partial_1\sO'',\partial_1\sO)\geq d_1/2$. The conclusion now follows from Proposition \ref{prop:KochEstimate}, Equations \eqref{eq:LocalizedHestonVariationalEquation} and \eqref{eq:LocalizedSource}, and Lemmas \ref{lem:CommutatorEstimate} and \ref{lem:AuxiliaryWeightedH1Estimate}.
\end{proof}

The far more elementary a priori estimate in Lemma \ref{lem:ExistenceUniquenessEllipticHestonAprioriEstimate} may also be localized by an argument very similar to that used to prove Proposition \ref{prop:InteriorKochEstimate}.

\begin{lem}[Interior $H^1$ a priori estimate for a solution to the variational equation]
\label{lem:H1BoundSolutionHestonVarEqnSubdomainInterior}
Let $\sO\subseteqq\HH$ be a domain and let $d_1>0$. Then there is a constant $C=C(\Lambda,\nu_0,d_1)$ such that the following holds. Let $f\in L^2(\sO,\fw)$ and suppose that $u\in H^1(\sO,\fw)$ satisfies the variational equation \eqref{eq:HestonVariationalEquation}. If $\sO'\subset\sO$ is a subdomain such that $\bar\sO'\subset\underline{\sO}$ and $\dist(\partial_1\sO',\partial_1\sO)\geq d_1$, then\footnotemark[\value{footnote}]
\begin{equation}
\label{eq:H1BoundSolutionHestonVarEqnSubdomain}
\|u\|_{H^1(\sO',\fw)} \leq C\left(\|(1+y)^{1/2}f\|_{L^2(\sO,\fw)} + \|(1+y)u\|_{L^2(\sO,\fw)}\right).
\end{equation}
\end{lem}

\begin{proof}
Choose a cutoff function, $\zeta \in C^\infty(\bar\sO)$, such that $0\leq\zeta\leq 1$ on $\sO$ and $\zeta=1$ on $\sO'$ and $\supp\zeta\subset\underline\sO''$, for a subdomain $\sO''\subset\sO$ such that $\bar\sO''\subset\underline{\sO}$ and $\dist(\partial_1\sO'',\partial_1\sO)\geq d_1/2$. The conclusion now follows from Lemma \ref{lem:ExistenceUniquenessEllipticHestonAprioriEstimate}, Equations \eqref{eq:LocalizedHestonVariationalEquation} and \eqref{eq:LocalizedSource}, and Lemmas \ref{lem:CommutatorEstimate} and \ref{lem:AuxiliaryWeightedH1Estimate}.
\end{proof}

\subsection{Interior $H^2$ regularity}
\label{subsec:InteriorH2Regularity}

We denote the finite difference with respect to $x$ of a function $v$ on $\sO$ by
\begin{equation}
\label{eq:FiniteDifferencex}
\delta_x^hv(x,y) := \frac{1}{h}\left(v(x+h,y)-v(x,y)\right),
\end{equation}
for $h\in\RR\less\{0\}$ and all $(x,y) \in \sO$ with $(x+h,y)\in\sO$. We have the following analogue and extension of \cite[Theorem 5.8.3]{Evans} or \cite[Lemmas 7.23 and 7.24]{GilbargTrudinger}.

\begin{lem}[Convergence and bounds on finite differences]
\label{lem:FDBounds}
Let $\sO\subseteqq\HH$ be a domain and let $\sO'\subset\sO$ be a subdomain such that $\bar\sO' \subset \underline{\sO}$.
\begin{enumerate}
\item
\label{item:FDBounds_DiffBound}
There is a constant $C=C(\dist(\partial_1\sO',\partial_1\sO))$ such that the following holds. If $u \in L^2(\sO,\fw)$ with $u_x \in L^2(\sO,\fw)$, then
$
\|\delta_x^h u\|_{L^2(\sO',\fw)} \leq C\|u_x\|_{L^2(\sO,\fw)},
$
for all $h\in\RR$ such that $0<2|h|<\dist(\partial_1\sO',\partial_1\sO)$.
\item
\label{item:FDBounds_WeakDerivExist}
If $u \in L^2(\sO,\fw)$ and there is a constant $K>0$ such that
$
\|\delta_x^h u\|_{L^2(\sO',\fw)} \leq K,
$
for all $h\in\RR$ such that $0<2|h|<\dist(\partial_1\sO',\partial_1\sO)$, then $u_x \in L^2(\sO',\fw)$ exists and obeys $\|u_x\|_{L^2(\sO',\fw)} \leq K$.
\end{enumerate}
\end{lem}

\begin{proof}
The proof of Item \eqref{item:FDBounds_DiffBound} adapts line-by-line from the proofs of \cite[Theorem 5.8.3 (i)]{Evans} or \cite[Lemma 7.23]{GilbargTrudinger}. To prove Item \eqref{item:FDBounds_WeakDerivExist},  it is enough to notice that $L^2(\sO,\fw)$ is a separable Hilbert space (therefore, reflexive also) and so \cite[Problem 5.4]{GilbargTrudinger} applies. The proof of \cite[Lemma 7.24]{GilbargTrudinger} now adapts line-by-line.
\end{proof}

We shall adapt the proof of \cite[Theorem 8.8]{GilbargTrudinger} in order to establish

\begin{thm}[Interior regularity of second-order derivatives parallel to the degenerate boundary]
\label{thm:InteriorPartialH2Parallel}
Let $\sO\subseteqq\HH$ be a domain and let $d_1>0$. Then there is a constant $C=C(\Lambda,\nu_0,d_1)$ such that the following holds. Let $f\in L^2(\sO,\fw)$ and suppose that $u\in H^1(\sO,\fw)$ satisfies the variational equation \eqref{eq:HestonVariationalEquation}. If $\sO'\subset\sO$ is a subdomain such that $\bar\sO'\subset\underline{\sO}$ and $\dist(\partial_1\sO',\partial_1\sO)\geq d_1$, then\footnote{The result trivially holds if $(1+y)u_x \notin L^2(\sO,\fw)$.}
$
yu_{xx}, \ yu_{xy} \in L^2(\sO',\fw),
$
and
$$
\|yDu_x\|_{L^2(\sO',\fw)} \leq C\left(\|f\|_{L^2(\sO,\fw)} + \|y^{1/2}Du\|_{L^2(\sO,\fw)} + \|(1+y)u_x\|_{L^2(\sO,\fw)} + \|u\|_{L^2(\sO,\fw)}\right).
$$
\end{thm}

\begin{rmk}[Comparison with regularity results and their proofs in \cite{Daskalopoulos_Feehan_statvarineqheston}]
While stronger results than Theorem \ref{thm:InteriorPartialH2Parallel} are proved as Corollary 5.15 and Theorem 5.17 in \cite{Daskalopoulos_Feehan_statvarineqheston}, where (in the case of \cite[Corollary 5.15]{Daskalopoulos_Feehan_statvarineqheston}) the subdomain $\sO'$ is replaced by $\sO$ under suitable hypotheses on $\partial_1\sO$ and $Du_x$ is replaced by $D^2u$, the proof of \cite[Corollary 5.15]{Daskalopoulos_Feehan_statvarineqheston} relies on a hypothesis (see \cite[Theorem 5.2]{Daskalopoulos_Feehan_statvarineqheston}) in \cite{Daskalopoulos_Feehan_statvarineqheston} that there exist solutions $u\in C^\infty(\bar\sO)$ to $Au=f$ on $\sO$ and $u=0$ on $\partial_1\sO$ when $f\in C^\infty_0(\sO)$ and $\partial_1\sO$ is $C^\infty$-transverse to $\partial\HH$. In contrast, our proof of Theorem \ref{thm:InteriorPartialH2Parallel} does not rely on \cite[Theorem 5.2]{Daskalopoulos_Feehan_statvarineqheston}, whose proof appears difficult, and instead uses more elementary methods (finite differences, in
particular). See also Remark \ref{rmk:KochEstimate}.
\end{rmk}

Using the $L^2(\sO,\fw)$-analogue\footnote{The proof adapts line-by-line.} of the finite-difference integration-by-parts formula \cite[Equation (6.3.16)]{Evans}, we find that, for $f, v \in L^2(\sO,\fw)$,
\begin{equation}
\label{eq:FDIntegrationByParts}
-(f, \delta_x^{-h}v)_{L^2(\sO,\fw)} = ((\fw^h/\fw)\delta_x^{h}f, v)_{L^2(\sO,\fw)} + ((\delta_x^{h}\fw/\fw)f, v)_{L^2(\sO,\fw)},
\end{equation}
where the finite-difference product rule \cite[Equation (6.3.17)]{Evans} gives
$$
\delta_x^{h}(\fw f) = \fw^h \delta_x^{h} f + f\delta_x^{h}\fw \quad\hbox{a.e. on }\sO,
$$
with $\fw^h(x,y) := \fw(x+h,y)$.

\begin{proof}[Proof of Theorem \ref{thm:InteriorPartialH2Parallel}]
We may assume without loss of generality that $(1+y)u_x \in L^2(\sO,\fw)$. From the integral identities \eqref{eq:HestonWithKillingBilinearForm} and \eqref{eq:HestonVariationalEquation} (using our Assumption \ref{assump:HestonCoefficientb1} that $b_1=0$), we have
\begin{align*}
{}&\frac{1}{2}\int_\sO\left(u_xv_x + \varrho\sigma u_yv_x
+ \varrho\sigma u_xv_y + \sigma^2u_yv_y\right)y\,\fw\,dx\,dy
\\
&= \frac{\gamma}{2}\int_\sO\left(u_x + \varrho\sigma u_y\right)v \,\frac{x}{\sqrt{1+x^2}}y\,\fw\,dx\,dy
\\
&\quad + \int_\sO a_1u_xv y\,\fw\,dx\,dy - \int_\sO c_0uv\,\fw\,dx\,dy  + \int_\sO fv\,\fw\,dx\,dy, \quad\forall v \in C^\infty_0(\underline{\sO}).
\end{align*}
We may replace $v$ by the difference quotient, $\delta_x^{-h}v$, in the preceding identity, provided $|h|<\frac{1}{2}\dist(\supp v,\partial_1\sO)$, and use
the $L^2(\sO,\fw)$ finite-difference integration-by-parts formula \eqref{eq:FDIntegrationByParts} to find that
\begin{align*}
{}&\int_\sO (\fw^h/\fw)\left((\delta_x^h u_x)v_x + \varrho\sigma (\delta_x^h u_y)v_x
+ \varrho\sigma (\delta_x^h u_x)v_y + \sigma^2(\delta_x^h u_y)v_y\right)y\,\fw\,dx\,dy
\\
& + \int_\sO (\delta_x^h\fw/\fw)\left(u_xv_x + \varrho\sigma u_yv_x
+ \varrho\sigma u_xv_y + \sigma^2u_yv_y\right)y\,\fw\,dx\,dy
\\
&= -\int_\sO\left(u_x(\delta_x^{-h}v_x) + \varrho\sigma u_y(\delta_x^{-h}v_x)
+ \varrho\sigma u_x(\delta_x^{-h}v_y) + \sigma^2u_y(\delta_x^{-h}v_y)\right)y\,\fw\,dx\,dy
\end{align*}
and give
\begin{align*}
{}&-\frac{1}{2}\int_\sO (\fw^h/\fw)\left((\delta_x^h u_x)v_x + \varrho\sigma (\delta_x^h u_y)v_x
+ \varrho\sigma (\delta_x^h u_x)v_y + \sigma^2(\delta_x^h u_y)v_y\right)y\,\fw\,dx\,dy
\\
&= \frac{1}{2}\int_\sO (\delta_x^h\fw/\fw)\left(u_xv_x + \varrho\sigma u_yv_x
+ \varrho\sigma u_xv_y + \sigma^2u_yv_y\right)y\,\fw\,dx\,dy
\\
&\quad + \frac{\gamma}{2}\int_\sO\left(u_x + \varrho\sigma u_y\right)(\delta_x^{-h}v) \,\frac{x}{\sqrt{1+x^2}}y\,\fw\,dx\,dy
\\
&\quad + \int_\sO a_1u_x(\delta_x^{-h}v) y\,\fw\,dx\,dy - \int_\sO c_0u(\delta_x^{-h}v)\,\fw\,dx\,dy  + \int_\sO f(\delta_x^{-h}v)\,\fw\,dx\,dy.
\end{align*}
Therefore,
\begin{align*}
{}&\frac{1}{2}\int_\sO (\fw^h/\fw)\left((\delta_x^h u_x)v_x + \varrho\sigma (\delta_x^h u_y)v_x
+ \varrho\sigma (\delta_x^h u_x)v_y + \sigma^2(\delta_x^h u_y)v_y\right)y\,\fw\,dx\,dy
\\
&\leq C\|y^{1/2}Du\|_{L^2(\sO,\fw)}\left(\|y^{1/2}Dv\|_{L^2(\sO,\fw)} + \|y^{1/2}\delta_x^{-h}v\|_{L^2(\sO,\fw)}\right)
\\
&\quad  + C\left(\|u\|_{L^2(\sO,\fw)} + \|f\|_{L^2(\sO,\fw)}\right)\|\delta_x^{-h}v\|_{L^2(\sO,\fw)}
\\
&\leq C\|y^{1/2}Du\|_{L^2(\sO,\fw)}\|y^{1/2}Dv\|_{L^2(\sO,\fw)} + C\left(\|u\|_{L^2(\sO,\fw)} + \|f\|_{L^2(\sO,\fw)}\right)\|v_x\|_{L^2(\sO,\fw)},
\end{align*}
where the final inequality follows from Lemma \ref{lem:FDBounds} \eqref{item:FDBounds_DiffBound}. Now choose $\zeta\in C^\infty(\underline{\sO})$ with $0\leq \zeta\leq 1$ on $\sO$ and $\zeta=1$ on $\sO'$ and $\supp\zeta\subset\underline\sO$, and set $v = y\zeta^2\delta_x^h u$ with $|h|<\frac{1}{2}\dist(\supp\zeta,\partial_1\sO)$. Therefore, applying \eqref{eq:HestonModulusEllipticity}, we obtain
\begin{align*}
{}&\nu_0\int_\sO (\fw^h/\fw)|\zeta D\delta_x^h u|^2 y^2\,\fw\,dx\,dy
\\
&\leq \frac{1}{2}\int_\sO (\fw^h/\fw)\zeta^2\left((\delta_x^h u_x)^2 + 2\varrho\sigma (\delta_x^h u_y)(\delta_x^h u_x)
+ \sigma^2(\delta_x^h u_y)^2\right)y^2\,\fw\,dx\,dy,
\end{align*}
and using
\begin{align*}
y\zeta^2\delta_x^hu_x &= (y\zeta^2\delta_x^hu)_x - 2y\zeta\zeta_x\delta_x^hu = v_x - 2y\zeta\zeta_x\delta_x^hu,
\\
y\zeta^2\delta_x^hu_y &= (y\zeta^2\delta_x^hu)_y - \zeta^2\delta_x^hu - 2y\zeta\zeta_y\delta_x^hu = v_y - (\zeta^2 + 2y\zeta\zeta_y)\delta_x^hu,
\end{align*}
we obtain
\begin{align*}
{}&\nu_0\int_\sO (\fw^h/\fw)|\zeta D\delta_x^h u|^2 y^2\,\fw\,dx\,dy
\\
&\leq \frac{1}{2}\int_\sO (\fw^h/\fw)\left((\delta_x^h u_x)\left(v_x - 2y\zeta\zeta_x\delta_x^hu\right)
+ \varrho\sigma (\delta_x^h u_y)\left(v_x - 2y\zeta\zeta_x\delta_x^hu\right)  \right.
\\
&\qquad + \left. \varrho\sigma (\delta_x^h u_x)\left(v_y - (\zeta^2 + 2y\zeta\zeta_y)\delta_x^hu\right)
+ \sigma^2(\delta_x^h u_y)\left(v_y - (\zeta^2 + 2y\zeta\zeta_y)\delta_x^hu\right)\right)y\,\fw\,dx\,dy
\\
&= \frac{1}{2}\int_\sO (\fw^h/\fw)\left((\delta_x^h u_x)v_x + \varrho\sigma (\delta_x^h u_y)v_x
+ \varrho\sigma (\delta_x^h u_x)v_y + \sigma^2(\delta_x^h u_y)v_y\right)y\,\fw\,dx\,dy
\\
&\quad - \frac{1}{2}\int_\sO (\fw^h/\fw)\left((\delta_x^h u_x)2y\zeta\zeta_x\delta_x^hu
+ \varrho\sigma (\delta_x^h u_y) 2y\zeta\zeta_x\delta_x^hu \right.
\\
&\qquad + \left. \varrho\sigma (\delta_x^h u_x)(\zeta^2 + 2y\zeta\zeta_y)\delta_x^hu
+ \sigma^2(\delta_x^h u_y)(\zeta^2 + 2y\zeta\zeta_y)\delta_x^hu \right)y\,\fw\,dx\,dy.
\end{align*}
Hence, there is a positive constant $C=C(\Lambda,\nu_0,\dist(\partial_1\sO',\partial_1\sO))$
such that
\begin{align*}
{}&\|\zeta yD\delta_x^h u\|_{L^2(\sO,\fw)}^2
\\
&\quad\leq C\|y^{1/2}Du\|_{L^2(\sO,\fw)}\|y^{1/2}Dv\|_{L^2(\sO,\fw)} + C\left(\|u\|_{L^2(\sO,\fw)} + \|f\|_{L^2(\sO,\fw)}\right)\|v_x\|_{L^2(\sO,\fw)}
\\
&\qquad + C\|\zeta yD\delta_x^h u\|_{L^2(\sO,\fw)}\|(1+y)\delta_x^h u\|_{L^2(\supp\zeta,\fw)}
\\
&\quad\leq C\|y^{1/2}Du\|_{L^2(\sO,\fw)}\left(\|\zeta yD\delta_x^h u\|_{L^2(\sO,\fw)}  + \|(1+y)\delta_x^h u\|_{L^2(\supp\zeta,\fw)}\right)
\\
&\qquad + C\left(\|u\|_{L^2(\sO,\fw)} + \|f\|_{L^2(\sO,\fw)}\right)\left(\|\zeta y\delta_x^h u_x\|_{L^2(\sO,\fw)} + \|y\delta_x^h u\|_{L^2(\supp\zeta,\fw)}\right)
\\
&\qquad + C\|\zeta yD\delta_x^h u\|_{L^2(\sO,\fw)}\|(1+y)\delta_x^h u\|_{L^2(\supp\zeta,\fw)},
\end{align*}
where, to obtain the final inequality, we used
$$
v_x = y\zeta^2\delta_x^hu_x + 2y\zeta\zeta_x \delta_x^hu, \quad v_y = y\zeta^2\delta_x^hu_y + (\zeta^2 + 2y\zeta\zeta_y)\delta_x^hu.
$$
Applying Young's inequality (that is, $2ab \leq \eps a^2 + \eps^{-1}b^2$, for any $\eps>0$ and $a,b\in\RR$) to the terms on the right containing the factor $\|\zeta yD\delta_x^h u\|_{L^2(\sO,\fw)}$ and rearranging, we obtain
\begin{align*}
{}&\|\zeta yD\delta_x^h u\|_{L^2(\sO,\fw)}^2
\\
&\quad\leq C\|y^{1/2}Du\|_{L^2(\sO,\fw)}^2 + C\|y^{1/2}Du\|_{L^2(\sO,\fw)}\|(1+y)\delta_x^h u\|_{L^2(\supp\zeta,\fw)}
\\
&\qquad + C\left(\|u\|_{L^2(\sO,\fw)} + \|f\|_{L^2(\sO,\fw)}\right)^2 + C\left(\|u\|_{L^2(\sO,\fw)}
+ \|f\|_{L^2(\sO,\fw)}\right)\|y\delta_x^h u\|_{L^2(\supp\zeta,\fw)}
\\
&\qquad + C\|(1+y)\delta_x^h u\|_{L^2(\supp\zeta,\fw)}^2,
\end{align*}
for all $h\in\RR$ such that $|h|<\frac{1}{2}\dist(\supp\zeta,\partial_1\sO)$. Again applying Lemma \ref{lem:FDBounds} \eqref{item:FDBounds_DiffBound}, we see that
\begin{align*}
{}&\|yDu_x\|_{L^2(\sO',\fw)}^2
\\
&\quad\leq C\|y^{1/2}Du\|_{L^2(\sO,\fw)}\left(\|y^{1/2}Du\|_{L^2(\sO,\fw)} + \|(1+y)u_x\|_{L^2(\sO,\fw)}\right)
\\
&\quad + C\left(\|u\|_{L^2(\sO,\fw)} + \|f\|_{L^2(\sO,\fw)}\right)\left(\|u\|_{L^2(\sO,\fw)} + \|f\|_{L^2(\sO,\fw)} + \|yu_x\|_{L^2(\sO,\fw)}\right)
\\
&\qquad + C\|(1+y)u_x\|_{L^2(\sO,\fw)}^2,
\end{align*}
and taking square roots completes the proof.
\end{proof}

Proceeding by analogy with the proof of \cite[Theorem 8.12]{GilbargTrudinger} to estimate $yu_{yy}$, we obtain

\begin{lem}[Interior regularity of second-order derivatives orthogonal to the degenerate boundary]
\label{lem:InteriorPartialH2Orthogonal}
There is a constant $C=C(\Lambda,\nu_0)$ such that the following holds. Let $\sO\subseteqq\HH$ be a domain and let $\sO'\subseteqq\sO$ be a subdomain. Let $f\in L^2(\sO,\fw)$ and suppose that $u\in H^1(\sO,\fw)$ satisfies the variational equation \eqref{eq:HestonVariationalEquation}. If $yu_{xx}, \ yu_{xy} \in L^2(\sO',\fw)$, then
$
yu_{yy} \in L^2(\sO',\fw),
$
and
\begin{equation}
\label{eq:InteriorPartialH2Orthogonal}
\begin{aligned}
\|yu_{yy}\|_{L^2(\sO',\fw)} &\leq C\left(\|yu_{xx}\|_{L^2(\sO',\fw)} + \|yu_{xy}\|_{L^2(\sO',\fw)} + \|(1+y)Du\|_{L^2(\sO',\fw)} \right.
\\
&\qquad + \left. \|u\|_{L^2(\sO',\fw)} + \|f\|_{L^2(\sO',\fw)}\right).
\end{aligned}
\end{equation}
\end{lem}

\begin{proof}
From \cite[Theorem 8.8]{GilbargTrudinger}, we know that $u\in W^{2,2}_{\loc}(\sO)$ and $Au=f$ a.e. on $\sO$, and thus by \eqref{eq:OperatorHestonIntro}, we have
$$
\frac{\sigma^2}{2}yu_{yy} = -\frac{y}{2}\left(u_{xx} + 2\varrho\sigma u_{xy}\right) - \left(c_0-q-\frac{y}{2}\right)u_x - \kappa(\theta-y)u_y + c_0u - f.
$$
Hence, there is a constant, $C=C(\Lambda,\nu_0)$, such that \eqref{eq:InteriorPartialH2Orthogonal} holds.
\end{proof}

Therefore, we find that

\begin{thm}[Interior regularity of second-order derivatives]
\label{thm:InteriorD2u}
Let $\sO\subseteqq\HH$ be a domain and let $d_1>0$. Then there is a constant $C=C(\Lambda,\nu_0,d_1)$ such that the following holds. Let $f\in L^2(\sO,\fw)$ and suppose that $u\in H^1(\sO,\fw)$ satisfies the variational equation \eqref{eq:HestonVariationalEquation}. If $(1+y)^{1/2}f$ and $(1+y)u$ belong to $L^2(\sO,\fw)$ and $\sO'\subset\sO$ is a subdomain such that $\bar\sO'\subset\underline{\sO}$ and $\dist(\partial_1\sO',\partial_1\sO)\geq d_1$, then
$
yu_{xx}, \ yu_{xy}, \ yu_{yy} \in L^2(\sO',\fw),
$
and
$$
\|yD^2u\|_{L^2(\sO',\fw)} \leq C\left(\|(1+y)^{1/2}f\|_{L^2(\sO,\fw)} + \|(1+y)u\|_{L^2(\sO,\fw)}\right).
$$
\end{thm}

\begin{proof}
The conclusion follows by combining the estimates in Proposition \ref{prop:InteriorKochEstimate}, Theorem \ref{thm:InteriorPartialH2Parallel}, Lemma \ref{lem:InteriorPartialH2Orthogonal} and the a priori $H^1(\sO,\fw)$-estimate for a solution $u$ given by Lemma \ref{lem:ExistenceUniquenessEllipticHestonAprioriEstimate} and the $L^2(\sO,\fw)$-estimate for $yDu$ in Lemma \ref{lem:AuxiliaryWeightedH1Estimate}.
\end{proof}

Consequently, we have the

\begin{thm}[Interior $H^2$ regularity and a priori estimate]
\label{thm:InteriorH2}
Let $\sO\subseteqq\HH$ be a domain and let $d_1>0$. Then there is a constant, $C=C(\Lambda,\nu_0,d_1)$, such that the following holds. If $f\in L^2(\sO,\fw)$ and $u\in H^1(\sO,\fw)$ satisfies the variational equation \eqref{eq:HestonVariationalEquation}, then $u \in H^2_{\loc}(\underline\sO,\fw)$. Moreover, if $(1+y)^{1/2}f$ and $(1+y)u$ belong to $L^2(\sO,\fw)$ and $\sO'\subset\sO$ is a subdomain such that $\bar\sO'\subset\underline{\sO}$ and $\dist(\partial_1\sO',\partial_1\sO)\geq d_1$, then $u \in H^2(\sO',\fw)$ and
$$
\|u\|_{H^2(\sO',\fw)} \leq C\left(\|(1+y)^{1/2}f\|_{L^2(\sO,\fw)} + \|(1+y)u\|_{L^2(\sO,\fw)}\right).
$$
\end{thm}

\begin{proof}
The conclusion follows by combining Proposition \ref{prop:InteriorKochEstimate} with Theorem \ref{thm:InteriorD2u}, the a priori $H^1(\sO,\fw)$-estimate for a solution $u$ given by Lemma \ref{lem:ExistenceUniquenessEllipticHestonAprioriEstimate}, and the $L^2(\sO,\fw)$-estimate for $yDu$ in Lemma \ref{lem:AuxiliaryWeightedH1Estimate}.
\end{proof}

In the sequel, we shall most often apply Theorem \ref{thm:InteriorH2} in the following special form.

\begin{thm}[Interior $H^2$ regularity for a solution to the variational equation]
\label{thm:H2BoundSolutionHestonVarEqnSubdomainInterior}
Let $\sO\subseteqq\HH$ be a domain and let $R<R_0$ be positive constants. Then there is a positive constant, $C=C(\Lambda,\nu_0,R,R_0)$, such that the following holds. If $f\in L^2(\sO,\fw)$ and $u \in H^1(\sO,\fw)$ is a solution to the variational equation \eqref{eq:HestonVariationalEquation}, and $z_0 \in \partial_0\sO$ is such that
$
\HH\cap B_{R_0}(z_0) \subset \sO,
$
then
$
u \in H^2(B_R^+(z_0),\fw),
$
and
\begin{equation}
\label{eq:H2BoundSolutionHestonVarEqnSubdomain}
\|u\|_{H^2(B_R^+(z_0),\fw)} \leq C\left(\|f\|_{L^2(B_{R_0}^+(z_0),\fw)} + \|u\|_{L^2(B_{R_0}^+(z_0),\fw)}\right).
\end{equation}
\end{thm}

\section{Higher-order Sobolev regularity for solutions to the variational equation}
\label{sec:SobolevRegularity}
In this section, we develop higher-order ``interior'' regularity results for a solution, $u\in H^1(\sO,\fw)$, to the variational equation \eqref{eq:HestonVariationalEquation}. After providing motivation for their construction in \S \ref{subsec:HigherWeightedSobolevNorms}, we describe the families of higher-order weighted Sobolev spaces which we shall need for this article, namely $\sH^k(\sO,\fw)$ (Definition \ref{defn:HkWeightedSobolevSpaceNormPowery}) and $W^{k,p}(\sO,\fw)$ (Definition \ref{defn:HkWeightedSobolevSpaceNormSingleWeight}).

We begin our development of higher-order Sobolev regularity theory in \S \ref{subsec:InteriorH2DxRegularity}, where we establish $H^2_{\loc}(\underline\sO, \fw)$-regularity (Proposition \ref{prop:SobolevRegularity_ukx}) of the derivatives, $D_x^k u$, of a solution, $u \in H^1(\sO,\fw)$, to the variational equation \eqref{eq:HestonVariationalEquation}, while in \S \ref{subsec:InteriorH2DyRegularity}, we establish $H^2_{\loc}(\underline\sO, \fw)$-regularity (Proposition \ref{prop:SobolevRegularity_uy}) of the derivative, $u_y$. The preceding regularity results are combined in \S \ref{subsec:InteriorH3Regularity} to give $\sH^3_{\loc}(\underline\sO, \fw)$-regularity (Theorem \ref{thm:H3SobolevRegularityInterior}) of a solution, $u \in H^1(\sO,\fw)$. We conclude in \S \ref{subsec:InteriorHk+2Regularity} with a proof of two of the main results of our article, Theorems \ref{thm:HkSobolevRegularityDomain} and Theorem \ref{thm:ExistUniqueHkSobolevRegularityDomain}, and $\sH^{k+2}_{\loc}(\underline\sO, \fw)$-
regularity of a solution, $u \in H^1(\sO,\fw)$, for any integer $k\geq 0$.

\subsection{Motivation and definition of higher-order weighted Sobolev norms}
\label{subsec:HigherWeightedSobolevNorms}
We now extend our previous definition of $H^\ell(\sO,\fw)$ when $\ell=0,1,2$ (see \cite[Definitions 2.15 and 2.20]{Daskalopoulos_Feehan_statvarineqheston}) to allow $\ell\geq 2$. For $k\geq 0$, it is natural to define $H^{k+2}(\sO,\fw)$ as a Sobolev space contained in the domain of $D_x^{k-m}D_y^m A$, so the operators
$$
D_x^{k-m}D_y^m A:H^{k+2}(\sO,\fw)\to L^2(\sO,\fw), \quad m \in\NN, \quad 0\leq m\leq k,
$$
are bounded, and we use this principle as a guide to our definition.

From the expression \eqref{eq:OperatorHestonIntro} for $A$, we have, for $v\in C^\infty(\sO)$,
$$
[D_x, A]v := D_xAv - AD_xv = 0 \quad\hbox{on }\sO,
$$
and so
$$
[D_x^m, A]v \equiv AD_x^mv - D_x^mAv = 0 \quad\hbox{on }\sO, \quad\forall m\in\NN,
$$
whereas
\begin{equation}
\label{eq:SimpleDyCommutator}
[D_y,A]v \equiv D_yAv - AD_yv = -\frac{1}{2}\left(v_{xx} + 2\varrho\sigma v_{xy} + \sigma^2 v_{yy}\right) + \frac{1}{2}v_x + \kappa v_y  \quad\hbox{on }\sO,
\end{equation}
is a (non-trivial) second-order, elliptic operator with constant coefficients (and therefore commutes with both $D_x$ and $D_y$). Hence,
\begin{align*}
D_y^2Av &= D_yAD_yv + D_y[D_y,A]v
\\
&= AD_y^2v + [D_y,A]D_yv + [D_y,A]D_yv
\\
&= AD_y^2v + 2[D_y,A]D_yv,
\end{align*}
while
\begin{align*}
D_y^3Av &= D_y^2AD_yv + D_y^2[D_y,A]v
\\
&= AD_y^3v + 2[D_y,A]D_y^2v + [D_y,A]D_y^2v
\\
&= AD_y^3v + 3[D_y,A]D_y^2v,
\end{align*}
and, by induction,
$$
[D_y^m, A]v \equiv D_y^mAv - AD_y^mv = m[D_y,A]D_y^{m-1}v \quad\hbox{on }\sO, \quad\forall m \in \NN.
$$
By combining the two cases, we obtain
\begin{equation}
\label{eq:GeneralizedDxDyCommutator}
[D_x^{k-m}D_y^m, A]v \equiv D_x^{k-m}D_y^m Av - AD_x^{k-m}D_y^mv = m[D_y,A]D_x^{k-m}D_y^{m-1}v \quad\hbox{on }\sO,
\end{equation}
for all $k, m \in \NN$ with $0\leq m\leq k$.

Given $v \in H^k(\sO,\fw)$ and $k\geq 2$ and a suitable definition of $H^k(\sO,\fw)$, we should expect that
$$
D_x^mD_y^n Av \in L^2(\sO,\fw), \quad 0\leq m+n\leq k-2,
$$
and so,
$$
AD_x^mD_y^nv, \ [D_y,A]D_x^mD_y^{n-1}v \in L^2(\sO,\fw), \quad 0\leq m+n\leq k-2.
$$
The second condition is fulfilled when
$$
D_x^mD_y^nv \in  L^2(\sO,\fw), \quad 0\leq m+n\leq k-1,
$$
whereas the expression \eqref{eq:OperatorHestonIntro} for $A$ implies that the first condition is fulfilled when
\begin{align*}
yD_x^mD_y^nv &\in  L^2(\sO,\fw), \quad 1\leq m+n\leq k,
\\
D_x^mD_y^nv &\in  L^2(\sO,\fw), \quad 0\leq m+n\leq k-1.
\end{align*}
Therefore, when $k\geq 2$, and keeping in mind that we want $H^k(\sO,\fw) \subset H^1(\sO,\fw)$, for all $k\geq 2$, we make the

\begin{defn}[Higher-order weighted Sobolev spaces]
\label{defn:HkWeightedSobolevSpaceNorm}
Let $\sO\subseteqq\HH$ be a domain. For any integer $k\geq 1$, set
$$
H^{k+2}(\sO,\fw) := \left\{v \in W^{k+2,2}_{\loc}(\sO): \|v\|_{H^{k+2}(\sO,\fw)} < \infty \right\},
$$
where
\begin{equation}
\label{eq:HkWeightedSobolevSpaceNorm}
\begin{aligned}
\|v\|_{H^{k+2}(\sO,\fw)}^2 &:= \int_\sO y^2|D^{k+2}v|^2\,\fw\,dx\,dy + \sum_{j=1}^{k+1}\int_\sO(1+y)^2|D^jv|^2\,\fw\,dx\,dy
\\
&\qquad + \int_\sO(1+y)v^2\,\fw\,dx\,dy,
\end{aligned}
\end{equation}
and $D^jv$ denotes the vector $(D_x^{j-m}D_y^m v: 0\leq m\leq j)$, for $0\leq j\leq k$.
\end{defn}

As we shall later see, Definition \ref{defn:HkWeightedSobolevSpaceNorm} is not well-adapted to a development of a higher-order regularity theory for solutions to \eqref{eq:IntroBoundaryValueProblem} or \eqref{eq:HestonVariationalEquation}, and it is best regarded as a stepping stone to the one we ultimately adopt for our regularity theory, namely Definition \ref{defn:HkWeightedSobolevSpaceNormPowery}. By way of motivation, we observe that the expression \eqref{eq:SimpleDyCommutator} for the commutator $[D_y,A]$ involves both derivatives with respect to $x$ and $y$. The alternative ``commutator'' provided by \eqref{eq:RefinedCommutator} will prove more useful than \eqref{eq:SimpleDyCommutator} in our approach to the higher-order regularity since it only involves derivatives with respect to $x$.

\begin{lem}[Alternative commutator of $A$ and $D_y$]
\label{lem:AlternativeDyAcommutator}
For any integer $m \geq 1$,
\begin{equation}
\label{eq:RefinedCommutator}
D_y^mAv - A_mD_y^mv  = mBD_y^{m-1}v \quad\hbox{on }\sO, \quad \forall v\in C^\infty(\HH),
\end{equation}
where
\begin{equation}
\label{eq:B}
Bv := -\frac{1}{2}v_{xx} + \frac{1}{2}v_x,
\end{equation}
and $A_m$ is obtained from the expression for $A$ in \eqref{eq:OperatorHestonIntro} by replacing $\theta$ by $\theta_m := \theta+m\sigma^2/(2\kappa)$ (and $\beta$
defined in \eqref{eq:DefnBetaMu}
is replaced by $\beta_m:=\beta+m$), and $q$ by $q_m := q-m\varrho\sigma$, and $c_0$ by $c_{0,m} := c_0+m\kappa$.
\end{lem}

\begin{proof}
We compute that
\begin{align*}
D_yAv &= -\frac{y}{2}\left(v_{xxy} + 2\varrho\sigma v_{xyy} + \sigma^2v_{yyy}\right) - \left(c_0-q-\frac{y}{2}\right)v_{xy} - \kappa(\theta-y)v_{yy}
\\
&\qquad -\frac{1}{2}\left(v_{xx} + 2\varrho\sigma v_{xy} + \sigma^2 v_{yy}\right) + \frac{1}{2}v_x + \kappa v_y + c_0v_y
\\
&= -\frac{y}{2}\left(v_{xxy} + 2\varrho\sigma v_{xyy} + \sigma^2v_{yyy}\right) - \left(c_0-q+\varrho\sigma-\frac{y}{2}\right)v_{xy} - \kappa\left(\theta+\frac{\sigma^2}{2\kappa}-y\right)v_{yy}
\\
&\qquad -\frac{1}{2}v_{xx} + \frac{1}{2}v_x + (c_0+\kappa)v_y
\\
&=  A_1D_yv + Bv,
\end{align*}
where $A_1$ is defined by replacing $\theta$ by $\theta_1 = \theta+\sigma^2/(2\kappa)$ (and $\beta$ by $\beta_1=\beta+1$), and $q$ by $q_1 = q-\varrho\sigma$, and the coefficient, $c_0$, of $v$ by $c_{0,1} = c_0+\kappa$, and $Bv = -\frac{1}{2}v_{xx} + \frac{1}{2}v_x$. Note that $B$ is a linear, second-order differential operator which commutes with $D_y$. Computing $D_y^2Av$, we see that
\begin{align*}
D_y^2Av &= D_y(D_yAv) = D_y\left(A_1D_yv + Bv\right)
\\
&= A_2D_y^2v + BD_yv + D_yBv
\\
&= A_2D_y^2v + 2BD_yv,
\end{align*}
where $A_2$ is defined by replacing $\theta_1$ by $\theta_2 = \theta_1+\sigma^2/(2\kappa) = \theta+2\sigma^2/(2\kappa)$ (and $\beta_1$ by $\beta_2=\beta_1+1=\beta+2$), and $q_1$ by $q_2 = q_1-\varrho\sigma = q-2\varrho\sigma$, and the coefficient, $c_{0,1}$, of $v$ by $c_{0,2} = c_{0,1}+\kappa = c_0+2\kappa$. It is now clear that the stated formula for $D_y^mAv$ follows by induction.
\end{proof}

Recall that the weight function \eqref{eq:HestonWeight} for our weighted Sobolev spaces is given by
$$
\fw(x,y) = y^{\beta-1}e^{-\gamma\sqrt{1+x^2}-\mu y}, \quad (x,y) \in \HH,
$$
where $\beta=2\kappa\theta/\sigma^2$ and $\mu = 2\kappa/\sigma^2$, and thus is defined by the coefficients of $A$ (and $\gamma$). We denote the weight defined by the corresponding coefficients of the operator $A_m$ by
\begin{equation}
\label{eq:HestonWeightn}
\fw_m(x,y) := y^m\fw(x,y) = y^{\beta+m-1}e^{-\gamma\sqrt{1+x^2}-\mu y}, \quad (x,y) \in \HH, \quad m\geq 1,
\end{equation}
noting that $\beta_m = \beta+m$ and $\kappa_m = \kappa$. Similarly, the bilinear map, $\fa(u,v)$, defined in \eqref{eq:HestonWithKillingBilinearForm} by the coefficients of $A$ (and $\gamma$) has an analogue, which we denote by $\fa_m(u,v)$, defined by the coefficients of $A_m$ (and $\gamma$), for $u\in C^\infty(\bar\sO)$ and $v\in C^\infty_0(\underline{\sO})$, with the property that
\begin{equation}
\label{eq:HestonWithKillingBilinearForm_powery}
\fa_m(u,v) =(A_mu, v)_{L^2(\sO,\fw_m)}, \quad\forall u \in C^\infty(\bar\sO), \ v\in C^\infty_0(\underline{\sO}).
\end{equation}
When $k, m\in\NN$, one may define $H^k(\sO,\fw_m)$ by simply replacing the weight $\fw$ by $\fw_m$ in the definitions of $H^k(\sO,\fw)$. We shall need to introduce the following alternative definition of higher-order Sobolev spaces which lie between $H^{k+2}(\sO,\fw_k)$ and  $H^{k+2}(\sO,\fw)$, when $k\geq 1$.

\begin{defn}[Alternative higher-order weighted Sobolev spaces]
\label{defn:HkWeightedSobolevSpaceNormPowery}
Let $\sO\subseteqq\HH$ be a domain. For $\ell=0,1,2$, define $\sH^\ell(\sO,\fw):=H^\ell(\sO,\fw)$ and, for any integer $k\geq 1$, set
$$
\sH^{k+2}(\sO,\fw) := \left\{v \in W^{k+2,2}_{\loc}(\sO): \|v\|_{\sH^{k+2}(\sO,\fw)} < \infty \right\},
$$
where
\begin{equation}
\label{eq:HkWeightedSobolevSpaceNormPowery}
\begin{aligned}
\|v\|_{\sH^{k+2}(\sO,\fw)}^2 &:= \int_\sO y^2\left(|D^{k+2}_x v|^2 + |D^{k+1}_xD_y v|^2 + |D^{k}_xD_y^2 v|^2 \right)\,\fw\,dx\,dy
\\
&\qquad + \sum_{m=1}^{k} \int_\sO y^2|D^{k-m}_xD_y^{m+2} v|^2\,\fw_m\,dx\,dy
\\
&\qquad + \sum_{j=0}^{k}\int_\sO (1+y)^2\left(|D^{j+1}_x v|^2 + |D^{j}_xD_y v|^2\right)\,\fw\,dx\,dy
\\
&\qquad + \sum_{j=1}^{k}\sum_{m=1}^{j}\int_\sO (1+y)^2|D^{j-m}_xD_y^{m+1} v|^2\,\fw_m\,dx\,dy
\\
&\qquad + \int_\sO (1+y)v^2\,\fw\,dx\,dy.
\end{aligned}
\end{equation}
We denote $\sH^2(\sO,\fw) = H^2(\sO,\fw)$, when $k=0$.
\end{defn}

For example, if $k=1$,
\begin{equation}
\label{eq:H3WeightedSobolevSpaceNormPowery}
\begin{aligned}
\|v\|_{\sH^3(\sO,\fw)}^2 &:= \int_\sO y^2\left(v_{xxx}^2 + v_{xxy}^2 + v_{xyy}^2 + yv_{yyy}^2\right)\,\fw\,dx\,dy
\\
&\qquad + \int_\sO (1+y)^2\left(v_{xx}^2 + v_{xy}^2 + yv_{yy}^2\right)\,\fw\,dx\,dy
\\
&\qquad + \int_\sO (1+y)^2\left(v_x^2 + v_y^2\right)\,\fw\,dx\,dy + \int_\sO (1+y)v^2\,\fw\,dx\,dy.
\end{aligned}
\end{equation}
Observe that if $\sO\subset\HH$ is a subdomain of finite height, then
\begin{equation}
\label{eq:SobolevInclusionDifferentWeights}
L^2(\sO,\fw) \subset L^2(\sO,\fw_m),
\end{equation}
for all $m \in \NN$. When $k\geq 1$,
$$
\|v\|_{H^{k+2}(\sO,\fw_k)} \leq C\|v\|_{\sH^{k+2}(\sO,\fw)} \leq C\|v\|_{H^{k+2}(\sO,\fw)},
$$
and so
$
H^{k+2}(\sO,\fw) \subset \sH^{k+2}(\sO,\fw) \subset  H^{k+2}(\sO,\fw_k),
$
when $\sO$ has finite height. Definition \ref{defn:HkWeightedSobolevSpaceNormPowery} gives the following inductive inequality,
\begin{equation}
\label{eq:Hk+3induction}
\begin{aligned}
\|v\|_{\sH^{k+3}(\sO,\fw)}^2 &\leq \int_\sO y^2\left(|D^{k+3}_x v|^2 + |D^{k+2}_xD_y v|^2 + |D^{k+1}_xD_y^2 v|^2 \right)\,\fw\,dx\,dy
\\
&\qquad + \sum_{m=1}^{k+1} \int_\sO y^2|D^{k+1-m}_xD_y^{m+2} v|^2\,\fw_m\,dx\,dy
\\
&\qquad + \int_\sO (1+y)^2\left(|D^{k+2}_x v|^2 + |D^{k+1}_xD_y v|^2\right)\,\fw\,dx\,dy
\\
&\qquad + \sum_{m=1}^{k+1}\int_\sO (1+y)^2|D^{k+1-m}_xD_y^{m+1} v|^2\,\fw_m\,dx\,dy
\\
&\qquad + \|v\|_{\sH^{k+2}(\sO,\fw)}^2,
\end{aligned}
\end{equation}
for $k\geq 0$. Equation
\eqref{eq:H3WeightedSobolevSpaceNormPowery}
gives the inductive inequality for $\sH^3(\sO,\fw)$.

We recall the definition of a weighted Sobolev space and norm, where the weight is the same for all derivatives of the function (denoted $W^{k,p}_w(\sO)$ in \cite[Definition 2.1.1]{Turesson_2000}, though we shall not require that $w$ be an $A_p$ weight in this article).

\begin{defn}[Higher-order weighted Sobolev spaces with a single weight]
\label{defn:HkWeightedSobolevSpaceNormSingleWeight}
Let $\sO\subseteqq\HH$ be a domain and let $w \in L^1_{\loc}(\sO)$ be a \emph{weight function}, so that $w>0$ a.e. on $\sO$. For any $1\leq p < \infty$ and integer $k\geq 0$, set
$$
W^{k,p}(\sO,w) := \left\{v \in W^{k,p}_{\loc}(\sO): \|v\|_{W^{k,p}(\sO,w)} < \infty \right\},
$$
where
\begin{equation}
\label{eq:HkWeightedSobolevSpaceNormSingleWeight}
\|v\|_{W^{k,p}(\sO,w)} := \left(\sum_{j=0}^k \int_\sO |D^jv|^p\,w\,dx\,dy\right)^{1/p}.
\end{equation}
We denote $W^{k,p}(\sO,w) = L^p(\sO,w)$, when $k=0$.
\end{defn}

Finally, we shall need the following ``interior'' versions of the weighted Sobolev spaces defined in this subsection.

\begin{defn}[Interior weighted Sobolev norms]
Let $T\subset\partial\sO$ be relatively open in $\RR^2$ and let $k\geq 0$ be an integer. We say that $v \in H^k_{\loc}(\sO\cup T, \fw)$ (respectively, $\sH^k_{\loc}(\sO\cup T, \fw)$ or $W^{k,p}_{\loc}(\sO\cup T, \fw)$) if for every subdomain $U\subset\sO$ such that $U\Subset\sO\cup T$, we have $v \in H^k(U, \fw)$ (respectively, $\sH^k(U, \fw)$ or $W^{k,p}(U, \fw)$).
\end{defn}

\subsection{Interior $H^2$ regularity for first-order derivatives parallel to the degenerate boundary}
\label{subsec:InteriorH2DxRegularity}
We proceed in a manner similar to that in \cite[p. 186]{GilbargTrudinger}.

\begin{lem}[Variational equation for the derivative of a solution with respect to $x$]
\label{lem:VarEqn_Dx}
Let $\sO\subseteqq\HH$ be a domain with \emph{finite height},\footnote{As one can see from the proof, the hypothesis that $\sO$ has finite height is only used in a very mild way and the condition could be removed using more precise bounds, but we shall not need such an extension.} let $f\in L^2(\sO,\fw)$, and suppose that $u\in H^1(\sO,\fw)$ satisfies the variational equation \eqref{eq:HestonVariationalEquation}. If\footnote{While the right-hand side of the identity \eqref{eq:IntroHestonWeakMixedProblemHomogeneous_Dx} is well-defined when $f_x \in L^2(\sO,\fw)$, we appeal to an approximation argument requiring $f\in W^{1,2}(\sO,\fw)$.}
$$
f\in W^{1,2}(\sO,\fw), \quad u \in H^2(\sO,\fw), \quad\hbox{and}\quad u_x \in H^1(\sO,\fw),
$$
then $u_x \in H^1(\sO,\fw)$ obeys
\begin{equation}
\label{eq:IntroHestonWeakMixedProblemHomogeneous_Dx}
\fa(u_x,v) = (f_x,v)_{L^2(\sO,\fw)},
\end{equation}
for all $v\in H^1_0(\underline{\sO},\fw)$.
\end{lem}

\begin{proof}
Suppose first that $u \in C^\infty(\bar\sO)$ and $v \in C^\infty_0(\underline{\sO})$. Then $v_x \in C^\infty_0(\underline{\sO})$ and
\begin{align*}
\fa(u,-v_x) &= (Au,-v_x)_{L^2(\sO,\fw)} \quad\hbox{(by Lemma \ref{lem:HestonIntegrationByParts})}
\\
&= (Au_x,v)_{L^2(\sO,\fw)} + (Au,v(\log\fw)_x)_{L^2(\sO,\fw)}
\\
&= \fa(u_x,v) + (Au,v(\log\fw)_x)_{L^2(\sO,\fw)},
\end{align*}
where from \eqref{eq:HestonWeight} we see that
$$
(\log \fw)_x = -\gamma\frac{x}{\sqrt{1+x^2}} \quad\hbox{on }\HH.
$$
Now suppose, more generally, that $u\in H^2(\sO,\fw)$ with $u_x\in H^1(\sO,\fw)$, as in our hypotheses. For $v\in C^\infty_0(\underline\sO)$, we may choose a subdomain $\sO'\Subset\underline\sO$ such that $\supp v \subset\underline\sO'$ and $\partial_1\sO'$ is $C^1$-orthogonal to $\partial\HH$ in the sense of Definition \ref{defn:C1Orthogonal}. According to Theorem \ref{thm:KufnerPowerWeightBoundedDerivatives}, there is a sequence $\{u_n\}_{n\in \NN} \subset C^\infty(\bar\sO')$ such that $u_n \to u$ in $H^2(\sO',\fw)$ as $n\to\infty$ and hence, for each $v \in C^\infty_0(\underline{\sO})$
with $\supp v \subset\underline\sO'$,
$$
(Au_n, v)_{L^2(\sO,\fw)} \to (Au, v)_{L^2(\sO,\fw)} \quad\hbox{and}\quad \fa(u_{n,x},v) \to \fa(u_x,v), \quad n\to\infty,
$$
since, in the second case,  by \eqref{eq:HestonWithKillingBilinearForm} we see that
\begin{align*}
\left|\fa(u_{n,x} - u_x,v)\right| &\leq C\left(\|yD(u_{n,x}-u_x)\|_{L^2(\sO,\fw)} + \|u_{n,x}-u_x\|_{L^2(\sO,\fw)}\right)\|v\|_{W^{1,2}(\sO,\fw)}
\\
&\leq C\|u_n - u\|_{H^2(\sO,\fw)}\|v\|_{W^{1,2}(\sO,\fw)},
\end{align*}
where $C=C(\height(\sO))$. Therefore, by approximation and also the fact that $v\in C^\infty_0(\underline{\sO})$ is arbitrary,
the preceding variational equation continues to hold for $u\in H^2(\sO,\fw)$ with $u_x\in H^1(\sO,\fw)$, that is,
$$
\fa(u,-v_x) = \fa(u_x,v) + (Au,v(\log\fw)_x)_{L^2(\sO,\fw)}, \quad\forall v\in C^\infty_0(\underline{\sO}).
$$
Since $u\in H^2(\sO,\fw)$, then \eqref{eq:HestonVariationalEquation} implies that $Au = f$ a.e. on $\sO$ by Lemma \ref{lem:HestonIntegrationByParts},
and thus the preceding identity gives
$$
\fa(u,-v_x) = \fa(u_x,v) + (f,v(\log\fw)_x)_{L^2(\sO,\fw)}, \quad\forall v\in C^\infty_0(\underline{\sO}).
$$
Moreover, by substituting $-v_x$ for $v$ in \eqref{eq:HestonVariationalEquation}, using the fact that $f\in W^{1,2}(\sO,\fw)$ by hypothesis, so $f_x\in L^2(\sO,\fw)$, and appealing to Theorem \ref{thm:KufnerPowerWeightBoundedDerivatives} to choose
a sequence $\{f_n\}_{n\in \NN} \subset C^\infty(\bar\sO)$ such that $f_n \to f$ in $W^{1,2}(\sO,\fw)$ and so $f_{n,x} \to f_x$ in $L^2(\sO,\fw)$ as $n\to\infty$, we obtain
\begin{align*}
\fa(u,-v_x) &= (f,-v_x)_{L^2(\sO,\fw)}
\\
&= (f_x,v)_{L^2(\sO,\fw)} + (f,v(\log\fw)_x)_{L^2(\sO,\fw)}, \quad\forall v\in C^\infty_0(\underline{\sO}),
\end{align*}
where the integration-by-parts identity is justified by approximation, just as in the proof of Lemma \ref{lem:HestonIntegrationByParts}.
Combining these identities yields \eqref{eq:IntroHestonWeakMixedProblemHomogeneous_Dx} for all $v\in C^\infty_0(\underline{\sO})$, and hence for all $v\in H^1_0(\underline{\sO},\fw)$.
\end{proof}

\begin{rmk}[Need for the regularity condition on $u$ in Lemma \ref{lem:VarEqn_Dx}]
If we only knew that $u\in H^2(\sO,\fw)$, then the definition \eqref{eq:H2WeightedSobolevSpace} of $H^2(\sO,\fw)$ would imply that $y|D^2u|, \ (1+y)|Du| \in L^2(\sO,\fw)$ and so $y^{1/2}|Du_x|, \ (1+y)^{1/2}u_x \in L^2(\sO,y\fw) = L^2(\sO,\fw_1)$ and thus $u_x \in H^1(\sO,\fw_1)$, \emph{but not necessarily $H^1(\sO,\fw)$}, by the definition \eqref{eq:H1WeightedSobolevSpace} of $H^1(\sO,\fw)$.
\end{rmk}

\begin{lem}[Variational equation for higher-order derivatives of a solution with respect to $x$]
\label{lem:VarEqn_Dkx}
Let $\sO\subseteqq\HH$ be a domain with \emph{finite height}, let $k\geq 1$ be an integer, let $f\in L^2(\sO,\fw)$, and suppose that $u\in H^1(\sO,\fw)$ satisfies the variational equation \eqref{eq:HestonVariationalEquation}. If
$$
D_x^k f \in L^2(\sO,\fw), \quad u \in \sH^{k+1}(\sO,\fw), \quad\hbox{and}\quad D_x^k u \in H^1(\sO,\fw),
$$
then $D_x^k u \in H^1(\sO,\fw)$ obeys
\begin{equation}
\label{eq:IntroHestonWeakMixedProblemHomogeneous_Dkx}
\fa(D_x^k u, v) = (D_x^k f, v)_{L^2(\sO,\fw)}, \quad v\in H^1_0(\underline{\sO},\fw).
\end{equation}
\end{lem}

\begin{proof}
By hypothesis,
$$
f_x \in L^2(\sO,\fw), \quad u \in H^2(\sO,\fw), \quad\hbox{and}\quad u_x \in H^1(\sO,\fw),
$$
and so Lemma \ref{lem:VarEqn_Dx} implies that $u_x$ obeys
$$
\fa(u_x,v) = (f_x,v)_{L^2(\sO,\fw)}, \quad \forall v \in H^1_0(\underline\sO,\fw).
$$
By induction we may assume that the conclusion holds when $k$ is replaced by $k-1$. Note that $u \in \sH^{k+1}(\sO,\fw)$ by hypothesis and so, by Definition \ref{defn:HkWeightedSobolevSpaceNormPowery}, we see that $u_x$ obeys
\begin{align*}
yD_x^k u_x, \ yD_x^{k-1}D_y u_x, \ yD_x^{k-2}D_y^2 u_x &\in L^2(\sO,\fw),
\\
yD_x^{k-m}D_y^m u_x &\in L^2(\sO,\fw_{m-2}), \quad 3\leq m\leq k.
\end{align*}
Therefore, $u \in \sH^{k+1}(\sO,\fw)$ implies that $u_x \in \sH^k(\sO,\fw)$ when $k\geq 2$. Since
$$
D_x^{k-1} f_x = D_x^k f  \in L^2(\sO,\fw), \quad u_x \in \sH^k(\sO,\fw), \quad\hbox{and}\quad D_x^{k-1} u_x = D_x^k u \in H^1(\sO,\fw),
$$
we can apply Lemma \ref{lem:VarEqn_Dkx} to the preceding variational equation, with $k-1$ and $u_x \in H^1(\sO,\fw)$ and $f_x \in L^2(\sO,\fw)$ replacing $k$ and $u \in H^1(\sO,\fw)$ and $f \in L^2(\sO,\fw)$, respectively, to give
$$
\fa(D_x^k u,v) = \fa(D_x^{k-1} u_x, v) = (D_x^{k-1} f_x, v)_{L^2(\sO,\fw)} = (D_x^k f ,v)_{L^2(\sO,\fw)}, 
$$
for all $v \in H^1_0(\underline\sO,\fw)$. This completes the proof.
\end{proof}

In order to establish a refinement of Lemma \ref{lem:VarEqn_Dx} which yields $u_x \in H^1(\sO,\fw)$ as a conclusion, assuming only $u\in H^2(\sO,\fw)$, we shall need to substitute $v \in C^\infty_0(\underline{\sO}')$, where $\sO'\Subset\underline{\sO}$, by a finite difference, $-\delta_x^{-h}v$, rather than $-v_x$, by analogy with the proof of \cite[Theorem 8.8]{GilbargTrudinger}.

\begin{prop}[Variational equation for the derivative of a solution with respect to $x$]
\label{prop:VarEqn_H1Dx}
Let $\sO\subseteqq\HH$ be a domain and let $d_1, \Upsilon$ be positive constants. Then there is a positive constant, $C=C(\Lambda,\nu_0,d_1,\Upsilon)$, such that the following holds. Let $f\in L^2(\sO,\fw)$ and suppose that $u\in H^1(\sO,\fw)$ satisfies the variational equation \eqref{eq:HestonVariationalEquation}. If
$
f_x \in L^2(\sO,\fw),
$
then $u_x \in H^1_{\loc}(\underline\sO,\fw)$ and, for any subdomain $\sO'\subset\sO$ with $\bar\sO'\subset\underline\sO$ and $\dist(\partial_1\sO',\partial_1\sO)\geq d_1$ and $\height(\sO')\leq \Upsilon$, one has
$
u_x \in H^1(\sO',\fw),
$
and
\begin{equation}
\label{eq:IntroHestonWeakMixedProblemHomogeneous_DxSubdomain}
\fa(u_x,v) = (f_x,v)_{L^2(\sO',\fw)}, \quad\forall v\in H^1_0(\underline{\sO}',\fw),
\end{equation}
and
$$
\|u_x\|_{H^1(\sO',\fw)} \leq C\left(\|f_x\|_{L^2(\sO,\fw)} + \|f\|_{L^2(\sO,\fw)} + \|u\|_{L^2(\sO,\fw)}\right).
$$
\end{prop}

\begin{proof}
We partially follow the idea of the proof of \cite[Theorem 8.8]{GilbargTrudinger}, but the argument is simpler here because of the relatively strong hypothesis that $f_x\in L^2(\sO,\fw)$, as well as $f\in L^2(\sO,\fw)$).

Choose a subdomain $\sO''\subset\sO$ such that $\bar\sO'\subset\underline\sO''$ and $\bar\sO''\subset\underline\sO$ and $\partial_1\sO''$ is $C^1$-orthogonal to $\partial\HH$, while $\dist(\partial_1\sO',\partial_1\sO'')\geq d_1/4$ and $\dist(\partial_1\sO'',\partial_1\sO)\geq d_1/2$ and $\height(\sO'')\leq 2\Upsilon$. Observe that if $(x,y)\in \sO''$, then $(x\pm h,y)\in\sO$ provided $\dist((x,y), \partial_1\sO)<|h|$, so in choosing $h$, we shall always assume that $0<|h|<\frac{1}{2}\dist(\sO'',\partial_1\sO)$. For any $v\in C^\infty_0(\underline{\sO}'')$, observe that $\delta_x^{-h}v \in C^\infty_0(\underline{\sO})$, so we may substitute $-\delta_x^{-h}v$ for $v$ as a test function in \eqref{eq:HestonVariationalEquation}.

For any $w\in C^\infty(\bar\sO)$, noting that $\delta_x^h Aw = A\delta_x^h w$ on $\sO''$, we obtain
\begin{align*}
-\fa(w, \delta_x^{-h}v) &= -(Aw, \delta_x^{-h}v)_{L^2(\sO,\fw)} \quad\hbox{(by Lemma \ref{lem:HestonIntegrationByParts})}
\\
&= (\delta_x^{h}Aw, (\fw^h/\fw)v)_{L^2(\sO,\fw)} + (Aw, (\delta_x^{h}\fw/\fw)v)_{L^2(\sO,\fw)} \quad\hbox{(by \eqref{eq:FDIntegrationByParts})}
\\
&= (A\delta_x^{h}w, (\fw^h/\fw)v)_{L^2(\sO,\fw)} + (Aw, (\delta_x^{h}\fw/\fw)v)_{L^2(\sO,\fw)}
\\
&= \fa(\delta_x^{h}w, (\fw^h/\fw)v) + \fa(w, (\delta_x^{h}\fw/\fw)v),
\quad\forall v\in C^\infty_0(\underline{\sO}''),
\end{align*}
where, in the last equality, we use the fact that $(\fw^h/\fw)v \in C^\infty_0(\underline{\sO}'')$ when $v\in C^\infty_0(\underline{\sO}'')$, recalling by \eqref{eq:HestonWeight} that
$
\fw(x,y) = y^{\beta-1}e^{-\gamma\sqrt{1+x^2} - \mu y}, 
$
for all $(x,y) \in \HH$.
Recall that $\supp v \subset\underline\sO''$. Since $C^\infty(\bar\sO'')$ is dense in $H^1(\sO'',\fw)$ by Theorem \ref{thm:KufnerPowerWeightBoundedDerivatives},
we may choose $\{w_n\}_{n\in\NN}\subset C^\infty(\bar\sO'')$ with $w_n\to u$ strongly in $H^1(\sO'',\fw)$ to see that
$$
-\fa(u, \delta_x^{-h}v) = \fa(\delta_x^{h}u, (\fw^h/\fw)v) + \fa(u, (\delta_x^{h}\fw/\fw)v),
\quad\forall v\in C^\infty_0(\underline{\sO}'').
$$
Therefore, for all $v\in C^\infty_0(\underline{\sO}'')$, the preceding identity yields
\begin{align*}
\fa(\delta_x^{h}u, (\fw^h/\fw)v) &= -\fa(u, \delta_x^{-h}v) - \fa(u, (\delta_x^{h}\fw/\fw)v)
\\
&= -(f, \delta_x^{-h}v)_{L^2(\sO,\fw)} - (f, (\delta_x^{h}\fw/\fw)v)_{L^2(\sO,\fw)} \quad\hbox{(by \eqref{eq:HestonVariationalEquation})}
\\
&= (\delta_x^{h}f, (\fw^h/\fw)v)_{L^2(\sO,\fw)} \quad\hbox{(by \eqref{eq:FDIntegrationByParts})},
\end{align*}
and consequently,
$
\fa(\delta_x^{h}u, v) = (\delta_x^{h}f, v)_{L^2(\sO'',\fw)}, 
$
for all $v\in C^\infty_0(\underline{\sO}'')$.
Since $C^\infty_0(\underline{\sO}'')$ is dense in $H^1_0(\underline{\sO}'',\fw)$ by definition, we obtain
\begin{equation}
\label{eq:BilinearFormFDTestFunction}
\fa(\delta_x^{h}u, v) = (\delta_x^{h}f, v)_{L^2(\sO'',\fw)}, \quad\forall v\in H^1_0(\underline{\sO}'',\fw).
\end{equation}
The interior a priori estimate \eqref{eq:H1BoundSolutionHestonVarEqnSubdomain} for solutions to the preceding equation yields
$$
\|\delta_x^{h}u\|_{H^1(\sO',\fw)} \leq C\left(\|\delta_x^{h}f\|_{L^2(\sO'',\fw)} + \|\delta_x^{h}u\|_{L^2(\sO'',\fw)}\right), 
$$
for $0<2|h|<\dist(\sO'',\partial_1\sO)$,
where $C=C(\Lambda,\nu_0,\Upsilon)$ is a positive constant. Choose a subdomain $\sO'''\subset\sO$ such that $\bar\sO''\subset\underline\sO'''$ and $\bar\sO'''\subset\underline\sO$, while $\dist(\partial_1\sO''',\partial_1\sO)\geq d_1/4$ and $\dist(\partial_1\sO'',\partial_1\sO''')\geq d_1/8$ and $\height(\sO'')\leq 4\Upsilon$. By Lemma \ref{lem:FDBounds} \eqref{item:FDBounds_DiffBound} and the facts that $f_x\in L^2(\sO,\fw)$ by hypothesis and $u_x\in L^2(\sO''',\fw)$ by Proposition \ref{prop:InteriorKochEstimate}, we see that
$$
\|\delta_x^{h}f\|_{L^2(\sO'',\fw)} + \|\delta_x^{h}u\|_{L^2(\sO'',\fw)} \leq \|f_x\|_{L^2(\sO''',\fw)} + \|u_x\|_{L^2(\sO''',\fw)},
$$
for $0<2|h|<\dist(\sO'',\partial_1\sO''')$,
and so
$$
\|\delta_x^{h}u\|_{H^1(\sO',\fw)} \leq C\left(\|f_x\|_{L^2(\sO''',\fw)} + \|u_x\|_{L^2(\sO''',\fw)}\right),
$$
for $0<2|h|<\dist(\sO'',\partial_1\sO''')$.
Therefore, since $(\delta_x^{h}u)_x = \delta_x^{h}u_x$ and $(\delta_x^{h}u)_y = \delta_x^{h}u_y$, we have
\begin{align*}
\|y^{1/2}\delta_x^{h}u_x\|_{L^2(\sO',\fw)} &\leq C_1,
\\
\|y^{1/2}\delta_x^{h}u_y\|_{L^2(\sO',\fw)} &\leq C_1,
\\
\|\delta_x^{h}u\|_{L^2(\sO',\fw)} &\leq C_1, \quad\hbox{for } 0<2|h|<\dist(\sO',\partial_1\sO''),
\end{align*}
where $C_1 := C(\|f_x\|_{L^2(\sO''',\fw)} + \|u_x\|_{L^2(\sO''',\fw)})$. Lemma \ref{lem:FDBounds} \eqref{item:FDBounds_WeakDerivExist} (and its proof) gives $u_x\in H^1(\sO',\fw)$ and weak convergence, after passing to a diagonal subsequence,
$$
y^{1/2}\delta_x^h u_x \rightharpoonup y^{1/2}u_{xx}, \quad y^{1/2}\delta_x^h u_y \rightharpoonup y^{1/2}u_{xy}, \quad \delta_x^h u \rightharpoonup u_x
\quad\hbox{weakly in $L^2(\sO',\fw)$ as $h\to 0$,}
$$
and thus,
$$
\delta_x^h u \rightharpoonup u_x \quad\hbox{weakly in $H^1(\sO',\fw)$ as $h\to 0$.}
$$
Therefore, we have that
$$
\|u_x\|_{H^1(\sO',\fw)} \leq \liminf_{h\to 0}\|\delta_x^{h}u\|_{H^1(\sO',\fw)}.
$$
and, by combining the preceding inequalities, we obtain
$$
\|u_x\|_{H^1(\sO',\fw)} \leq C\left(\|f_x\|_{L^2(\sO''',\fw)} + \|u_x\|_{L^2(\sO''',\fw)}\right).
$$
Finally, Proposition \ref{prop:InteriorKochEstimate} yields
$$
\|u_x\|_{L^2(\sO''',\fw)} \leq C\left(\|f\|_{L^2(\sO,\fw)} + \|u\|_{L^2(\sO,\fw)}\right),
$$
for $C=C(\Lambda,\nu_0,d_1,\Upsilon)$ and the conclusion follows from the preceding two estimates.
\end{proof}

Clearly, by repeatedly applying Proposition \ref{prop:VarEqn_H1Dx}, induction on $k\geq 1$ yields the following refinement of Lemma \ref{lem:VarEqn_Dkx}.

\begin{prop}[Variational equation for higher-order derivatives of a solution with respect to $x$]
\label{prop:VarEqn_H1Dkx}
Let $\sO\subseteqq\HH$ be a domain, let $d_1, \Upsilon$ be positive constants, and let $k\geq 1$ be an integer. Then there is a positive constant, $C=C(\Lambda,\nu_0,d_1,k,\Upsilon)$,
such that the following holds. Let $f\in L^2(\sO,\fw)$ and suppose that $u\in H^1(\sO,\fw)$ satisfies the variational equation \eqref{eq:HestonVariationalEquation}. If
$
D_x^j f \in L^2(\sO,\fw), 
$
for $1\leq j\leq k$,
then $D_x^k u \in H^1_{\loc}(\underline\sO,\fw)$ and, for any subdomain $\sO'\subset\sO$ with $\bar\sO'\subset\underline\sO$ and $\dist(\partial_1\sO',\partial_1\sO)\geq d_1$ and $\height(\sO')\leq \Upsilon$, one has
$
D_x^k u \in H^1(\sO',\fw),
$
and
$$
\fa(D_x^k u,v) = (D_x^k f,v)_{L^2(\sO',\fw)}, \quad\forall v\in H^1_0(\underline{\sO}',\fw),
$$
and
$$
\|D_x^k u\|_{H^1(\sO',\fw)} \leq C\left(\sum_{j=0}^k\|D_x^j f\|_{L^2(\sO,\fw)} + \|u\|_{L^2(\sO,\fw)}\right).
$$
\end{prop}

\begin{proof}
Proposition \ref{prop:VarEqn_H1Dx} yields the conclusion when $k=1$ and so we can take $k\geq 2$ and assume, by induction, that the result holds for $k-1$ in place of $k$. Choose a subdomain $\sO''\subset\sO$ with $\bar\sO'\subset\underline\sO''$ and $\bar\sO''\subset\underline\sO$, while $\dist(\partial_1\sO',\partial_1\sO'')\geq d_1/4$ and $\dist(\partial_1\sO'',\partial_1\sO)\geq d_1/2$ and $\height(\sO'')\leq 2\Upsilon$. By the induction hypothesis, $D_x^{k-1} u \in H^1_{\loc}(\underline\sO,\fw)\cap H^1(\sO'',\fw)$ and $D_x^{k-1} u$ obeys
$$
\fa(D_x^{k-1} u,v) = (D_x^{k-1} f,v)_{L^2(\sO'',\fw)}, \quad\forall v\in H^1_0(\underline{\sO}'',\fw).
$$
Hence, by applying Proposition \ref{prop:VarEqn_H1Dx} to the preceding variational equation in place of \eqref{eq:HestonVariationalEquation}, we see that $D_x^k u \in H^1(\sO',\fw)$ and, because the choice of subdomain $\sO''\subset\sO$ with $\bar\sO''\subset\underline\sO$ was arbitrary, that also $D_x^k u \in H^1_{\loc}(\underline\sO,\fw)$. Moreover, Proposition \ref{prop:VarEqn_H1Dx} yields
$$
\|D_x^k u\|_{H^1(\sO',\fw)} \leq C\left(\|D_x^k f\|_{L^2(\sO'',\fw)} + \|D_x^{k-1} f\|_{L^2(\sO'',\fw)} + \|D_x^{k-1} u\|_{L^2(\sO'',\fw)}\right),
$$
for a positive constant, $C=C(\Lambda,\nu_0,d_1,\Upsilon)$, while the induction hypothesis gives
$$
\|D_x^{k-1} u\|_{H^1(\sO'',\fw)} \leq C\left(\sum_{j=0}^{k-1}\|D_x^j f\|_{L^2(\sO,\fw)} + \|u\|_{L^2(\sO,\fw)}\right),
$$
for $C=C(\Lambda,\nu_0,d_1,k,\Upsilon)$. We obtain the conclusion by combining the preceding estimates.
\end{proof}

As the regularity questions of interest to us only concern regularity of a solution to the variational equation \eqref{eq:HestonVariationalEquation}, it will be convenient to consider, for $z_0\in\partial\HH$ and $0<R<R_1<R_0$, half-balls $U\subset U' \subset V$, where
\begin{equation}
\label{eq:UVHalfBalls}
U := B_R^+(z_0), \quad U' := B_{R_1}^+(z_0), \quad\hbox{and}\quad V = B_{R_0}^+(z_0),
\end{equation}
and we recall that $B_R^+(z_0) := B_R(z_0)\cap\sO$, for any $R>0$ and $z_0\in\RR^2$. Note that $U\Subset \underline{U}'$ and $U'\Subset \underline{V}$.

\begin{prop}[Interior $H^2$ regularity for higher-order derivatives of a solution with respect to $x$]
\label{prop:SobolevRegularity_ukx}
Let $R<R_0$ be positive constants and let $k\geq 1$ be an integer. Then there is a positive constant, $C=C(\Lambda,\nu_0,k,R,R_0)$, such that the following holds. Let $\sO\subseteqq\HH$ be a domain, let $f\in L^2(\sO,\fw)$, and suppose that $u \in H^1(\sO,\fw)$ is a solution to the variational equation \eqref{eq:HestonVariationalEquation}. If $U\subset V$ are as in \eqref{eq:UVHalfBalls} with $V\Subset\underline{\sO}$, and
$
D_x^j f \in L^2(V,\fw),
$
for all $1\leq j\leq k$, then
$
D_x^k u \in H^2(U,\fw),
$
and
\begin{equation}
\label{eq:SobolevRegularity_ukx}
\|D_x^k u\|_{H^2(U,\fw)} \leq C\left(\sum_{j=0}^k\|D_x^j f\|_{L^2(V,\fw)} + \|u\|_{L^2(V,\fw)}\right).
\end{equation}
\end{prop}

\begin{proof}
Choose an auxiliary half-ball, $U'$ as in \eqref{eq:UVHalfBalls}, with $U'=B^+_{R_1}(z_0)$ and $U\subset U'\subset V$, and fix $R_1=(R+R_0)/2$. Since $D_x^j f \in L^2(V,\fw)$, $1\leq j\leq k$ by hypothesis, we can apply Proposition \ref{prop:VarEqn_H1Dkx} to give $D_x^k u\in H^1(U',\fw)$ and
$$
\fa(D_x^k u, v) = (D_x^k f, v)_{L^2(U',\fw)}, \quad\forall v\in H^1_0(\underline U',\fw).
$$
We can now apply Theorem \ref{thm:H2BoundSolutionHestonVarEqnSubdomainInterior} to the preceding variational equation to give $D_x^k u \in H^2(U,\fw)$ and
$$
\|D_x^k u\|_{H^2(U,\fw)} \leq C\left(\|D_x^k f\|_{L^2(U',\fw)} + \|D_x^k u\|_{L^2(U',\fw)}\right),
$$
where $C=C(\Lambda,\nu_0,R,R_1)$ is a positive constant. But
$$
\|D_x^k u\|_{L^2(U',\fw)} \leq \|D_x^k u\|_{H^1(U',\fw)},
$$
and by Proposition \ref{prop:VarEqn_H1Dkx}, we obtain
$$
\|D_x^k u\|_{H^1(U',\fw)} \leq C\left(\sum_{j=0}^k\|D_x^j f\|_{L^2(V,\fw)} + \|u\|_{L^2(V,\fw)}\right),
$$
where $C=C(\Lambda,\nu_0,k,R_1,R_0)$ is a positive constant. Combining the preceding estimates completes the proof.
\end{proof}

\subsection{Interior $H^2$ regularity for first-order derivatives orthogonal to the degenerate boundary}
\label{subsec:InteriorH2DyRegularity}
We have the following analogue of Lemma \ref{lem:VarEqn_Dx}. Observe that if $u\in H^2(\sO,\fw)$, then the definition \eqref{eq:H2WeightedSobolevSpace} of $H^2(\sO,\fw)$ implies that $y|D^2u|, \ (1+y)|Du| \in L^2(\sO,\fw)$ and so $y^{1/2}|Du_y|, \ (1+y)^{1/2}u_y \in L^2(\sO,y\fw) = L^2(\sO,\fw_1)$ and thus $u_y \in H^1(\sO,\fw_1)$ by the definition \eqref{eq:H1WeightedSobolevSpace} of $H^1(\sO,\fw)$.

\begin{lem}[Variational equation for the derivative of a solution with respect to $y$]
\label{lem:VarEqn_Dy}
Let $\sO\subseteqq\HH$ be a domain, let $f\in L^2(\sO,\fw)$, and suppose that $u\in H^1(\sO,\fw)$ satisfies the variational equation \eqref{eq:HestonVariationalEquation}. If\footnote{While the right-hand side of the identity \eqref{eq:IntroHestonWeakMixedProblemHomogeneous_Dy} is well-defined when $f_y \in L^2(\sO,\fw_1)$, we appeal to an approximation argument requiring at least $f\in H^1(\sO,\fw)$ to justify integration by parts involving $f$.}
$$
f \in H^1(\sO,\fw), \quad u \in H^2(\sO,\fw), \quad\hbox{and}\quad u_{xx} \in L^2(\sO,\fw_1),
$$
then $u_y$ obeys
\begin{equation}
\label{eq:IntroHestonWeakMixedProblemHomogeneous_Dy}
\fa_1(u_y,v) = (f_y,v)_{L^2(\sO,\fw_1)}  - (Bu,v)_{L^2(\sO,\fw_1)},
\end{equation}
for all $v\in H^1_0(\underline{\sO},\fw_1)$.
\end{lem}

\begin{proof}
Again, suppose first that $u \in C^\infty(\bar\sO)$ and $v \in C^\infty_0(\underline{\sO})$. Then $(yv)_y \in C^\infty_0(\underline{\sO})$ too and
\begin{align*}
\fa(u,(yv)_y) &= (Au,(yv)_y)_{L^2(\sO,\fw)} \quad\hbox{(by Lemma \ref{lem:HestonIntegrationByParts})}
\\
&= -((Au)_y,yv)_{L^2(\sO,\fw)}  - (Au,yv(\log\fw)_y)_{L^2(\sO,\fw)}
\\
&= -(A_1u_y,yv)_{L^2(\sO,\fw)} -(Bu,yv)_{L^2(\sO,\fw)} - (Au,yv(\log\fw)_y)_{L^2(\sO,\fw)}
\quad\hbox{(by \eqref{eq:RefinedCommutator})}
\\
&= -(A_1u_y,v)_{L^2(\sO,\fw_1)} - (Bu,v)_{L^2(\sO,\fw_1)} - (Au,yv(\log\fw)_y)_{L^2(\sO,\fw)}
\quad\hbox{(by \eqref{eq:HestonWeightn})}
\\
&= -\fa_1(u_y,v) - (Bu,v)_{L^2(\sO,\fw_1)} - (Au,yv(\log\fw)_y)_{L^2(\sO,\fw)}
\quad\hbox{(by \eqref{eq:HestonWithKillingBilinearForm_powery})},
\end{align*}
where from \eqref{eq:HestonWeight} we see that
$$
(\log \fw)_y = (\beta-1)y^{-1} - \mu \quad\hbox{on }\HH.
$$
As in the proof of Lemma \ref{lem:VarEqn_Dx}, for $v\in C^\infty_0(\underline\sO)$, we may choose a subdomain $\sO'\Subset\underline\sO$ such that $\supp v \subset\underline\sO'$ and $\partial_1\sO'$ is $C^1$-orthogonal to $\partial\HH$.
If we now assume only that $u \in H^2(\sO,\fw)$ and $u_{xx} \in L^2(\sO,\fw_1)$, as in our hypotheses, there is a sequence, $\{u_n\}_{n\in \NN} \subset C^\infty(\bar\sO')$, such that $u_n \to u$ in $H^2(\sO',\fw)$ as $n\to\infty$ by Theorem \ref{thm:KufnerPowerWeightBoundedDerivatives}.
But then $u_{n,xx}\rightharpoonup u_{xx}$ weakly in $L^2(\sO',\fw_1)$ as $n\to\infty$ since, for $v\in C^\infty_0(\underline{\sO})$ with $\supp\subset\underline\sO'$ and all $n\in\NN$,
\begin{align*}
\left|(u_{n,xx} - u_{xx}, v)_{L^2(\sO,\fw_1)}\right| &= \left|(yu_{n,xx} - yu_{xx}, v)_{L^2(\sO,\fw)}\right|
\\
&\leq \|y(u_{n,xx} - u_{xx})\|_{L^2(\sO,\fw)}\|v\|_{L^2(\sO,\fw)}
\\
&\leq \|u_n - u\|_{H^2(\sO,\fw)}\|v\|_{L^2(\sO,\fw)}.
\end{align*}
Therefore, by approximation, the variational identity continues to hold for $u\in H^2(\sO,\fw)$, which ensures $u_x \in L^2(\sO,\fw)\subset L^2(\sO,\fw_1)$, and $u_{xx} \in L^2(\sO,\fw_1)$ (thus $Bu \in L^2(\sO,\fw_1)$), that is,
\begin{align*}
\fa(u,(yv)_y) = -\fa_1(u_y,v) - (Bu,v)_{L^2(\sO,\fw_1)}
- (Au,yv(\log\fw)_y)_{L^2(\sO,\fw)}, 
\end{align*}
for all $v\in C^\infty_0(\underline{\sO})$.
Also, since $u\in H^2(\sO,\fw)$, then \eqref{eq:HestonVariationalEquation} implies that $Au = f$ a.e. on $\sO$ by Lemma \ref{lem:HestonIntegrationByParts}. Hence, \eqref{eq:HestonVariationalEquation} and the fact that $f\in H^1(\sO,\fw)$, and thus $f_y\in L^2(\sO,\fw_1)$ by hypothesis, yields
\begin{align*}
\fa(u,(yv)_y) &= (f,(yv)_y)_{L^2(\sO,\fw)}
\\
&= -(f_y,yv)_{L^2(\sO,\fw)} - (f,yv(\log\fw)_y)_{L^2(\sO,\fw)}
\\
&= -(f_y,v)_{L^2(\sO,\fw_1)} - (f,yv(\log\fw)_y)_{L^2(\sO,\fw)}, \quad\forall v\in C^\infty_0(\underline{\sO}),
\end{align*}
while the preceding variational identity gives
\begin{align*}
\fa(u,(yv)_y) = -\fa_1(u_y,v) - (Bu,v)_{L^2(\sO,\fw_1)}
-(f,yv(\log\fw)_y)_{L^2(\sO,\fw)}, 
\end{align*}
for all $v\in C^\infty_0(\underline{\sO})$.
Combining these variational identities yields \eqref{eq:IntroHestonWeakMixedProblemHomogeneous_Dy}, for all $v\in C^\infty_0(\underline{\sO})$, and hence the variational identity holds for all $v\in H^1_0(\underline{\sO},\fw)$.
\end{proof}

\begin{prop}[Interior $H^2$ regularity for a derivative of a solution with respect to $y$]
\label{prop:SobolevRegularity_uy}
Let $R<R_0$ be positive constants. Then there is a positive constant, $C=C(\Lambda,\nu_0,R,R_0)$, such that the following holds. Let $\sO\subseteqq\HH$ be a domain and let $U\subset V$ be as in \eqref{eq:UVHalfBalls}, with $V\Subset\underline{\sO}$. Suppose that $f\in L^2(\sO,\fw)$ and $u \in H^1(\sO,\fw)$ is a solution to the variational equation \eqref{eq:HestonVariationalEquation}. If
$$
f \in W^{1,2}(V,\fw), \quad u \in H^2(V,\fw), \quad\hbox{and}\quad u_x \in H^1(V,\fw),
$$
then $u_y \in H^2(U,\fw_1)$ and
\begin{equation}
\label{eq:SobolevRegularity_uy}
\|u_y\|_{H^2(U,\fw_1)} \leq C\left(\|f\|_{W^{1,2}(V,\fw)} + \|u\|_{L^2(V,\fw)}\right).
\end{equation}
\end{prop}

\begin{proof}
The argument is similar to the proof of Proposition \ref{prop:SobolevRegularity_ukx}, except that the appeal to Proposition \ref{prop:VarEqn_H1Dx} is replaced by an appeal to Lemma \ref{lem:VarEqn_Dy} and we need to keep track of the different Sobolev weights which now arise. Notice that $u \in H^1(\sO,\fw)$ by hypothesis, and so $y^{1/2}u_x \in L^2(V,\fw)$ or equivalently $u_x \in L^2(V,\fw_1)$. Moreover, $f \in H^1(V,\fw)$, since $f \in W^{1,2}(V,\fw)$ by hypothesis. Also, $u_x \in H^1(V,\fw)$ by hypothesis, and so $y^{1/2}u_{xx} \in L^2(V,\fw)$ or, equivalently, $u_{xx} \in L^2(V,\fw_1)$. Finally, the hypothesis $u \in H^2(\sO,\fw)$ implies $u_y \in H^1(V,\fw_1)$. Therefore, Lemma \ref{lem:VarEqn_Dy}, with $V$ in place of $\sO$, gives
$$
\fa_1(u_y,v) = (f_y,v)_{L^2(V,\fw_1)}  - (Bu,v)_{L^2(V,\fw_1)}, \quad \forall v\in H^1_0(\underline{V},\fw_1).
$$
Choose an auxiliary half-ball, $U'$ as in \eqref{eq:UVHalfBalls}, with $U'=B^+_{R_1}(z_0)$ and $U\subset U'\subset V$, and fix $R_1=(R+R_0)/2$. We can apply Theorem \ref{thm:H2BoundSolutionHestonVarEqnSubdomainInterior} to the preceding equation in place of \eqref{eq:HestonVariationalEquation} to deduce that $u_y \in H^2(U,\fw_1)$ and
\begin{align*}
\|u_y\|_{H^2(U,\fw_1)} &\leq C\left(\|f_y-Bu\|_{L^2(U',\fw_1)} + \|u_y\|_{L^2(U',\fw_1)}\right)
\\
&\leq C\left(\|f_y\|_{L^2(U',\fw_1)} + \|u_x\|_{L^2(U',\fw_1)} + \|u_{xx}\|_{L^2(U',\fw_1)} + \|u_y\|_{L^2(U',\fw_1)}\right),
\end{align*}
where $C=C(\Lambda,\nu_0,R,R_1)$ is a positive constant. But
$$
\|Du\|_{L^2(U',\fw_1)} \leq \|u\|_{H^2(U',\fw_1)} \leq C\|u\|_{H^2(U',\fw)},
$$
where the first inequality follows from \eqref{eq:H2WeightedSobolevSpace} and the second from \eqref{eq:SobolevInclusionDifferentWeights}, with $C=C(R_1)$. By Theorem \ref{thm:H2BoundSolutionHestonVarEqnSubdomainInterior}, since $u$ obeys \eqref{eq:HestonVariationalEquation}, we obtain
$$
\|u\|_{H^2(U',\fw)} \leq C\left(\|f\|_{L^2(V,\fw)} + \|u\|_{L^2(V,\fw)}\right),
$$
where $C=C(\Lambda,\nu_0,R_1,R_0)$ is a positive constant. Finally,
$$
\|u_{xx}\|_{L^2(U',\fw_1)} \leq \|u_x\|_{H^2(U',\fw_1)} \leq C\|u_x\|_{H^2(U',\fw)},
$$
and applying Proposition \ref{prop:SobolevRegularity_ukx}, we obtain
$$
\|u_x\|_{H^2(U',\fw)} \leq C\left(\|f_x\|_{L^2(V,\fw)} + \|f\|_{L^2(V,\fw)} + \|u\|_{L^2(V,\fw)}\right).
$$
Combining the preceding estimates gives
\begin{align*}
\|u_y\|_{H^2(U,\fw_1)} &\leq C\left(\|f_y\|_{L^2(U',\fw_1)} + \|f_x\|_{L^2(V,\fw)} + \|f\|_{L^2(V,\fw)} + \|u\|_{L^2(V,\fw)}\right)
\\
&\leq C\left(\|f_y\|_{L^2(V,\fw)} + \|f_x\|_{L^2(V,\fw)} + \|f\|_{L^2(V,\fw)} + \|u\|_{L^2(V,\fw)}\right),
\end{align*}
and this completes the proof.
\end{proof}

\subsection{Interior $\sH^3$  regularity}
\label{subsec:InteriorH3Regularity}
By combining Propositions \ref{prop:SobolevRegularity_ukx} and \ref{prop:SobolevRegularity_uy}, we obtain

\begin{thm}[Interior $\sH^3$ regularity]
\label{thm:H3SobolevRegularityInterior}
Let $R<R_0$ be positive constants. Then there is a positive constant, $C=C(\Lambda,\nu_0,R,R_0)$, such that the following holds. Let $\sO\subseteqq\HH$ be a domain and let $U\subset V$ be as in \eqref{eq:UVHalfBalls}, with $ V\Subset\underline{\sO}$. Suppose that $f\in L^2(\sO,\fw)$ and that $u \in H^1(\sO,\fw)$ is a solution to the variational equation \eqref{eq:HestonVariationalEquation}. If
$
f\in W^{1,2}(V,\fw),
$
then $u \in \sH^3(U,\fw)$ and
\begin{equation}
\label{eq:H3SobolevRegularity}
\|u\|_{\sH^3(U,\fw)} \leq C\left(\|f\|_{W^{1,2}(V,\fw)} + \|u\|_{L^2(V,\fw)}\right).
\end{equation}
\end{thm}

\begin{proof}
Since $f\in L^2(\sO,\fw)$, Theorem \ref{thm:H2BoundSolutionHestonVarEqnSubdomainInterior} implies that $u\in H^2(V,\fw)$. Choose an auxiliary half-ball, $U'$ as in \eqref{eq:UVHalfBalls}, with $U'=B^+_{R_1}(z_0)$ and $U\subset U'\subset V$, and fix $R_1=(R+R_0)/2$. By hypothesis, we have $f\in W^{1,2}(V,\fw)$ and so Proposition \ref{prop:SobolevRegularity_ukx} yields $u_x\in H^2(U',\fw)$ and
\begin{align*}
{}&\|yu_{xxx}\|_{L^2(U',\fw)} + \|yu_{xxy}\|_{L^2(U',\fw)}
+ \|yu_{xyy}\|_{L^2(U',\fw)}
 + \|(1+y)u_{xx}\|_{L^2(U',\fw)} + \|(1+y)u_{xy}\|_{L^2(U',\fw)}
\\
&\quad \leq \|u_x\|_{H^2(U',\fw)}
\\
&\quad \leq C\left(\|f_x\|_{L^2(V,\fw)} + \|f\|_{L^2(V,\fw)} + \|u\|_{L^2(V,\fw)}\right).
\end{align*}
Because $f\in W^{1,2}(U',\fw)$ by hypothesis, and $u\in H^2(U',\fw)$, and $u_x\in H^1(U',\fw)$ (since $u_x\in H^2(U',\fw)$), then Proposition \ref{prop:SobolevRegularity_uy} gives $u_y\in H^2(U,\fw_1)$ and
\begin{align*}
{}&\|yu_{yyy}\|_{L^2(U,\fw_1)} + \|(1+y)u_{yy}\|_{L^2(U,\fw_1)}
\\
&\quad \leq \|u_y\|_{H^2(U,\fw_1)}
\\
&\quad \leq C\left(\|Df\|_{L^2(V,\fw)} + \|f\|_{L^2(V,\fw)} + \|u\|_{L^2(V,\fw)}\right).
\end{align*}
Because $u\in H^2(V,\fw)$ by hypothesis, we obtain $u\in \sH^3(U,\fw)$ from Definition \ref{defn:HkWeightedSobolevSpaceNormPowery}, since
\begin{align*}
\|u\|_{\sH^3(U,\fw)}^2 &= \|yu_{xxx}\|_{L^2(U,\fw)}^2 + \|yu_{xxy}\|_{L^2(U,\fw)}^2
+ \|yu_{xyy}\|_{L^2(U,\fw)}^2
 + \|yu_{yyy}\|_{L^2(U,\fw_1)}^2
\\
&\quad + \|(1+y)u_{xx}\|_{L^2(U,\fw)}^2 + \|(1+y)u_{xy}\|_{L^2(U,\fw)}^2 + \|(1+y)u_{yy}\|_{L^2(U,\fw_1)}^2
\\
&\quad + \|(1+y)u_x\|_{L^2(U,\fw)}^2 + \|(1+y)u_y\|_{L^2(U,\fw)}^2 + \|(1+y)^{1/2}u\|_{L^2(U,\fw)}^2,
\end{align*}
and hence
\begin{equation}
\label{eq:H3induction}
\begin{aligned}
\|u\|_{\sH^3(U,\fw)}^2 &\leq \|yu_{xxx}\|_{L^2(U,\fw)}^2 + \|yu_{xxy}\|_{L^2(U,\fw)}^2
+ \|yu_{xyy}\|_{L^2(U,\fw)}^2
+ \|yu_{yyy}\|_{L^2(U,\fw_1)}^2
\\
&\quad + \|(1+y)u_{xx}\|_{L^2(U,\fw)}^2 + \|(1+y)u_{xy}\|_{L^2(U,\fw)}^2 + \|(1+y)u_{yy}\|_{L^2(U,\fw_1)}^2
\\
&\quad + \|u\|_{H^2(U,\fw)}^2.
\end{aligned}
\end{equation}
Since $u$ obeys \eqref{eq:HestonVariationalEquation}, Theorem \ref{thm:H2BoundSolutionHestonVarEqnSubdomainInterior} yields
$$
\|u\|_{H^2(U,\fw)} \leq C\left(\|f\|_{L^2(V,\fw)} + \|u\|_{L^2(V,\fw)}\right),
$$
and combining the preceding estimates gives \eqref{eq:H3SobolevRegularity}.
\end{proof}

\subsection{Interior $\sH^{k+2}$ regularity}
\label{subsec:InteriorHk+2Regularity}
We can iterate the preceding arguments, used to establish $u\in \sH^3(U,\fw)$, given $u\in H^2(V,\fw)$ and additional hypotheses on $f$, to give higher-order Sobolev regularity, where $U\subset V$ are as in \eqref{eq:UVHalfBalls} and $V\Subset\underline{\sO}$. We begin with the following combined generalization of Lemmas \ref{lem:VarEqn_Dkx} and \ref{lem:VarEqn_Dy}.

\begin{prop}[Variational equation for higher-order derivatives of a solution with respect to $x$ and $y$]
\label{prop:VarEqn_Dkxy}
Let $\sO\subseteqq\HH$ be a domain with \emph{finite height}\footnote{Proposition \ref{prop:VarEqn_Dkxy} should, of course, hold without a hypothesis that $\sO$ has finite height, but its already technical proof is simpler with this hypothesis included and we shall only apply the result to domains of finite height.}, let $k\geq 1$ and $0\leq m\leq k$ be integers, let $f\in L^2(\sO,\fw)$, and suppose that $u\in H^1(\sO,\fw)$ satisfies the variational equation \eqref{eq:HestonVariationalEquation}. If
$$
f \in W^{k,2}(\sO,\fw),  \quad u \in \sH^{k+1}(\sO, \fw), \quad\hbox{and}\quad D_x^k u \in H^1(\sO,\fw) \quad (m=0,1),
$$
then $D_x^{k-m}D_y^m u \in H^1(\sO,\fw)$ obeys
\begin{equation}
\label{eq:IntroHestonWeakMixedProblemHomogeneous_Dkxy}
\fa_m(D_x^{k-m}D_y^m u,v) = (D_x^{k-m}D_y^m f,v)_{L^2(\sO,\fw_m)}  - m(BD_x^{k-m}D_y^{m-1} u,v)_{L^2(\sO,\fw_m)},
\end{equation}
for all $v\in H^1_0(\underline{\sO},\fw_m)$.
\end{prop}

\begin{rmk}[Need for the auxiliary regularity condition when $m=0,1$]
The role of the auxiliary regularity condition, $D_x^k u \in H^1(\sO,\fw)$, when $m=0$ or $1$ is explained in Appendix \ref{app:VarEqn_DkxyAuxiliarymZero}.
\end{rmk}

\begin{proof}[Proof of Proposition \ref{prop:VarEqn_Dkxy}]
Lemma \ref{lem:VarEqn_Dkx} implies that \eqref{eq:IntroHestonWeakMixedProblemHomogeneous_Dkxy} holds when $m=0$ and \emph{any} $k\geq 1$, while Lemma \ref{lem:VarEqn_Dy} gives the conclusion when $k=m=1$. So we may assume without loss of generality that $k\geq 2$ and $m\geq 1$ in our proof of Proposition \ref{prop:VarEqn_Dkxy}. Therefore, to establish \eqref{eq:IntroHestonWeakMixedProblemHomogeneous_Dkxy}, it suffices to consider the inductive step $(k,m-1)\implies (k,m)$ (one extra derivative with respect to $y$), assuming \eqref{eq:IntroHestonWeakMixedProblemHomogeneous_Dkxy} holds with $m$ replaced by $m-1$. The argument for this inductive step follows the pattern of proof of Lemma \ref{lem:VarEqn_Dy}.

As usual, suppose first that $u \in C^\infty(\bar\sO)$ and $v \in C^\infty_0(\underline{\sO})$. Then $(yv)_y \in C^\infty_0(\underline{\sO})$ too and
\begin{align*}
{}&\fa_{m-1}(D_x^{k-m}D_y^{m-1} u, (yv)_y)
\\
&= (A_{m-1}D_x^{k-m}D_y^{m-1} u, (yv)_y)_{L^2(\sO,\fw_{m-1})} \quad\hbox{(by Lemma \ref{lem:HestonIntegrationByParts})}
\\
&= -((A_{m-1}D_x^{k-m}D_y^{m-1} u)_y, yv)_{L^2(\sO,\fw_{m-1})}
- (A_{m-1}D_x^{k-m}D_y^{m-1} u, yv(\log\fw)_y)_{L^2(\sO,\fw_{m-1})}
\\
&= -(A_mD_x^{k-m}D_y^m u, yv)_{L^2(\sO,\fw_{m-1})} -(BD_x^{k-m}D_y^{m-1} u, yv)_{L^2(\sO,\fw_{m-1})}
\\
&\quad - (A_{m-1}D_x^{k-m}D_y^{m-1} u, yv(\log\fw)_y)_{L^2(\sO,\fw_{m-1})} \quad\hbox{(by \eqref{eq:RefinedCommutator}),}
\end{align*}
that is, by Lemma \ref{lem:HestonIntegrationByParts} and \eqref{eq:HestonWeightn}, we have that
\begin{equation}
\label{eq:IntroHestonWeakMixedProblemHomogeneous_Dkxy_prelimLHS}
\begin{aligned}
\fa_{m-1}(D_x^{k-m}D_y^{m-1} u, (yv)_y) &= -\fa_m(D_x^{k-m}D_y^m u,v)  -(BD_x^{k-m}D_y^{m-1} u, v)_{L^2(\sO,\fw_m)}
\\
&\quad - (A_{m-1}D_x^{k-m}D_y^{m-1} u, yv(\log\fw)_y)_{L^2(\sO,\fw_{m-1})},
\end{aligned}
\end{equation}
for all $v\in C^\infty_0(\underline{\sO})$. We next establish the

\begin{claim}
\label{claim:IntroHestonWeakMixedProblemHomogeneous_Dkxy_prelimLHS_Sobolevu}
The identity \eqref{eq:IntroHestonWeakMixedProblemHomogeneous_Dkxy_prelimLHS} continues to hold when the requirement $u \in C^\infty(\bar\sO)$ is relaxed to $u \in \sH^{k+1}(\sO,\fw)$ together with, when $m=1$, $D_x^{k+1} u \in L^2(\sO,\fw_1)$.
\end{claim}

\begin{proof}
The terms in the right and left-hand sides of the identity \eqref{eq:IntroHestonWeakMixedProblemHomogeneous_Dkxy_prelimLHS} are well-defined when
$D_x^{k-m}D_y^{m-1} u \in H^1(\sO,\fw_{m-1})$, $D_x^{k-m}D_y^m u \in H^1(\sO,\fw_m)$,
$D_x^{k+1-m}D_y^{m-1} u$, $D_x^{k+2-m}D_y^{m-1} u \in L^2(\sO,\fw_m)$,
and $D_x^{k-m}D_y^{m-1} u \in H^2(\sO,\fw_{m-1})$.
We consider each of the five preceding terms in turn. First, according to Definition \ref{defn:HkWeightedSobolevSpaceNormPowery}, we have that $u \in \sH^{k+1}(\sO,\fw)$ implies
\begin{align*}
D_x^{k-m}D_y^{m-1} u &\in \begin{cases}L^2(\sO,\fw_{m-2}), &m \geq 3, \\ L^2(\sO,\fw), &m =1,2, \end{cases}
\\
D_x^{k+1-m}D_y^{m-1} u &\in \begin{cases}L^2(\sO,\fw_{m-2}), &m \geq 3, \\ L^2(\sO,\fw), &m =1,2, \end{cases}
\\
D_x^{k-m}D_y^m u &\in \begin{cases}L^2(\sO,\fw_{m-1}), &m \geq 2, \\ L^2(\sO,\fw), &m =1. \end{cases}
\end{align*}
Since $L^2(\sO,\fw_{m-2}) \subset L^2(\sO,\fw_{m-1})$ (for any $m\geq 2$) and $L^2(\sO,\fw) \subset L^2(\sO,\fw_{m-1})$ (for any $m\geq 1$), then
$D_x^{k-m}D_y^{m-1} u$, $y^{1/2}D_x^{k+1-m}D_y^{m-1} u$, $y^{1/2}D_x^{k-m}D_y^m u \in L^2(\sO,\fw_{m-1})$, when $1\leq m \leq k$,
and the definition \eqref{eq:H1WeightedSobolevSpace} of $H^1(\sO,\fw_{m-1})$ gives
$$
u \in \sH^{k+1}(\sO,\fw) \implies D_x^{k-m}D_y^{m-1} u \in H^1(\sO,\fw_{m-1}), \quad 1\leq m \leq k.
$$
Second, according to Definition \ref{defn:HkWeightedSobolevSpaceNormPowery}, we see $u \in \sH^{k+1}(\sO,\fw)$ implies
\begin{align*}
D_x^{k-m}D_y^m u &\in \begin{cases}L^2(\sO,\fw_{m-1}), &m \geq 2, \\ L^2(\sO,\fw), &m =1, \end{cases}
\\
yD_x^{k+1-m}D_y^m u &\in \begin{cases}L^2(\sO,\fw_{m-2}), &m \geq 3, \\ L^2(\sO,\fw), &m =1,2, \end{cases}
\\
yD_x^{k-m}D_y^{m+1} u &\in \begin{cases}L^2(\sO,\fw_{m-1}), &m \geq 2, \\ L^2(\sO,\fw), &m =1, \end{cases}
\end{align*}
that is,
\begin{align*}
D_x^{k-m}D_y^m u &\in L^2(\sO,\fw_{m-1}), \quad m \geq 1,
\\
D_x^{k+1-m}D_y^m u &\in \begin{cases}L^2(\sO,\fw_m), &m \geq 3, \\ L^2(\sO,\fw_2), &m =1,2, \end{cases}
\\
D_x^{k-m}D_y^{m+1} u &\in L^2(\sO,\fw_{m+1}), \quad m \geq 1.
\end{align*}
Therefore, using $L^2(\sO,\fw_{m-1}) \subset L^2(\sO,\fw_m)$ (for any $m\geq 1$), we obtain that the weighted derivatives
$D_x^{k-m}D_y^m u$, $y^{1/2}D_x^{k+1-m}D_y^m u$, and $y^{1/2}D_x^{k-m}D_y^{m+1} u$ belong to $L^2(\sO,\fw_m)$, for $1\leq m \leq k$, and the definition \eqref{eq:H1WeightedSobolevSpace} of $H^1(\sO,\fw_m)$ gives
$$
u \in \sH^{k+1}(\sO,\fw) \implies D_x^{k-m}D_y^m u \in H^1(\sO,\fw_m), \quad 1\leq m \leq k.
$$
Third, we have seen that $u \in \sH^{k+1}(\sO,\fw)$ implies
$$
D_x^{k+1-m}D_y^{m-1} u \in \begin{cases}L^2(\sO,\fw_{m-2}), &m \geq 3, \\ L^2(\sO,\fw), &m =1,2, \end{cases}
$$
and so, using $L^2(\sO,\fw_{m-2}) \subset L^2(\sO,\fw_m)$ (for any $m\geq 2$), we obtain
$$
D_x^{k+1-m}D_y^{m-1} u \in L^2(\sO,\fw_m), \quad 1\leq m \leq k.
$$
For the fifth term, $D_x^{k-m}D_y^{m-1} u$ (we shall consider the fourth term last), observe that $u \in \sH^{k+1}(\sO,\fw)$ implies
\begin{align*}
yD_x^{k+2-m}D_y^{m-1} u  &\in \begin{cases}L^2(\sO, \fw_{m-3}), & m \geq 4, \\ L^2(\sO, \fw), & m = 1,2,3,\end{cases}
\\
yD_x^{k+1-m}D_y^{m} u  &\in \begin{cases}L^2(\sO, \fw_{m-2}), & m \geq 3, \\ L^2(\sO, \fw), & m = 1,2,\end{cases}
\\
yD_x^{k-m}D_y^{m+1} u  &\in \begin{cases}L^2(\sO, \fw_{m-1}), & m \geq 2, \\ L^2(\sO, \fw), & m = 1,\end{cases}
\\
D_x^{k+1-m}D_y^{m-1} u  &\in \begin{cases}L^2(\sO, \fw_{m-2}), & m \geq 3, \\ L^2(\sO, \fw), & m = 1,2,\end{cases}
\\
D_x^{k-m}D_y^{m} u  &\in \begin{cases}L^2(\sO, \fw_{m-1}), & m \geq 2, \\ L^2(\sO, \fw), & m = 1,\end{cases}
\\
D_x^{k-m}D_y^{m-1} u  &\in \begin{cases}L^2(\sO, \fw_{m-2}), & m \geq 3, \\ L^2(\sO, \fw), & m = 1,2.\end{cases}
\end{align*}
Hence, from the definition \eqref{eq:H2WeightedSobolevSpace} of $H^2(\sO, \fw_{m-1})$, we see that
$
D_x^{k-m}D_y^{m-1} u \in H^2(\sO, \fw_{m-1}),
$
for all $1\leq m \leq k$.
Finally, considering the fourth term\footnote{As explained in Appendix \ref{app:VarEqn_DkxyAuxiliarymZero}, it is only in the case $m=1$ that $D_x^{k+2-m}D_y^{m-1} u \in L^2(\sO,\fw_m)$ is not implied by $u\in \sH^{k+1}(\sO,\fw)$, and this case is explicitly covered by the additional hypothesis, $D_x^k u \in H^1(\sO,\fw)$, which ensures, by definition \eqref{eq:H1WeightedSobolevSpace} of $H^1(\sO,\fw)$, that $y^{1/2}D_x^{k+1} u \in L^2(\sO,\fw)$ or, equivalently, $D_x^{k+1} u \in L^2(\sO,\fw_1)$.}, observe that for each $v\in C^\infty_0(\underline\sO)$, we may choose a subdomain $\sO'\Subset\underline\sO$ such that $\supp v \subset\underline\sO'$ and $\partial_1\sO'$ is $C^1$-orthogonal to $\partial\HH$, in the sense of Definition \ref{defn:C1Orthogonal}. According to Theorem \ref{thm:KufnerPowerWeightBoundedDerivatives}, there is a sequence $\{u_n\}_{n\in \NN} \subset C^\infty(\bar\sO')$ such that $u_n \to u$ in $\sH^{k+1}(\sO',\fw)$ as $n\to\infty$ and hence, for each $v \in C^\infty_0(\underline{\sO})$ with
$\supp v \subset\underline\sO'$, we have
$$
(D_x^{k+2-m}D_y^{m-1} u_n, v)_{L^2(\sO, \fw_m)} \to  (D_x^{k+2-m}D_y^{m-1} u, v)_{L^2(\sO, \fw_m)},
 \quad\hbox{ as }n\to\infty,
$$
since, for all $n\in\NN$,
\begin{align*}
{}& \left|(D_x^{k+2-m}D_y^{m-1}(u_n - u), v)_{L^2(\sO,\fw_m)}\right|
\\
&= \left|(y^mD_x^{k+2-m}D_y^{m-1}(u_n - u), v)_{L^2(\sO,\fw)}\right|
\\
&\leq \|y^mD_x^{k+2-m}D_y^{m-1}(u_n - u)\|_{L^2(\sO,\fw)}\|v\|_{L^2(\sO,\fw)}
\\
&\leq C\|u_n - u\|_{\sH^{k+1}(\sO,\fw)}\|v\|_{L^2(\sO,\fw)},
\end{align*}
where $C=C(\height(\sO))$ is a positive constant, noting that $u \in \sH^{k+1}(\sO,\fw)$ implies, by Definition \ref{defn:HkWeightedSobolevSpaceNormPowery},
$$
yD_x^{k+2-m}D_y^{m-1} u \in \begin{cases} L^2(\sO,\fw_{m-3}), &m \geq 4, \\ L^2(\sO,\fw), &m = 1,2,3, \end{cases}
$$
and, for $C=C(\height(\sO))$, a positive constant,
\begin{align*}
\|y^mD_x^{k+2-m}D_y^{m-1}(u_n - u)\|_{L^2(\sO,\fw)} &= \|y^{3+(m-3)/2}D_x^{k+2-m}D_y^{m-1}(u_n - u)\|_{L^2(\sO,\fw_{m-3})}
\\
&\leq C\|yD_x^{k+2-m}D_y^{m-1}(u_n - u)\|_{L^2(\sO,\fw_{m-3})}
\\
&\leq C\|u_n - u\|_{\sH^{k+1}(\sO,\fw)}, \quad m\geq 4,
\\
\|y^mD_x^{k+2-m}D_y^{m-1}(u_n - u)\|_{L^2(\sO,\fw)} &= \|y^{1+(m-1)}D_x^{k+2-m}D_y^{m-1}(u_n - u)\|_{L^2(\sO,\fw)}
\\
&\leq \|yD_x^{k+2-m}D_y^{m-1}(u_n - u)\|_{L^2(\sO,\fw)}
\\
&\leq C\|u_n - u\|_{\sH^{k+1}(\sO,\fw)}, \quad m = 1,2,3.
\end{align*}
Therefore, by approximation, the identity \eqref{eq:IntroHestonWeakMixedProblemHomogeneous_Dkxy_prelimLHS} continues to hold for $u\in \sH^{k+1}(\sO,\fw)$ and, when $m=1$, that $D_x^{k+1} u \in L^2(\sO,\fw_1)$. This completes the proof of Claim \ref{claim:IntroHestonWeakMixedProblemHomogeneous_Dkxy_prelimLHS_Sobolevu}.
\end{proof}

By induction on $m$, the identity \eqref{eq:IntroHestonWeakMixedProblemHomogeneous_Dkxy} holds for $(k-1,m-1)$ in place of $(k,m)$, and so for all $v\in C^\infty_0(\underline{\sO})$, and thus $(yv)_y\in C^\infty_0(\underline{\sO})$, we have
\begin{align*}
{}&\fa_{m-1}(D_x^{k-m}D_y^{m-1} u, (yv)_y)
\\
&= (D_x^{k-m}D_y^{m-1} f, (yv)_y)_{L^2(\sO,\fw_{m-1})} - (m-1)(BD_x^{k-m}D_y^{m-2} u, (yv)_y)_{L^2(\sO,\fw_{m-1})}.
\end{align*}
Therefore, integrating by parts with respect to $y$ on the right-hand side of the preceding identity and applying \eqref{eq:HestonWeightn} yields
\begin{align*}
{}&\fa_{m-1}(D_x^{k-m}D_y^{m-1} u, (yv)_y)
\\
&= -(D_x^{k-m}D_y^m f, v)_{L^2(\sO,\fw_m)}  + (m-1)(BD_x^{k-m}D_y^{m-1} u, v)_{L^2(\sO,\fw_m)}
\\
&\quad - (D_x^{k-m}D_y^{m-1} f, yv(\log\fw)_y)_{L^2(\sO,\fw_{m-1})} + (m-1)(BD_x^{k-m}D_y^{m-2} u, yv(\log\fw)_y)_{L^2(\sO,\fw_{m-1})},
\\
&\qquad\forall v\in C^\infty_0(\underline{\sO}).
\end{align*}
But $D_x^{k-m}D_y^{m-1} f \in L^2(\sO,\fw_{m-1})$, since $f\in W^{k,2}(\sO,\fw)$ by hypothesis, and $Au=f$ a.e. on $\sO$ yields
\begin{align*}
D_x^{k-m}D_y^{m-1} f &= D_x^{k-m}D_y^{m-1} Au \quad\hbox{a.e. on }\sO
\\
&= A_{m-1}D_x^{k-m}D_y^{m-1} u + (m-1)BD_x^{k-m}D_y^{m-2} u \quad\hbox{a.e. on }\sO \quad\hbox{(by \eqref{eq:RefinedCommutator})},
\end{align*}
noting that $u \in \sH^{k+1}(\sO,\fw)$ by hypothesis, and so (by an analysis very similar to that in the proof of Claim \ref{claim:IntroHestonWeakMixedProblemHomogeneous_Dkxy_prelimLHS_Sobolevu}),
$$
A_{m-1}D_x^{k-m}D_y^{m-1} u, \ BD_x^{k-m}D_y^{m-2} u \in L^2(\sO,\fw_{m-1}).
$$
Substituting this identity for $D_x^{k-m}D_y^{m-1} f$ into the preceding variational equation yields
\begin{equation}
\label{eq:IntroHestonWeakMixedProblemHomogeneous_Dkxy_prelimRHS}
\begin{aligned}
{}&\fa_{m-1}(D_x^{k-m}D_y^{m-1} u, (yv)_y)
\\
&= -(D_x^{k-m}D_y^m f, v)_{L^2(\sO,\fw_m)}  + (m-1)(BD_x^{k-m}D_y^{m-1} u, v)_{L^2(\sO,\fw_m)}
\\
&\quad - (A_{m-1}D_x^{k-m}D_y^{m-1} u, yv(\log\fw)_y)_{L^2(\sO,\fw_{m-1})}, 
\end{aligned}
\end{equation}
for all $v\in C^\infty_0(\underline{\sO})$.
Combining the variational equations \eqref{eq:IntroHestonWeakMixedProblemHomogeneous_Dkxy_prelimLHS} and \eqref{eq:IntroHestonWeakMixedProblemHomogeneous_Dkxy_prelimRHS} yields
$$
\fa_m(D_x^{k-m}D_y^m u,v) = (D_x^{k-m}D_y^m f,v)_{L^2(\sO,\fw_m)}  - m(BD_x^{k-m}D_y^{m-1} u,v)_{L^2(\sO,\fw_m)}, \quad v\in C^\infty_0(\underline{\sO}),
$$
and hence \eqref{eq:IntroHestonWeakMixedProblemHomogeneous_Dkxy} holds for all $v\in H^1_0(\underline{\sO},\fw)$. This completes the proof of Proposition \ref{prop:VarEqn_Dkxy}.
\end{proof}

We now show that $D_x^{k-m}D_y^m u \in H^2(U,\fw_m)$, where $U\Subset\underline{\sO}$ is as in \eqref{eq:UVHalfBalls}, for any $k\geq 1$ and $0\leq m \leq k$, and provide estimates for these derivatives analogous to those in Propositions \ref{prop:SobolevRegularity_ukx} and Propositions \ref{prop:SobolevRegularity_uy}.

\begin{prop}[Interior $H^2$ regularity for higher-order derivatives of a solution with respect to $x$ and $y$]
\label{prop:SobolevRegularity_ukxy}
Let $R<R_0$ be positive constants and let $k\geq 1$ and $0\leq m\leq k$ be integers. Then there is a positive constant, $C=C(\Lambda,\nu_0,k,m,R,R_0)$, such that the following holds. Let $\sO\subseteqq\HH$ be a domain and let $U\subset V$ be as in \eqref{eq:UVHalfBalls}, with $V\Subset\underline{\sO}$. Suppose that $f\in L^2(\sO,\fw)$ and $u \in H^1(\sO,\fw)$ is a solution to the variational equation \eqref{eq:HestonVariationalEquation}. If
$f\in W^{k,2}(V,\fw)$ and $u \in H^2(V,\fw)$,
then
$
D_x^{k-m}D_y^m u \in H^2(U,\fw_m),
$
and
\begin{equation}
\label{eq:SobolevRegularity_ukxy}
\|D_x^{k-m}D_y^m u\|_{H^2(U,\fw_m)} \leq C\left(\|f\|_{W^{k,2}(V,\fw)} + \|u\|_{L^2(V,\fw)}\right).
\end{equation}
\end{prop}

\begin{proof}
Proposition \ref{prop:SobolevRegularity_ukx} yields the conclusion for any $k\geq 1$, when $m=0$, while Proposition \ref{prop:SobolevRegularity_uy} gives the conclusion when $k=m=1$. Therefore, we may assume that $k\geq 2$ and $m\geq 1$.

By Proposition \ref{prop:VarEqn_Dkxy}, we see that $D_x^{k-m}D_y^m u \in H^1(\sO,\fw)$ obeys \eqref{eq:IntroHestonWeakMixedProblemHomogeneous_Dkxy}, that is
$$
\fa_m(D_x^{k-m}D_y^m u,v) = (D_x^{k-m}D_y^m f,v)_{L^2(V,\fw_m)}  - m(BD_x^{k-m}D_y^{m-1} u,v)_{L^2(V,\fw_m)},
$$
for all $v\in H^1_0(\underline{V},\fw_m)$, \emph{provided} (in addition to $f\in W^{k,2}(V,\fw)$) that $u \in \sH^{k+1}(V,\fw)$ and $D_x^k u \in H^1(\sO,\fw)$, when $m=0,1$.
The condition $D_x^k u \in H^1(\sO,\fw)$ follows from Proposition \ref{prop:SobolevRegularity_ukx}, since $f\in W^{k,2}(V,\fw)$ by hypothesis.

Thus, it remains to verify that $u \in \sH^{k+1}(V,\fw)$ and justify the application of Proposition \ref{prop:VarEqn_Dkxy}, noting that, by induction on $k$, we may assume that Proposition \ref{prop:SobolevRegularity_ukxy} holds for $k$ replaced by $k-1$, and so we may assume that $u \in \sH^k(V,\fw)$.

\begin{claim}
\label{claim:uHk+1}
$u \in \sH^{k+1}(V,\fw)$, for $k\geq 2$.
\end{claim}

\begin{proof}
According to \eqref{eq:Hk+3induction}, we have
\begin{align*}
\|v\|_{\sH^{k+1}(U,\fw)}^2 &\leq \|yD^{k+1}_x u\|_{L^2(U,\fw)}^2 + \|yD^k_xD_y u\|_{L^2(U,\fw)}^2 + \|yD^{k-1}_xD_y^2 u\|_{L^2(U,\fw)}^2
\\
&+ \sum_{m=1}^{k-1}\|yD^{k-1-m}_xD_y^{m+2} u\|_{L^2(U,\fw_m)}^2
\\
&+ \|(1+y)D^k_x u\|_{L^2(U,\fw)}^2 + \|(1+y)D^{k-1}_xD_y u\|_{L^2(U,\fw)}^2
\\
&+ \sum_{m=1}^{k-1}\|(1+y)D^{k-1-m}_xD_y^{m+1} u\|_{L^2(U,\fw_m)}^2
+ \|v\|_{\sH^k(U,\fw)}^2,
\end{align*}
and so we may conclude that $u \in \sH^{k+1}(U,\fw)$, if the terms on the right hand side are finite.

By induction on $k$, Proposition \ref{prop:SobolevRegularity_ukxy} gives $D^{k-1-m}_xD_y^m u \in H^2(U,\fw_m)$, for $0\leq m\leq k-1$, and
$$
\|D^{k-1-m}_xD_y^m u\|_{H^2(U,\fw_m)} \leq C\left(\|f\|_{W^{k-1,2}(V,\fw)} + \|u\|_{L^2(V,\fw)}\right),
$$
where $C=C(\Lambda,\nu_0,k,m,R,R_0)$ is a positive constant. The preceding estimate yields
\begin{align*}
{}&\|yD^{k+1}_x u\|_{L^2(U,\fw)}^2 + \|yD^k_xD_y u\|_{L^2(U,\fw)}^2 + \|yD^{k-1}_xD_y^2 u\|_{L^2(U,\fw)}^2
+ \sum_{m=1}^{k-1}\|yD^{k-1-m}_xD_y^{m+2} u\|_{L^2(U,\fw_m)}^2
\\
&+ \|(1+y)D^k_x u\|_{L^2(U,\fw)}^2 + \|(1+y)D^{k-1}_xD_y u\|_{L^2(U,\fw)}^2
+ \sum_{m=1}^{k-1}\|(1+y)D^{k-1-m}_xD_y^{m+1} u\|_{L^2(U,\fw_m)}^2
\\
&\quad\leq C\left(\|f\|_{W^{k-1,2}(V,\fw)} + \|u\|_{L^2(V,\fw)}\right)^2.
\end{align*}
Combining the preceding estimates yields $u \in \sH^{k+1}(U,\fw)$, for $k\geq 2$, and completes the proof of Claim \ref{claim:uHk+1}.
\end{proof}

We now proceed to verify the estimate \eqref{eq:SobolevRegularity_ukxy}. Because $D_x^{k-m}D_y^m f \in L^2(V,\fw_m)$ by hypothesis, we can apply Theorem \ref{thm:H2BoundSolutionHestonVarEqnSubdomainInterior} to \eqref{eq:IntroHestonWeakMixedProblemHomogeneous_Dkxy} and conclude that $D_x^{k-m}D_y^m u \in H^2(U,\fw_m)$ and
\begin{align*}
\|D_x^{k-m}D_y^m u\|_{H^2(U,\fw_m)} &\leq C\left(\|D_x^{k-m}D_y^m f\|_{L^2(U',\fw_m)} + \|D_x^{k+1-m}D_y^{m-1} u\|_{L^2(U',\fw_m)} \right.
\\
&\qquad + \left. \|D_x^{k+2-m}D_y^{m-1} u\|_{L^2(U',\fw_m)} + \|D_x^{k-m}D_y^m u\|_{L^2(U',\fw_m)}\right),
\end{align*}
where $U'$ is as in \eqref{eq:UVHalfBalls}, with $U'=B^+_{R_1}(z_0)$ and $U\subset U'\subset V$ and $R_1=(R+R_0)/2$, and $C=C(\Lambda,\nu_0,R,R_1)$ is a positive constant.

We now estimate the terms on the right-hand side of the preceding inequality. Observe that
$$
\|D_x^{k+1-m}D_y^{m-1} u\|_{L^2(U',\fw_m)} \leq \|D_x^{k-m}D_y^{m-1} u\|_{H^2(U',\fw_m)} \leq C\|D_x^{k-m}D_y^{m-1} u\|_{H^2(U',\fw_{m-1})},
$$
where the first inequality follows from \eqref{eq:H2WeightedSobolevSpace} and the second from \eqref{eq:SobolevInclusionDifferentWeights}, with $C=C(R_1)$. By induction on $k$ and $m$, we may assume that Proposition \ref{prop:SobolevRegularity_ukxy} holds for $k-1$ in place of $k$ and $m-1$ in place of $m$ and so $D_x^{k-m}D_y^{m-1} u  = D_x^{k-1-(m-1)}D_y^{m-1} u \in H^2(U',\fw_{m-1})$ with
$$
\|D_x^{k-m}D_y^{m-1} u\|_{H^2(U',\fw_{m-1})} \leq C\left(\|f\|_{W^{k-1,2}(V,\fw)} + \|u\|_{L^2(V,\fw)}\right),
$$
where $C=C(\Lambda,\nu_0,k,R_1,R_0)$ is a positive constant. Similarly, observe that
$$
\|D_x^{k-m}D_y^m u\|_{L^2(U',\fw_m)} \leq \|D_x^{k-m}D_y^{m-1} u\|_{H^2(U',\fw_m)} \leq C\|D_x^{k-m}D_y^{m-1} u\|_{H^2(U',\fw_{m-1})},
$$
where the last term is estimated above. Finally, we notice that
$$
\|D_x^{k+2-m}D_y^{m-1} u\|_{L^2(U',\fw_m)} \leq \|D_x^{k+1-m}D_y^{m-1} u\|_{H^2(U',\fw_m)} \leq C\|D_x^{k+1-m}D_y^{m-1} u\|_{H^2(U',\fw_{m-1})},
$$
where $C=C(R_1)$ is a positive constant. For a given $k\geq 2$, we may assume by induction on $m$ that Proposition \ref{prop:SobolevRegularity_ukxy} holds for $m-1$ in place of $m$, and so
$$
D_x^{k-(m-1)}D_y^{m-1} u = D_x^{k+1-m}D_y^{m-1} u \in H^2(U',\fw_{m-1}),
$$
with
$$
\|D_x^{k+1-m}D_y^{m-1} u\|_{H^2(U',\fw_{m-1})} \leq C\left(\|f\|_{W^{k,2}(V,\fw)} + \|u\|_{L^2(V,\fw)}\right),
$$
where $C=C(\Lambda,\nu_0,k,m,R_1,R_0)$ is a positive constant. Combining the preceding estimates gives \eqref{eq:SobolevRegularity_ukxy}.
\end{proof}

We can now combine our results for higher-order derivatives with respect to $x$ and $y$ to prove the extension, Theorem \ref{thm:HkSobolevRegularityInterior}, of Theorem \ref{thm:H3SobolevRegularityInterior} from the case $k=1$ to $k\geq 1$.

\begin{proof}[Proof of Theorem \ref{thm:HkSobolevRegularityInterior}]
When $k=0$, the conclusion is given by Theorem \ref{thm:H2BoundSolutionHestonVarEqnSubdomainInterior} while if $k=1$, the conclusion follows from Theorem \ref{thm:H3SobolevRegularityInterior}, so we may assume that $k\geq 2$. According to \eqref{eq:Hk+3induction}, we have
\begin{align*}
\|v\|_{\sH^{k+2}(U,\fw)}^2 &\leq \|yD^{k+2}_x u\|_{L^2(U,\fw)}^2 + \|yD^{k+1}_xD_y u\|_{L^2(U,\fw)}^2 + \|yD^{k}_xD_y^2 u\|_{L^2(U,\fw)}^2
\\
&+ \sum_{m=1}^{k}\|yD^{k-m}_xD_y^{m+2} u\|_{L^2(U,\fw_m)}^2
\\
&+ \|(1+y)D^{k+1}_x u\|_{L^2(U,\fw)}^2 + \|(1+y)D^{k}_xD_y u\|_{L^2(U,\fw)}^2
\\
&+ \sum_{m=1}^{k}\|(1+y)D^{k-m}_xD_y^{m+1} u\|_{L^2(U,\fw_m)}^2
 + \|v\|_{\sH^{k+1}(U,\fw)}^2,
\end{align*}
and so we may conclude that $u \in \sH^{k+2}(U,\fw)$, if the terms on the right hand side are finite.

By induction on $k$, we may assume that Theorem \ref{thm:HkSobolevRegularityInterior} holds for $k-1$ in place of $k$, and so $u \in \sH^{k+1}(U,\fw)$ and
$$
\|u\|_{\sH^{k+1}(U,\fw)} \leq C\left(\|f\|_{W^{k-1,2}(V,\fw)} + \|u\|_{L^2(V,\fw)}\right).
$$
Moreover, Proposition \ref{prop:SobolevRegularity_ukxy} gives $D^{k-m}_xD_y^m u \in H^2(U,\fw_m)$, for $0\leq m\leq k$, and
$$
\|D^{k-m}_xD_y^m u\|_{H^2(U,\fw_m)} \leq C\left(\|f\|_{W^{k,2}(V,\fw)} + \|u\|_{L^2(V,\fw)}\right),
$$
where $C=C(\Lambda,\nu_0,k,m,R,R_0)$ is a positive constant. The preceding estimate yields
\begin{align*}
{}&\|yD^{k+2}_x u\|_{L^2(U,\fw)}^2 + \|yD^{k+1}_xD_y u\|_{L^2(U,\fw)}^2 + \|yD^{k}_xD_y^2 u\|_{L^2(U,\fw)}^2
+ \sum_{m=1}^{k}\|yD^{k-m}_xD_y^{m+2} u\|_{L^2(U,\fw_m)}^2
\\
&+ \|(1+y)D^{k+1}_x u\|_{L^2(U,\fw)}^2 + \|(1+y)D^{k}_xD_y u\|_{L^2(U,\fw)}^2
+ \sum_{m=1}^{k}\|(1+y)D^{k-m}_xD_y^{m+1} u\|_{L^2(U,\fw_m)}^2
\\
&\quad\leq C\left(\|f\|_{W^{k,2}(V,\fw)} + \|u\|_{L^2(V,\fw)}\right)^2.
\end{align*}
Combining the preceding estimates yields $u \in \sH^{k+2}(U,\fw)$ and \eqref{eq:HkSobolevRegularity}, for $k\geq 2$.
\end{proof}

Next, we have

\begin{proof}[Proof of Theorem \ref{thm:HkSobolevRegularityDomain}]
The proof follows by a standard covering argument with the aid of \cite[Theorem 8.10]{GilbargTrudinger} and Theorem \ref{thm:HkSobolevRegularityInterior};
see \cite[Section 4.5]{Feehan_Pop_higherregularityweaksoln_v1} for a detailed proof.
\end{proof}

Similarly, we obtain

\begin{proof}[Proof of Theorem \ref{thm:ExistUniqueHkSobolevRegularityDomain}]
Uniqueness of a solution $u\in H^1(\sO,\fw)$ to the variational inequality \eqref{eq:HestonVariationalEquation} with boundary condition,
$u-g \in H^1_0(\underline\sO,\fw)$ defined by $g\in H^1(\sO,\fw)$, follows from \cite[Theorem 8.15]{Feehan_maximumprinciple}, noting that $f\in L^\infty(\sO,\fw)$ by hypothesis and that $L^\infty(\sO)\subset L^2(\sO,\fw)$ since $\vol(\sO,\fw)<\infty$, while $(1+y)g\in W^{2,\infty}(\sO)$ implies $g\in H^1(\sO,\fw)$. When $g\equiv 0$ on $\sO$, existence of a solution $u\in H^1(\sO,\fw)$ to the variational inequality \eqref{eq:HestonVariationalEquation} follows from \cite[Theorem 3.16]{Daskalopoulos_Feehan_statvarineqheston}, again noting that $f\in L^\infty(\sO)$ by hypothesis. For a non-zero $g$ with $(1+y)g\in W^{2,\infty}(\sO)$, we have $g\in H^2(\sO,\fw)$ and
$
\fa(g,v) = (Ag,v)_{L^2(\sO,\fw)}, 
$
for all $v \in H^1_0(\underline\sO,\fw)$,
by Lemma \ref{lem:HestonIntegrationByParts}. By replacing $u\in H^1(\sO,\fw)$ with $\tilde u := u-g \in H^1_0(\underline\sO,\fw)$ and noting that $\tilde f := f-Ag\in L^\infty(\sO)$, existence of a solution $\tilde u \in H^1_0(\underline\sO,\fw)$ to the variational inequality,
$
\fa(\tilde u,v) = (\tilde f,v)_{L^2(\sO,\fw)}, 
$
for all $v \in H^1_0(\underline\sO,\fw)$,
again follows from \cite[Theorem 3.16]{Daskalopoulos_Feehan_statvarineqheston}. Therefore, we obtain existence of a solution $u\in H^1(\sO,\fw)$ to the variational inequality \eqref{eq:HestonVariationalEquation} with boundary condition, $u-g \in H^1_0(\underline\sO,\fw)$. The facts that $u\in \sH^{k+2}_{\loc}(\underline\sO)$ and $u$ obeys \eqref{eq:HkSobolevRegularityDomain} follow from Theorem \ref{thm:HkSobolevRegularityDomain}.
\end{proof}

\section{Higher-order H\"older regularity for solutions to the variational equation}
\label{sec:HolderRegularity}
In this section, we extend the $C^\alpha_s(\underline\sO)$-regularity results from \cite{Feehan_Pop_regularityweaksoln} for solutions, $u\in H^1(\sO,\fw)$, to the variational equation \eqref{eq:HestonVariationalEquation} to $C^{k,\alpha}_s(\underline\sO)$-regularity, for any integer $k\geq 1$. We begin in \S \ref{subsec:HolderRegularityGradient} by proving $C^\alpha_s(\underline\sO)$-regularity (Theorem \ref{thm:HolderContinuity_Du}) for the gradient of a solution, $u \in H^1(\sO,\fw)$, to the variational equation \eqref{eq:HestonVariationalEquation}. In \S \ref{subsec:HolderRegularityHigherOrder}, we establish $C^\alpha_s(\underline\sO)$-regularity (Proposition \ref{prop:HolderContinuity_ukxy}) for higher-order derivatives of a solution, $u \in H^1(\sO,\fw)$, concluding with a proof of one of our main results, Theorem \ref{thm:HolderContinuity_Dku_Interior}, giving $C^{k,\alpha}_s(\underline\sO)$-regularity of a solution, $u \in H^1(\sO,\fw)$. We conclude in \S \ref{subsec:UniformHolderExponent} with proofs
of our remaining principal results, namely, Corollary \ref{cor:CinftyGlobal}, Theorem \ref{thm:CkalphasHolderContinuityDomain}, Corollary \ref{cor:CkalphasHolderContinuityDomainStrip}, Theorem \ref{thm:ExistUniqueCk2+alphasHolderContinuityDomain}, and Corollaries \ref{cor:ExistUniqueCk2+alphasHolderContinuityDomain} and \ref{cor:Ck2+alphasHolderContinuityDomainStrip}.

\subsection{H\"older regularity for first-order derivatives of solutions to the variational equation}
\label{subsec:HolderRegularityGradient}
We begin with

\begin{prop} [Interior $C^{\alpha}_s$ H\"older continuity of $u_x$ for a solution $u$ to the variational equation]
\label{prop:HolderContinuity_ux}
Let $p > \max\{4,2+\beta\}$ and let $R_0$ be a positive constant. Then there are positive constants, $R_1=R_1(R_0) < R_0$, and $C=C(\Lambda,\nu_0,p,R_0)$, and
$\alpha=\alpha(A,p,R_0)\in(0,1)$ such that the following holds. Let $\sO\subseteqq\HH$ be a domain. If $f \in L^2(\sO,\fw)$, and $u \in H^1(\sO,\fw)$ satisfies the variational equation \eqref{eq:HestonVariationalEquation}, and $z_0 \in \partial_0\sO$ is such that
$
B_{R_0}(z_0) \cap \HH \subset \sO,
$
and
$
f_x \in L^p(B^+_{R_0}(z_0),y^{\beta-1}),
$
then
$
u_x \in C^\alpha_s(\bar B_{R_1}^+(z_0)),
$
and
\begin{equation}
\label{eq:HolderContinuity_ux}
\begin{aligned}
\|u_x\|_{C^\alpha_s(\bar B_{R_1}^+(z_0))} &\leq \left(\|f_x\|_{L^p(B^+_{R_0}(z_0),y^{\beta-1})} + \|f\|_{L^2(B^+_{R_0}(z_0),y^{\beta-1})} \right.
\\
&\qquad + \left. \|u\|_{L^2(B^+_{R_0}(z_0),y^{\beta-1})}\right).
\end{aligned}
\end{equation}
\end{prop}

\begin{proof}
By hypothesis, $f_x \in L^2(B^+_{R_0}(z_0),\fw)$ since $p>2$ and so Proposition \ref{prop:VarEqn_H1Dx} implies that $u_x \in H^1(B^+_{R_2}(z_0),\fw)$, for any $R_2=R_2(R_0)$, to be determined, in the range $0<R_2<R_0$ and $u_x$ obeys
$$
\fa(u_x, v) = (f_x, v)_{L^2(B^+_{R_2}(z_0), \fw)}, \quad \forall v \in H^1_0(\underline{B}^+_{R_2}(z_0), \fw),
$$
and
$$
\|u_x\|_{H^1(B^+_{R_2}(z_0),\fw)} \leq C\left(\|f_x\|_{L^2(B^+_{R_0}(z_0),\fw)} + \|f\|_{L^2(B^+_{R_0}(z_0),\fw)} + \|u\|_{L^2(B^+_{R_0}(z_0),\fw)}\right).
$$
The conclusion $u_x \in C^\alpha_s(\bar B_{R_1}^+(z_0))$ and estimate, with $C=C(\Lambda,\nu_0,p,R_0)$, and $\alpha=\alpha(\Lambda,\nu_0,p,R_0)\in(0,1)$, and $R_1=R_1(R_2)<R_2$,
$$
\|u_x\|_{C^\alpha_s(\bar B_{R_1}^+(z_0))} \leq C\left(\|f_x\|_{L^p(B^+_{R_2}(z_0),y^{\beta-1})} + \|u_x\|_{L^2(B^+_{R_2}(z_0),y^{\beta-1})}\right),
$$
follow by applying Theorem \ref{thm:MainContinuityInteriorL2RHSu} to the variational equation for $u_x$ on $B^+_{R_2}(z_0)$. By definition \eqref{eq:HestonWeight} of $\fw$, we have
(see \cite[Section 5.1]{Feehan_Pop_higherregularityweaksoln_v1} for additional details)
\begin{align*}
\|u_x\|_{L^2(B^+_{R_2}(z_0),y^{\beta-1})}
&=  \left(\int_{B^+_{R_0}(z_0)} u_x^2 \ y^{\beta-1} \, dx\,dy \right)^{1/2}
\\
&\leq  Ce^{\frac{1}{2}\gamma |x_0|} \|u\|_{H^1(B^+_{R_2}(z_0),\fw)},
\end{align*}
for $C=C(\Lambda,\nu_0,R_2)$, a positive constant. Finally, we note that
\begin{align*}
\|f_x\|_{L^2(B^+_{R_0}(z_0),\fw)} &= \left(\int_{B^+_{R_0}(z_0)} f_x^2 \ y^{\beta-1}e^{-\gamma\sqrt{1+x^2}-\mu y} \, dx\,dy \right)^{1/2}
\\
&\leq Ce^{-\frac{1}{2}\gamma|x_0|}\|f_x\|_{L^p(B^+_{R_0}(z_0), y^{\beta-1})},
\end{align*}
where $C=C(p,R_0,R_2,\beta)$, and similarly for the terms $\|f\|_{L^2(B^+_{R_0}(z_0),\fw)}$ and $\|u\|_{L^2(B^+_{R_0}(z_0),\fw)}$. We may choose $R_2=R_0/2$ and combining the preceding inequalities yields \eqref{eq:HolderContinuity_ux}, with $C=C(\Lambda,\nu_0,p,R_0)$, a positive constant.
\end{proof}

Next, we have

\begin{prop} [Interior $C^{\alpha}_s$ H\"older continuity of $u_y$ for a solution $u$ to the variational equation]
\label{prop:HolderContinuity_uy}
Let $p > \max\{4,3+\beta\}$ and let $R_0$ be a positive constant. Then there are positive constants, $R_1=R_1(R_0) < R_0$, and $C=C(\Lambda,\nu_0,p,R_0)$, and
$\alpha=\alpha(\Lambda,\nu_0,p,R_0)\in(0,1)$, such that the following holds. Let $\sO\subseteqq\HH$ be a domain. If $f \in L^2(\sO,\fw)$, and $u \in H^1(\sO,\fw)$ satisfies the variational equation \eqref{eq:HestonVariationalEquation}, and $z_0 \in \partial_0\sO$ is such that
$
B_{R_0}(z_0) \cap \HH \subset \sO,
$
and $f$ obeys
\begin{equation}
\label{eq:Lsfxxfycondition_ball}
f_x, \ f_y, \ f_{xx} \in L^p(B^+_{R_0}(z_0),y^{\beta-1}),
\end{equation}
then
$
u_y \in C^\alpha_s(\bar B_{R_1}^+(z_0)),
$
and
\begin{equation}
\label{eq:HolderContinuity_uy}
\begin{aligned}
\|u_y\|_{C^\alpha_s(\bar B_{R_1}^+(z_0))}
&\leq C\left( \|f_{xx}\|_{L^p(B^+_{R_0}(z_0),y^{\beta-1})} + \|f\|_{W^{1,p}(B^+_{R_0}(z_0),y^{\beta-1})}  \right.
\\
&\qquad + \left. \|u\|_{L^2(B^+_{R_0}(z_0),y^{\beta-1})} \right).
\end{aligned}
\end{equation}
\end{prop}

\begin{proof}
Since $p>2$ and so $f\in W^{1,2}(B^+_{R_0}(z_0), y^{\beta-1})=W^{1,2}(B^+_{R_0}(z_0), \fw)$ by hypothesis, Theorem \ref{thm:H3SobolevRegularityInterior} ensures that $u\in\sH^3(B^+_{R_0/2}(z_0), \fw)$. By Definition \ref{defn:HkWeightedSobolevSpaceNormPowery}, we therefore have $u \in H^2(B^+_{R_0/2}(z_0),\fw)$ and $u_{xx} \in L^2(B^+_{R_0/2}(z_0),\fw_1)$ and so from Lemma \ref{lem:VarEqn_Dy}, we obtain that
$
u_y \in H^1(B^+_{R_0/2}(z_0), \fw_1),
$
and $u_y$ obeys the variational equation,
$$
\fa_1(u_y, v) = (f_y - Bu, v)_{L^2(B^+_{R_0/2}(z_0), \fw_1)}, \quad \forall v \in H^1_0(\underline{B}^+_{R_0/2}(z_0), \fw_1).
$$
We note that the preceding variational equation continues to hold on $B^+_{R_2}(z_0)$, for any $R_2$ in the range $0<R_2\leq R_0/2$ and still to be determined. To apply Theorem \ref{thm:MainContinuityInteriorL2RHSu} to the preceding variational equation and conclude that $u_y \in C^\alpha_s(\bar B_{R_1}^+(z_0))$ for some
$R_1=R_1(R_2)<R_2$ and, for a positive constant $C=C(\Lambda,\nu_0,p,R_2)$,
\begin{equation}
\label{eq:HolderContinuity_uyprelim}
\begin{aligned}
\|u_y\|_{C^\alpha_s(\bar B_{R_1}^+(z_0))} &\leq C\left( \|f_y\|_{L^p(B^+_{R_2}(z_0),y^\beta)}
+ \|u_x\|_{L^p(B^+_{R_2}(z_0),y^\beta)} + \|u_{xx}\|_{L^p(B^+_{R_2}(z_0),y^\beta)} \right.
\\
&\qquad + \left. \|u_y\|_{L^2(B^+_{R_2}(z_0),y^\beta)} \right),
\end{aligned}
\end{equation}
we need $f_y-Bu$ to obey the integrability condition \eqref{eq:Lsfcondition_ball} obeyed by $f$, with $\beta+1$ and $R_2$ in place of $\beta$ and $R_0$, respectively. In other words, $u$ must obey
\begin{equation}
\label{eq:SufficientConditionMainContinuityInterioruy}
u_x, \ u_{xx} \in L^p(B^+_{R_2}(z_0),y^\beta),
\end{equation}
while our hypothesis \eqref{eq:Lsfxxfycondition_ball} on $f$ ensures that $f_y \in L^p(B^+_{R_2}(z_0),y^{\beta-1}) \subset L^p(B^+_{R_2}(z_0),y^\beta)$. For the condition \eqref{eq:SufficientConditionMainContinuityInterioruy} on $u$, it is enough to show that $u_x, \ u_{xx} \in L^\infty(B^+_{R_2}(z_0))$.

Since $f_x, f_{xx} \in L^2(B^+_{R_0}(z_0),y^{\beta-1})$ by hypothesis, Proposition \ref{prop:VarEqn_H1Dkx} (with $k=1,2$) implies that $u_x, \ u_{xx}\in H^1_0(B^+_{R_0/2}(z_0), \fw)$ and that they obey
\begin{align*}
\fa(u_x, v) &= (f_x, v)_{L^2(B^+_{R_0/2}(z_0), \fw)},
\\
\fa(u_{xx}, v) &= (f_{xx}, v)_{L^2(B^+_{R_0/2}(z_0), \fw)},
\quad \forall v \in H^1_0(\underline{B}^+_{R_0/2}(z_0), \fw).
\end{align*}
Also, because $f_x, f_{xx} \in L^p(B^+_{R_0}(z_0),y^{\beta-1})$ by hypothesis, we can apply Theorem \ref{thm:MainSupremumEstimatesInterior}
to the preceding variational equations to give a positive constant $R_3=R_3(R_0)<R_0/2$, such that $u_x, u_{xx} \in L^\infty(B^+_{R_3}(z_0))$ and
\begin{align}
\label{eq:LinftyBoundux}
\|u_x\|_{L^\infty(B^+_{R_3}(z_0))} &\leq C\left(\|f_x\|_{L^p(B^+_{R_0/2}(z_0), y^{\beta-1})} + \|u_x\|_{L^2(B^+_{R_0/2}(z_0), y^{\beta-1})}\right),
\\
\label{eq:LinftyBounduxx}
\|u_{xx}\|_{L^\infty(B^+_{R_3}(z_0))} &\leq C\left(\|f_{xx}\|_{L^p(B^+_{R_0/2}(z_0), y^{\beta-1})} + \|u_{xx}\|_{L^2(B^+_{R_0/2}(z_0), y^{\beta-1})}\right).
\end{align}
We now choose $R_2=R_3$ and observe that condition \eqref{eq:SufficientConditionMainContinuityInterioruy} holds and the estimate \eqref{eq:HolderContinuity_uyprelim} is justified with
\begin{equation}
\label{eq:uxuxxLpLinfty}
\begin{aligned}
\|u_x\|_{L^p(B^+_{R_3}(z_0), y^{\beta-1})} &\leq C\|u_x\|_{L^\infty(B^+_{R_3}(z_0))},
\\
\|u_{xx}\|_{L^p(B^+_{R_3}(z_0), y^{\beta-1})} &\leq C\|u_{xx}\|_{L^\infty(B^+_{R_3}(z_0))},
\end{aligned}
\end{equation}
where $C=C(\beta,p,R_3)=(\int_{B^+_{R_3}(z_0)}y^{\beta-1}\,dx\,dy)^{1/p}$. The definition \eqref{eq:HkWeightedSobolevSpaceNormPowery} of $\sH^3(\sO, \fw)$ and the definition \eqref{eq:HestonWeight} of $\fw$ imply that
\begin{align*}
{}&\|u_y\|_{L^2(B^+_{R_0/2}(z_0), y^{\beta-1})} + \|u_x\|_{L^2(B^+_{R_0/2}(z_0), y^{\beta-1})} + \|u_{xx}\|_{L^2(B^+_{R_0/2}(z_0), y^{\beta-1})}
\\
&\leq Ce^{\frac{\gamma}{2}|x_0|}\left(\|u_y\|_{L^2(B^+_{R_0/2}(z_0), \fw)} + \|u_x\|_{L^2(B^+_{R_0/2}(z_0), \fw)} + \|u_{xx}\|_{L^2(B^+_{R_0/2}(z_0), \fw)}\right)
\\
&\leq Ce^{\frac{\gamma}{2}|x_0|}\|u\|_{\sH^3(B^+_{R_0/2}(z_0), \fw)},
\end{align*}
where the positive constant, $C=C(\Lambda,\nu_0,R_0)$, and the factor, $e^{\frac{\gamma}{2}|x_0|}$, arise just as in the proof of Proposition \ref{prop:HolderContinuity_ux}. Thus, by Theorem \ref{thm:H3SobolevRegularityInterior},
\begin{align*}
{}&\|u_y\|_{L^2(B^+_{R_0/2}(z_0), y^{\beta-1})} + \|u_x\|_{L^2(B^+_{R_0/2}(z_0), y^{\beta-1})} + \|u_{xx}\|_{L^2(B^+_{R_0/2}(z_0), y^{\beta-1})}
\\
&\leq Ce^{\frac{\gamma}{2}|x_0|}\left(\|f\|_{W^{1,2}(B^+_{R_0}(z_0), \fw)} + \|u\|_{L^2(B^+_{R_0}(z_0), \fw)}\right)
\\
&\leq Ce^{\frac{\gamma}{2}|x_0|}e^{-\frac{\gamma}{2}|x_0|}\left(\|f\|_{W^{1,2}(B^+_{R_0}(z_0), y^{\beta-1})} + \|u\|_{L^2(B^+_{R_0}(z_0), y^{\beta-1})}\right)
\\
&\leq C\left(\|f\|_{W^{1,p}(B^+_{R_0}(z_0), y^{\beta-1})} + \|u\|_{L^2(B^+_{R_0}(z_0), y^{\beta-1})}\right),
\end{align*}
where the positive constant, $C=C(\Lambda,\nu_0,p,R_0))$, and the factor, $e^{-\frac{\gamma}{2}|x_0|}$, arise just as in the proof of Proposition \ref{prop:HolderContinuity_ux}. The estimate \eqref{eq:HolderContinuity_uy} is obtained by combining the preceding inequality with \eqref{eq:HolderContinuity_uyprelim}, \eqref{eq:LinftyBoundux}, and \eqref{eq:LinftyBounduxx}, and \eqref{eq:uxuxxLpLinfty}.
\end{proof}

We may combine Propositions \ref{prop:HolderContinuity_ux} and \ref{prop:HolderContinuity_uy} to give

\begin{prop} [Interior $C^{\alpha}_s$ H\"older continuity of $Du$ for a solution $u$ to the variational equation]
\label{prop:HolderContinuity_Du}
Let $p > \max\{4, 3+\beta\}$ and let $R_0$ be a positive constant. Then there are positive constants, $R_1=R_1(R_0) < R_0$ and $C=C(\Lambda,\nu_0,p,R_0)$ and $\alpha=\alpha(\Lambda,\nu_0,p,R_0)\in(0,1)$, such that the following holds. Let $\sO\subseteqq\HH$ be a domain. If $f \in L^2(\sO,\fw)$ and $u \in H^1(\sO,\fw)$ satisfies the variational equation \eqref{eq:HestonVariationalEquation}, and $z_0 \in \partial_0\sO$ is such that
$
B_{R_0}(z_0) \cap \HH \subset \sO,
$
and $f$ obeys \eqref{eq:Lsfxxfycondition_ball}, then
$
u_x, \ u_y \in C^\alpha_s(\bar B_{R_1}^+(z_0)),
$
and
\begin{equation}
\label{eq:HolderContinuity_Du}
\begin{aligned}
\|Du\|_{C^\alpha_s(\bar B_{R_1}^+(z_0))}
&\leq
C\left( \|f_{xx}\|_{L^p(B^+_{R_0}(z_0),y^{\beta-1})} + \|f\|_{W^{1,p}(B^+_{R_0}(z_0),y^{\beta-1})}  \right.
\\
&\qquad + \left. \|u\|_{L^2(B^+_{R_0}(z_0),y^{\beta-1})} \right).
\end{aligned}
\end{equation}
\end{prop}

\begin{proof}
The conclusion $u_x,\ u_y \in C^\alpha_s(\bar B_{R_1}^+(z_0))$ and estimate \eqref{eq:HolderContinuity_Du} follow from Propositions \ref{prop:HolderContinuity_ux} and \ref{prop:HolderContinuity_uy}.
\end{proof}

Finally, we may combine Theorems \ref{thm:MainContinuityInteriorL2RHSu} and \ref{thm:H3SobolevRegularityInterior} and Proposition \ref{prop:HolderContinuity_Du} to give

\begin{thm} [Interior $C^{1,\alpha}_s$ H\"older continuity for a solution $u$ to the variational equation]
\label{thm:HolderContinuity_Du}
Let $p > \max\{4, 3+\beta\}$ and let $R_0$ be a positive constant. Then there are positive constants, $R_1=R_1(R_0) < R_0$ and $C=C(\Lambda,\nu_0,p,R_0)$ and $\alpha=\alpha(\Lambda,\nu_0,p,R_0)\in(0,1)$ such that the following holds. Let $\sO\subseteqq\HH$ be a domain. If $f \in L^2(\sO,\fw)$, and $u \in H^1(\sO,\fw)$ satisfies the variational equation \eqref{eq:HestonVariationalEquation}, and $z_0 \in \partial_0\sO$ is such that
$
B_{R_0}(z_0) \cap \HH \subset \sO,
$
and $f$ obeys \eqref{eq:Lsfcondition_ball} and \eqref{eq:Lsfxxfycondition_ball}, then
$
u \in C^{1,\alpha}_s(\bar B_{R_1}^+(z_0)),
$
and
\begin{equation}
\label{eq:HolderSobolevContinuity_C1alphasu}
\begin{aligned}
\|u\|_{C^{1,\alpha}_s(\bar B_{R_1}^+(z_0))}
&\leq
C\left(\|f_{xx}\|_{L^p(B^+_{R_0}(z_0),y^{\beta-1})} + \|f\|_{W^{1,p}(B^+_{R_0}(z_0),y^{\beta-1})} \right.
\\
&\qquad + \left. \|u\|_{L^2(B^+_{R_0}(z_0),y^{\beta-1})}\right).
\end{aligned}
\end{equation}
\end{thm}

\begin{proof}
Since $f, \ f_x, \ f_y \in L^p(B^+_{R_0}(z_0), y^{\beta-1}) \subset L^2(B^+_{R_0}(z_0), \fw)$, Theorem \ref{thm:H3SobolevRegularityInterior} implies that $u \in \sH^3(B_{R_2}^+(z_0),\fw)$, for any $0 < R_2 < R_0$. By applying Theorem \ref{thm:MainContinuityInteriorL2RHSu} and Proposition \ref{prop:HolderContinuity_Du}, with $R_2$ in place of $R_0$, we obtain $u \in C^{1,\alpha}_s(\bar B_{R_1}^+(z_0))$ for some $R_1<R_2$ and, say, $R_2=R_0/2$. The inequality \eqref{eq:HolderSobolevContinuity_C1alphasu} is obtained by combining \eqref{eq:MainContinuity3L2RHSu}
and \eqref{eq:HolderContinuity_Du}.
\end{proof}

\subsection{H\"older regularity for higher-order derivatives of solutions to the variational equation}
\label{subsec:HolderRegularityHigherOrder}
We first give an extension of Proposition \ref{prop:HolderContinuity_ux} from the case $k=1$ to arbitrary $k\geq 1$.

\begin{prop} [Interior $C^{\alpha}_s$ H\"older continuity of higher-order derivatives with respect to $x$ for a solution to the variational equation]
\label{prop:HolderContinuity_ukx}
Let $p>\max\{4, 2+\beta\}$, let $R_0$ be a positive constant, and let $k\geq 1$ be an integer. Then there are positive constants, $R_1=R_1(R_0) < R_0$, and $C=C(\Lambda,\nu_0,k,p,R_0)$, and $\alpha=\alpha(\Lambda,\nu_0,p,R_0)\in(0,1)$ such that the following holds. Let $\sO\subseteqq\HH$ be a domain. If $f \in L^2(\sO,\fw)$, and $u \in H^1(\sO,\fw)$ satisfies the variational equation \eqref{eq:HestonVariationalEquation}, and $z_0 \in \partial_0\sO$ is such that
$
B_{R_0}(z_0) \cap \HH \subset \sO,
$
and
$D_x^j f \in L^p(B^+_{R_0}(z_0),y^{\beta-1})$, for all $1 \leq j \leq k$, then
$
D_x^k u \in C^\alpha_s(\bar B_{R_1}^+(z_0)),
$
and
\begin{equation}
\label{eq:HolderContinuity_ukx}
\|D_x^k u\|_{C^\alpha_s(\bar B_{R_1}^+(z_0))}
\leq
C\left(\sum_{j=0}^k\|D_x^j f\|_{L^p(B^+_{R_0}(z_0),y^{\beta-1})}  + \|u\|_{L^2(B^+_{R_0}(z_0),y^{\beta-1})}\right).
\end{equation}
\end{prop}

\begin{proof}
The argument is similar to the proof of Proposition \ref{prop:HolderContinuity_ux}. By hypothesis, $D_x^j f \in L^2(B^+_{R_0}(z_0),\fw)$, $1\leq j\leq k$, since $p>2$ and so Proposition \ref{prop:VarEqn_H1Dkx} implies that $D_x^k u \in H^1(B^+_{R_2}(z_0),\fw)$, for any $R_2=R_2(R_0)$, to be determined, in the range $0<R_2<R_0$, and that $D_x^k u$ obeys
$$
\fa(D_x^k u, v) = (D_x^k f, v)_{L^2(B^+_{R_2}(z_0), \fw)}, \quad \forall v \in H^1_0(\underline{B}^+_{R_2}(z_0), \fw),
$$
and, for a positive constant, $C=C(\Lambda,\nu_0,k,R_0)$,
$$
\|D_x^k u\|_{H^1(B^+_{R_2}(z_0),\fw)} \leq C\left(\sum_{j=0}^k\|D_x^j f\|_{L^2(B^+_{R_0}(z_0), \fw)} + \|u\|_{L^2(B^+_{R_0}(z_0),\fw)}\right).
$$
The conclusion $D_x^k u \in C^\alpha_s(\bar B_{R_1}^+(z_0))$ and estimate, with $C=C(\Lambda,\nu_0,p,R_0)$, and\break
$\alpha=\alpha(\Lambda,\nu_0,p,R_0)\in(0,1)$, and
$R_1=R_1(R_2)<R_2$ (recall that $R_2=R_2(R_0)$),
$$
\|D_x^k u\|_{C^\alpha_s(\bar B_{R_1}^+(z_0))} \leq C\left(\|D_x^k f\|_{L^p(B^+_{R_2}(z_0),y^{\beta-1})} + \|D_x^k u\|_{L^2(B^+_{R_2}(z_0),y^{\beta-1})}\right),
$$
follow by applying Theorem \ref{thm:MainContinuityInteriorL2RHSu} to the variational equation for $D_x^k u$ on $B^+_{R_2}(z_0)$. As in the proof of Proposition \ref{prop:HolderContinuity_ux}, we have
\begin{align*}
\|D_x^k u\|_{L^2(B^+_{R_2}(z_0),y^{\beta-1})} &\leq Ce^{\frac{\gamma}{2}|x_0|}\|D_x^k u\|_{L^2(B^+_{R_2}(z_0),\fw)}
\\
&\leq Ce^{\frac{\gamma}{2}|x_0|}\|D_x^k u\|_{H^1(B^+_{R_2}(z_0),\fw)},
\end{align*}
where $C=C(\Lambda,\nu_0,R_2)$ is a positive constant. We may choose $R_2=R_0/2$ and observe that, by combining the preceding inequalities, we obtain \eqref{eq:HolderContinuity_ukx}, with $C=C(\Lambda,\nu_0,k,p,R_0)$, noting that the factor, $e^{\frac{\gamma}{2}|x_0|}$, cancels just as in the proof of Proposition \ref{prop:HolderContinuity_ux}.
\end{proof}

The extension of Propositions \ref{prop:HolderContinuity_ukx} and \ref{prop:HolderContinuity_uy} to the case of derivatives of the form $D_x^{k-m}D_y^m u$, when $1\leq m\leq k$, is best illustrated by an example when $k=2$ and $m=0,1,2$. The case $k=2, m=0$ is given by Proposition \ref{prop:HolderContinuity_ukx}, so
$$
\|u_{xx}\|_{C^\alpha_s(\bar B_{R_1}^+(z_0))}
\leq
C\left(\sum_{j=0}^2\|D_x^j f\|_{L^p(B^+_{R_0}(z_0),y^{\beta-1})} + \|u\|_{L^2(B^+_{R_0}(z_0),y^{\beta-1})}\right).
$$
For $k=2, m=1$, the pattern of proof of Proposition \ref{prop:HolderContinuity_uy} shows that a $C^\alpha_s$ estimate of $u_{xy}$ requires an $L^p$-bound on $f_{xy}$, $u_{xx}$, and $u_{xxx}$, thus an additional $L^\infty$-bound on $u_{xxx}$, and hence an additional $L^p$ bound on $f_{xxx}$, to give
\begin{align*}
\|u_{xy}\|_{C^\alpha_s(\bar B_{R_1}^+(z_0))}
&\leq
C\left(\sum_{i=0}^3\|D_x^i f\|_{L^p(B^+_{R_0}(z_0),y^{\beta-1})} + \|f_{xy}\|_{L^p(B^+_{R_0}(z_0),y^{\beta-1})} \right.
\\
&\qquad + \left. \|u\|_{L^2(B^+_{R_0}(z_0),y^{\beta-1})}\right),
\end{align*}
and thus the following will suffice,
$$
\|u_{xy}\|_{C^\alpha_s(\bar B_{R_1}^+(z_0))}
\leq
C\left(\|f\|_{W^{3,p}(B^+_{R_0}(z_0),y^{\beta-1})} + \|u\|_{L^2(B^+_{R_0}(z_0),y^{\beta-1})}\right).
$$
For $k=2, m=2$, the pattern of proof of Proposition \ref{prop:HolderContinuity_uy} shows that a $C^\alpha_s$ estimate of $u_{yy}$ requires an $L^p$-bound on $f_{yy}$, $u_{xy}$, and $u_{xxy}$, thus $L^\infty$-bounds on $u_{xy}$ and $u_{xxy}$, hence additional $L^p$ bounds on $f_{xxy}$, $f_{xxxx}$, to give
\begin{align*}
\|u_{yy}\|_{C^\alpha_s(\bar B_{R_1}^+(z_0))}
&\leq
C\left(\sum_{j=0}^4\|D_x^j f\|_{L^p(B^+_{R_0}(z_0),y^{\beta-1})} + \sum_{i=1}^2\|D_x^i f_y\|_{L^p(B^+_{R_0}(z_0),y^{\beta-1})} \right.
\\
&\qquad + \left. \|f_{yy}\|_{L^p(B^+_{R_0}(z_0),y^{\beta-1})} + \|u\|_{L^2(B^+_{R_0}(z_0),y^{\beta-1})}\right),
\end{align*}
and thus the following will suffice,
$$
\|u_{yy}\|_{C^\alpha_s(\bar B_{R_1}^+(z_0))}
\leq
C\left(\|f\|_{W^{4,p}(B^+_{R_0}(z_0),y^{\beta-1})} + \|u\|_{L^2(B^+_{R_0}(z_0),y^{\beta-1})}\right).
$$
The preceding examples motivate the statement of the following combined extension of Propositions \ref{prop:HolderContinuity_uy} and \ref{prop:HolderContinuity_ukx}.

\begin{prop} [Interior $C^{\alpha}_s$ H\"older continuity of higher-order derivatives of a solution to the variational equation]
\label{prop:HolderContinuity_ukxy}
Let $R_0$ be a positive constant, let $m, k$ be integers with $k\geq 1$ and $0\leq m\leq k$, and let $p > \max\{4, 2+m+\beta\}$. Then there are positive constants, $R_1=R_1(m,R_0) < R_0$, and $C=C(\Lambda,\nu_0,k,m,p,R_0)$, and $\alpha=\alpha(\Lambda,\nu_0,m,p,R_0)\in(0,1)$, such that the following holds. Let $\sO\subseteqq\HH$ be a domain. If $f \in L^2(\sO,\fw)$, and $u \in H^1(\sO,\fw)$ satisfies the variational equation \eqref{eq:HestonVariationalEquation}, and $z_0 \in \partial_0\sO$ is such that
$
B_{R_0}(z_0) \cap \HH \subset \sO,
$
and
$
f \in W^{k+m,p}(B^+_{R_0}(z_0),y^{\beta-1}),
$
then
$
D_x^{k-m}D_y^m u \in C^\alpha_s(\bar B_{R_1}^+(z_0)),
$
and
\begin{equation}
\label{eq:HolderContinuity_ukxy}
\|D_x^{k-m}D_y^m u\|_{C^\alpha_s(\bar B_{R_1}^+(z_0))}
\leq
C\left(\|f\|_{W^{k+m,p}(B^+_{R_0}(z_0),y^{\beta-1})} + \|u\|_{L^2(B^+_{R_0}(z_0),y^{\beta-1})}\right).
\end{equation}
\end{prop}

\begin{proof}
For arbitrary $\ell\geq 1$ and $f \in W^{\ell,p}(B^+_{R_0}(z_0),y^{\beta-1})$, Proposition \ref{prop:HolderContinuity_ukx} already implies that
$$
\|D^\ell_x u\|_{C^\alpha_s(\bar B_{R_3}^+(z_0))}
\leq
C\left(\|f\|_{W^{\ell,p}(B^+_{R_0}(z_0),y^{\beta-1})} + \|u\|_{L^2(B^+_{R_0}(z_0),y^{\beta-1})}\right),
$$
for some $R_2=R_2(R_0)$, and hence that \eqref{eq:HolderContinuity_ukxy} holds when $m=0$, so we may assume without loss of generality that $m\geq 1$ in our proof of Proposition \ref{prop:HolderContinuity_ukxy}. Therefore, to establish \eqref{eq:HolderContinuity_ukxy}, it suffices to consider the inductive step $(k,m-1)\implies (k,m)$ (one extra derivative with respect to $y$), assuming
\begin{equation}
\label{eq:HolderSobolevContinuityDkuInductiveHypothesis}
\begin{gathered}
\|D^{\ell-n}_xD_y^n u\|_{C^\alpha_s(\bar B_{R_2}^+(z_0))}
\leq
C\left(\|f\|_{W^{\ell+n,p}(B^+_{R_0}(z_0),y^{\beta-1})} + \|u\|_{L^2(B^+_{R_0}(z_0),y^{\beta-1})}\right),
\\
\hbox{for all }\ell \geq 1, \quad\hbox{all $n$ such that } 0\leq n \leq m-1, \quad\hbox{and } 1\leq m \leq k,
\end{gathered}
\end{equation}
where $R_2=R_2(m-1,R_0)$ (we point out the origin of the dependence on $m$ further along in the proof).
The proof of this inductive step follows the pattern of proof of Proposition \ref{prop:HolderContinuity_uy}.

By our hypotheses on $f$, Theorem \ref{thm:HkSobolevRegularityInterior} implies that
$$
u \in \sH^{k+2+m}(B^+_{R_2}(z_0),\fw) \subset \sH^{k+3}(B^+_{R_2}(z_0),\fw),
$$
(since we assume $m\geq 1$ for the inductive step), for any $R_2$ in the range $0<R_2<R_0$ and to be determined, and that
\begin{equation}
\label{eq:Hk+3Boundu}
\|u\|_{\sH^{k+3}(B^+_{R_2}(z_0),\fw)} \leq C\left(\|f\|_{W^{k+1,2}(B^+_{R_0}(z_0),\fw)} + \|u\|_{L^2(B^+_{R_0}(z_0), \fw)}\right).
\end{equation}
We have $D_x^{k-m}D_y^m u \in H^1(B^+_{R_2}(z_0),\fw_m)$ by Definition \ref{defn:HkWeightedSobolevSpaceNormPowery} of $\sH^{k+3}(B^+_{R_2}(z_0),\fw)$, since $u \in \sH^{k+3}(B^+_{R_2}(z_0),\fw)$ implies
\begin{align*}
D_x^{k+1-m}D_y^m u, \quad D_x^{k-m}D_y^m u
&\in \begin{cases} L^2(B^+_{R_2}(z_0),\fw_{m-1}), &2\leq m\leq k, \\ L^2(B^+_{R_2}(z_0),\fw), &m=1, \end{cases}
\\
D_x^{k-m}D_y^{m+1} u &\in L^2(B^+_{R_2}(z_0),\fw_m), \quad 1\leq m\leq k,
\end{align*}
and we have $D_x^k u \in H^1(B^+_{R_2}(z_0), \fw)$ by Proposition \ref{prop:VarEqn_H1Dkx}. Proposition \ref{prop:VarEqn_Dkxy} then ensures that $D_x^{k-m}D_y^mu$ obeys the variational equation on $B^+_{R_2}(z_0)$,
\begin{align*}
\fa_m(D_x^{k-m}D_y^m u, v) &= (D_x^{k-m}D_y^m f, v)_{L^2(B^+_{R_2}(z_0),\fw_m)}
\\
&\quad - m(BD_x^{k-m}D_y^{m-1} u,v)_{L^2(B^+_{R_2}(z_0),\fw_m)},
\end{align*}
for all $v \in H^1_0(\underline{B}^+_{R_2}(z_0),\fw_m)$. By hypothesis, $D_x^{k-m}D_y^m f \in L^p(B^+_{R_2}(z_0),y^{\beta+m-1})$, and \emph{provided} we also know that $BD_x^{k-m}D_y^{m-1} \in L^p(B^+_{R_2}(z_0),y^{\beta+m-1})$, that is,
\begin{equation}
\label{eq:SufficientConditionMainContinuityInteriorDk+2u}
\begin{aligned}
D_x^{k+1-m}D_y^{m-1} u &\in L^p(B^+_{R_2}(z_0),y^{\beta+m-1}),
\\
D_x^{k+2-m}D_y^{m-1} u &\in L^p(B^+_{R_2}(z_0),y^{\beta+m-1}), \quad 1\leq m\leq k,
\end{aligned}
\end{equation}
we can apply Theorem \ref{thm:MainContinuityInteriorL2RHSu} to the variational equation for $D_x^{k-m}D_y^m u$ and conclude that $D_x^{k-m}D_y^m u \in C^\alpha_s(\bar B_{R_1}^+(z_0))$, for some\footnote{The dependence on $m$ appears in this step.} $R_1=R_1(R_2)$, obeying $R_1<R_2$, and
\begin{equation}
\label{eq:CalphasBoundDkmxmyu_prelim}
\begin{aligned}
\|D_x^{k-m}D_y^m u\|_{C^\alpha_s(\bar B_{R_1}^+(z_0))} &\leq C\left(\|D_x^{k-m}D_y^m f\|_{L^p(B^+_{R_2}(z_0),y^{\beta+m-1})} \right.
\\
&\qquad + \sum_{j=1}^2\|D_x^{k+j-m}D_y^{m-1} u\|_{L^p(B^+_{R_2}(z_0), y^{\beta+m-1})}
\\
&\qquad + \left. \|D_x^{k-m}D_y^m u\|_{L^2(B^+_{R_2}(z_0), y^{\beta+m-1})} \right), \quad 1\leq m \leq k.
\end{aligned}
\end{equation}
It is important to note that the H\"older exponent, $\alpha$, in \eqref{eq:CalphasBoundDkmxmyu_prelim} depends on the coefficients defining the bilinear map, $\fa_m$, that is, on the coefficients of $A_m$ and thus on the coefficients of $A$ and on $m$, the number of derivatives with respect to $y$, and this is why we write $\alpha=\alpha(\Lambda,\nu_0,m,p,R_0)$ in the statement of Proposition \ref{prop:HolderContinuity_ukxy}. The integrability conditions \eqref{eq:SufficientConditionMainContinuityInteriorDk+2u} are implied by
\begin{equation}
\label{eq:LinftyConditionMainContinuityInteriorDk+2u}
\begin{aligned}
D_x^{k+1-m}D_y^{m-1} u &\in L^\infty(B^+_{R_2}(z_0)),
\\
D_x^{k+2-m}D_y^{m-1} u &\in L^\infty(B^+_{R_2}(z_0)), \quad 1\leq m\leq k.
\end{aligned}
\end{equation}
Note that $D_x^{k+1-m}D_y^{m-1} u = D_x^{k-(m-1)}D_y^{m-1} u$ and $D_x^{k+2-m}D_y^{m-1} u = D_x^{k+1-(m-1)}D_y^{m-1} u$ and the properties \eqref{eq:LinftyConditionMainContinuityInteriorDk+2u} hold
by the inductive hypothesis \eqref{eq:HolderSobolevContinuityDkuInductiveHypothesis}. Therefore, the inductive hypothesis \eqref{eq:HolderSobolevContinuityDkuInductiveHypothesis} gives
\begin{gather}
\label{eq:LinftyInteriorDkuBound}
\begin{aligned}
{}&\|D_x^{k+1-m}D_y^{m-1} u\|_{L^\infty(B^+_{R_2}(z_0))}
\\
&\quad \leq C\left(\|f\|_{W^{k+m-1,p}(B^+_{R_0}(z_0),y^{\beta-1})} + \|u\|_{L^2(B^+_{R_0}(z_0),y^{\beta-1})}\right),
\end{aligned}
\\
\label{eq:LinftyInteriorDk+1uBound}
\begin{aligned}
{}&\|D_x^{k+2-m}D_y^{m-1} u\|_{L^\infty(B^+_{R_2}(z_0))}
\\
&\quad \leq C\left(\|f\|_{W^{k+m,p}(B^+_{R_0}(z_0),y^{\beta-1})} + \|u\|_{L^2(B^+_{R_0}(z_0),y^{\beta-1})}\right),
\end{aligned}
\end{gather}
for $1\leq m\leq k$. Hence, the integrability conditions for the derivatives of $u$ in \eqref{eq:SufficientConditionMainContinuityInteriorDk+2u} are satisfied, and the estimate \eqref{eq:CalphasBoundDkmxmyu_prelim} holds.

Applying the Definition \ref{defn:HkWeightedSobolevSpaceNormPowery} of $\sH^{k+3}(B^+_{R_2}(z_0), \fw)$ and the $L^\infty$ estimates \eqref{eq:LinftyInteriorDkuBound} and \eqref{eq:LinftyInteriorDk+1uBound} for the derivatives of $u$, the inequality \eqref{eq:CalphasBoundDkmxmyu_prelim} yields
\begin{equation}
\label{eq:CalphasBoundDkmxmyu}
\|D_x^{k-m}D_y^m u\|_{C^\alpha_s(\bar B_{R_1}^+(z_0))} \leq C\left(\|f\|_{W^{k+m,p}(B^+_{R_0}(z_0),y^{\beta-1})}
+ \|u\|_{\sH^{k+3}(B^+_{R_2}(z_0), \fw)} \right).
\end{equation}
Combining inequalities \eqref{eq:Hk+3Boundu} and \eqref{eq:CalphasBoundDkmxmyu} completes the proof.
\end{proof}

Theorem \ref{thm:HolderContinuity_Dku_Interior} now follows easily, extending Theorem \ref{thm:HolderContinuity_Du} from the case $k=1$ to any $k\geq 1$.

\begin{proof}[Proof of Theorem \ref{thm:HolderContinuity_Dku_Interior}]
When $k=0$ or $1$, then the conclusion follows from Theorems \ref{thm:MainContinuityInteriorL2RHSu} or \ref{thm:HolderContinuity_Du}, respectively, so we may assume that $k\geq 2$ and, by induction, that the conclusion holds for $k-1$ in place of $k$. Since
\begin{equation}
\label{eq:InductiveCkalphasBound}
\|u\|_{C^{k,\alpha}_s(\bar B_{R_1}^+(z_0))} = \|D^ku\|_{C^\alpha_s(\bar B_{R_1}^+(z_0))} + \|u\|_{C^{k-1,\alpha}_s(\bar B_{R_1}^+(z_0))},
\end{equation}
by Definition \ref{defn:DHspaces}, it suffices to show that $D^{k-m}_xD_y^m u \in C^\alpha_s(\bar B_{R_1}^+(z_0))$, for $0\leq m\leq k$, but this inclusion and estimate are given by Proposition \ref{prop:HolderContinuity_ukxy}.
\end{proof}

We may combine Theorem \ref{thm:HolderContinuity_Dku_Interior} with standard results from \cite{GilbargTrudinger} for linear, second-order, elliptic differential equations to give a \emph{weak} version of Theorem \ref{thm:CkalphasHolderContinuityDomain} which will, nonetheless, provide a useful stepping stone to the proof of Theorem \ref{thm:CkalphasHolderContinuityDomain} itself. Although their statements appear similar, Proposition \ref{prop:CkalphasHolderContinuityDomain_Precompact} is nevertheless strictly weaker than Theorem \ref{thm:CkalphasHolderContinuityDomain}, despite the more relaxed hypothesis on $f$ because, in the former case, $\alpha=\alpha(\Lambda,\nu_0,d_1,k,p)$ depends on the choice of precompact subdomain, $\sO''\Subset\underline\sO$, through the constant $d_1$ whereas in the latter case, $\alpha=\alpha(\Lambda,\nu_0,k,p)$ is \emph{independent} of the choice of precompact subdomain, $\sO''\Subset\underline\sO$.

\begin{prop} [Interior $C^{k,\alpha}_s$ regularity on subdomains]
\label{prop:CkalphasHolderContinuityDomain_Precompact}
Let $k\geq 0$ be an integer, let $d_1<\Upsilon$ be positive constants, and let $p>\max\{4,2+k+\beta\}$. Then there are positive constants
$\alpha=\alpha(\Lambda,\nu_0,d_1,k,p)\in(0,1)$ and $C=C(\Lambda,\nu_0,k,d_1,\Upsilon,p)$ such that the following holds. If $f \in L^2(\sO,\fw)$ and $u \in H^1(\sO,\fw)$ is a solution to the variational equation \eqref{eq:HestonVariationalEquation}, and $f \in W^{2k,p}_{\loc}(\underline\sO,\fw)$, and $\sO''\subset\sO$ is a subdomain such that
$\sO''\Subset\underline\sO$ with $\sO''\subset(-\Upsilon,\Upsilon)\times(0,\Upsilon)$ and $\dist(\partial_1\sO'', \partial_1\sO) \geq d_1$, then
$$
u \in C^{k,\alpha}_s(\underline\sO'')\cap C^{\alpha}_s(\bar\sO'').
$$
Moreover, $u$ solves \eqref{eq:IntroBoundaryValueProblem} on $\sO''$ and if $\sO'\subset\sO''$ is a subdomain with $\sO'\Subset\underline\sO''$ and $\dist(\partial_1\sO', \partial_1\sO'') \geq d_1$, then
\begin{equation}
\label{eq:CkalphasHolderContinuityDomain_Precompact}
\|u\|_{C^{k,\alpha}_s(\bar\sO')} \leq C\left(\|f\|_{W^{2k,p}(\sO'',\fw)} + \|u\|_{L^2(\sO'',\fw)}\right).
\end{equation}
\end{prop}

\begin{proof}
Choose $R_0=d_1/2$ and let $R_1=R_1(k,R_0)<R_0$ be defined by Theorem \ref{thm:HolderContinuity_Dku_Interior}. Since $\sO''\subset(-\Upsilon,\Upsilon)\times(0,\Upsilon)$ and the rectangle is covered by balls, $B_{R_1}(z_l)\subset\RR^2$, with a finite sequence of centers $\{z_l\}\subset [-\Upsilon,\Upsilon]\times[0,\Upsilon]$ on rectangular grid with square cells of width $R_1$. We may now choose finite subsequences of points, $\{z_{0,i}\}\subset\{z_l\}\cap\partial_0\sO$ and $\{z_{1,j}\}\subset\{z_l\}\cap\sO$, such that
$$
\bar\sO'' \subset \bigcup_{i,j} B_{R_1}^+(z_{0,i}) \cup B_{R_1}(z_{1,j}),
$$
where $\HH\cap B_{R_0}(z_{0,i}) \subset \sO$ for all $i$ and $B_{R_0}(z_{1,j})\Subset\sO$ for all $j$. Let
$\alpha=\alpha(\Lambda,\nu_0,k,p,R_0)=\alpha(\Lambda,\nu_0,d_1,k,p)\in(0,1)$ be the constant defined by Theorem \ref{thm:HolderContinuity_Dku_Interior}.

According to \cite[Theorem 9.19]{GilbargTrudinger}, for each $r>0$ and ball $B_r\Subset\sO$ of radius $r$ and integer $k\geq 0$, we have $u \in W^{k+2,p}(B_r)$, since $f\in W^{2k,p}_{\loc}(\underline\sO)$ by hypothesis and, in particular, $f\in W^{k,p}_{\loc}(\sO)$. By the Sobolev embedding \cite[Theorem 5.4 (C${}'$)]{Adams_1975}, since $p>2$, there is a continuous embedding, $W^{k+2,p}(B_r)\hookrightarrow C^{k+1}(\bar B_r)$. Thus, $u \in C^{k,\alpha}(\sO)$ because, a fortiori,
\begin{equation}
\label{eq:uCk+1alphaB}
u \in C^{k+1}(\bar B_r), \quad \forall B_r\Subset\sO.
\end{equation}
Moreover, letting $B_{r/2}\subset B_{3r/4} \subset B_r$ denote concentric balls,
\begin{equation}
\label{eq:uCk+1lambdaBestimate}
\|u\|_{C^{k+1}(\bar B_{r/2})} \leq C\left(\|f\|_{W^{k,p}(B_r)} + \|u\|_{L^2(B_r)}\right),
\end{equation}
where $C=C(\Lambda,\nu_0,k,p,r)$ is a positive constant, since\footnote{The estimate \eqref{eq:uCk+1lambdaBestimate} also follows from \cite[Corollary 6.3 and Theorem 6.17]{GilbargTrudinger}.}
\begin{align*}
\|u\|_{C^{k+1}(\bar B_{r/2})} &\leq C\|u\|_{W^{k+2,p}(B_{r/2})}
\quad\hbox{(by $p>2$ and \cite[Theorem 5.4 (C${}'$)]{Adams_1975})}
\\
&\leq C\left(\|f\|_{W^{k,p}(B_{3r/4})} + \|u\|_{L^p(B_{3r/4})}\right) \quad\hbox{(by \cite[Theorems 9.11 and 9.19]{GilbargTrudinger})}
\\
&\leq C\left(\|f\|_{W^{k,p}(B_{3r/4})} + \|u\|_{W^{1,2}(B_{3r/4})}\right) \quad\hbox{(by $p>2$ and \cite[Theorem 5.4 (B)]{Adams_1975})}
\\
&\leq C\left(\|f\|_{W^{k,p}(B_r)} + \|u\|_{L^2(B_r)}\right) \quad\hbox{(by $p>2$ and \cite[Exercise 8.2]{GilbargTrudinger})}.
\end{align*}
The conclusion $u \in C^{k,\alpha}_s(\underline\sO'')\cap C(\bar\sO'')$ now follows from Theorem \ref{thm:HolderContinuity_Dku_Interior}.

For the estimate \eqref{eq:CkalphasHolderContinuityDomain_Precompact} of $u$ over $\sO'\Subset\underline\sO''$, observe that, since $\dist(\partial_1\sO', \partial_1\sO'')\geq d_1$, the closure $\bar\sO'$ is covered by finitely many half-balls, $B_{R_0}^+(z_{0,i})$ with $\HH\cap B_{R_0}(z_{0,i})\Subset\underline\sO''$, and balls, $B_{R_0}(z_{1,j})\Subset\underline\sO''$, where the total number of balls and half-balls of radius $R_0=R_0(d_1)$ is determined by $d_1$ and $\Upsilon$. We obtain \eqref{eq:CkalphasHolderContinuityDomain_Precompact} by applying \eqref{eq:HolderSobolevContinuity_Dku} to each half-ball $B_{R_0}^+(z_{0,i}) \subset \sO$ and applying \eqref{eq:uCk+1lambdaBestimate} to each ball $B_{R_0}(z_{1,j})\Subset\sO$, noting that, by definition \eqref{eq:HestonWeight} of $\fw$,
\begin{equation}
\label{eq:NontranslationalInvariantWeightAdjustment}
\|f\|_{W^{k,p}(B_r)} \leq e^{\gamma\sqrt{1+\Upsilon^2}}\|f\|_{W^{k,p}(B_r,\fw)},
\end{equation}
for each ball $B_r\Subset\sO''$, together with $\|f\|_{W^{k,p}(B_r,\fw)} \leq \|f\|_{W^{2k,p}(\sO'',\fw)}$. This completes the proof.
\end{proof}

\subsection{Proofs of Corollary \ref{cor:CinftyGlobal}, Theorem \ref{thm:CkalphasHolderContinuityDomain}, Corollary \ref{cor:CkalphasHolderContinuityDomainStrip}, Theorem \ref{thm:ExistUniqueCk2+alphasHolderContinuityDomain}, and Corollaries \ref{cor:ExistUniqueCk2+alphasHolderContinuityDomain} and \ref{cor:Ck2+alphasHolderContinuityDomainStrip}}
\label{subsec:UniformHolderExponent}
We first have the easy

\begin{proof}[Proof of Corollary \ref{cor:CinftyGlobal}]
For any $z_0\in \sO$, there is a constant $R_0>0$ and a ball $B_{R_0}(z_0)$ such that $B_{R_0}(z_0)\subset\sO$ and \cite[Theorem 6.17]{GilbargTrudinger} implies that $u\in C^\infty(B_{R_0/2}(z_0))$. If $z_0\in\partial_0\sO$, there is a constant $R_0>0$ such that $\HH\cap B_{R_0}(z_0) \subset \sO$ and Theorem \ref{thm:HolderContinuity_Dku_Interior} implies that there is a positive constant  $R_1=R_1(k,R_0)<R_0$ such that $u\in C^{k,\alpha}_s(B_{R_1}^+(z_0))$, for any integer $k\geq 1$. Hence, by combining these observations, $u\in C^\infty(\underline\sO)$.
\end{proof}

In order to prove Theorem \ref{thm:CkalphasHolderContinuityDomain} and obtain $u \in C^{k,\alpha}_s(\underline\sO)$ with an a priori interior estimate \eqref{eq:CkalphasHolderContinuityDomain} on each pair of subdomains $\sO'\subset\sO''\subset\sO$ with $\sO'\Subset\underline\sO''$ and
$\sO''\subset (-\Upsilon, \Upsilon) \times (0,\Upsilon)$, we shall need to examine $u$ near the ``corner points'', $z_0\in\overline{\partial_0\sO}\cap\overline{\partial_1\sO}$, as well as $u$ near ``interior'' points, $z_0 \in \partial_0\sO$, and $u$ near points $z_0 \in \sO$ away from $\partial_0\sO$ (where classical results from \cite{GilbargTrudinger} apply). Otherwise, as noted prior to the statement of Proposition \ref{prop:CkalphasHolderContinuityDomain_Precompact}, we would not obtain a H\"older exponent, $\alpha\in(0,1)$, which is independent of $\sO'$ and $\sO''$.

\begin{proof}[Proof of Theorem \ref{thm:CkalphasHolderContinuityDomain}]
Choose $R_0=1$ and let $R_1=R_1(k)<1$ and $\alpha=\alpha(\Lambda,\nu_0,k,p)\in(0,1)$ be the constants defined by Theorem \ref{thm:HolderContinuity_Dku_Interior}. Since $f \in W^{k,p}_{\loc}(\underline\sO)$ because, a fortiori, $f \in W^{2k+2,p}_{\loc}(\underline\sO)$ by hypothesis, we know from the proof of Proposition \ref{prop:CkalphasHolderContinuityDomain_Precompact} that $u\in C^{k+1}(\sO) \subset C^{k,\alpha}(\sO)$ because of \eqref{eq:uCk+1alphaB}
and that the estimate \eqref{eq:uCk+1lambdaBestimate} holds for any ball $B_r\Subset\sO$.

To complete the proof that $u \in C^{k,\alpha}_s(\underline\sO)$, it remains to check that for every point $z_0\in\partial_0\sO$, there is an open ball
$B_r(z_0)$, for some $r>0$, such that $\HH\cap B_r(z_0)\Subset\underline\sO$ and $u \in C^{k,\alpha}_s(\bar B_{r/2}^+(z_0))$, \emph{with $\alpha\in(0,1)$ as fixed at the beginning of the proof}. According to Theorem \ref{thm:HolderContinuity_Dku_Interior}, for each point $z_0 \in \partial_0\sO$ such that $\HH\cap B_{R_0}(z_0) \subset \sO$, we have $u \in C^{k,\alpha}_s(\bar B_{R_1}^+(z_0))$.

It remains to consider points $z_0 \in \partial_0\sO$ such that $\HH\cap B_{R_0}(z_0) \not\subset \sO$; in fact, our analysis of this case is valid regardless of whether $\HH\cap B_{R_0}(z_0) \subset \sO$ or $\HH\cap B_{R_0}(z_0) \not\subset \sO$. Choose $r>0$ small enough that $r\leq R_1$ and $\HH\cap B_r(z_0)\Subset \underline\sO$. Let $\zeta\in C^\infty_0(\bar\HH)$ be a cutoff function such that $0\leq\zeta\leq 1$ on $\HH$ and $\zeta=1$ on $B_{r/2}^+(z_0)$ and $\supp\zeta\subset \underline B_r^+(z_0)$. To prove $u \in C^{k,\alpha}_s(\bar B_{r/2}^+(z_0))$, it suffices to show that
\begin{equation}
\label{eq:CkalphasHolderContinuityDomain_SufficientRegularityConclusion}
D_x^{\ell-m}D_y^m u \in C^\alpha_s(\bar B_{r/2}^+(z_0)), \quad 0\leq \ell \leq k, \quad 0\leq m \leq \ell,
\end{equation}
and since the argument will be similar for any $0\leq\ell\leq k$,
it is enough to consider $\ell=k$.

We have $f \in W^{k,p}_{\loc}(\underline\sO)$ since, a fortiori, $f \in W^{2k+2,p}_{\loc}(\underline\sO)$ by hypothesis, and therefore
$u\in\sH^{k+2}(B_r^+(z_0),\fw)$,  by Theorem \ref{thm:HkSobolevRegularityInterior}, and so $D_x^{k-m}D_y^m u \in H^1(B_r^+(z_0), \fw_m)$ and \break
$D_x^k u \in H^1(B_r^+(z_0), \fw)$,
by Definition \ref{defn:HkWeightedSobolevSpaceNormPowery} of $\sH^{k+2}(B_r^+(z_0),\fw)$. Thus, Proposition \ref{prop:VarEqn_Dkxy} implies that $D_x^{k-m}D_y^m u$ obeys the variational equation on $B_r^+(z_0)$,
$$
\fa_m(D_x^{k-m}D_y^m u, v) = (f_{k,m,u}, v)_{L^2(B_r^+(z_0),\fw_m)}, \quad\forall v \in H^1_0(\underline B_r^+(z_0),\fw_m),
$$
where
$
f_{k,m,u} := D_x^{k-m}D_y^m f - mBD_x^{k-m}D_y^{m-1} u.
$
Consequently, since $\supp\zeta \subset\underline\sO$, Lemma \ref{lem:SimpleCommutatorInnerProduct} implies that $\zeta D_x^{k-m}D_y^m u$ obeys the variational equation on $\HH$,
$$
\fa_m(\zeta D_x^{k-m}D_y^m u, v) = (\zeta f_{k,m,u} + [A, \zeta]D_x^{k-m}D_y^m u, v)_{L^2(\HH,\fw_m)}, 
$$
for all $v \in H^1_0(\HH,\fw_m)$.
Moreover, $\zeta D^{k-m}_xD_y^m u \in H^1(\HH, \fw)$ and, \emph{provided}
\begin{equation}
\label{eq:CutfkmuCommutatorLpIntegrability}
\zeta f_{k,m,u} + [A, \zeta]D_x^{k-m}D_y^m u \in L^p(B_r^+(z_0),y^{\beta-1}),
\end{equation}
noting that $\supp\zeta\subset \underline B_r^+(z_0)$ and $r\leq R_0$ (in fact, $r\leq R_1<R_0$), Theorem \ref{thm:HolderContinuity_Dku_Interior} will apply to give
$
\zeta D^{k-m}_xD_y^m u \in C^\alpha_s(\bar B_r^+(z_0)),
$
where $\alpha$ was fixed at the beginning of the proof, and a positive constant, $C=C(\Lambda,\nu_0,k,p)$, noting that $R_0=1$ and $\supp\zeta\subset\underline B_r^+(z_0)$, such that
\begin{equation}
\label{eq:CkalphasHolderContinuityDomain_Corner}
\begin{aligned}
{}&\|\zeta D_x^{k-m}D_y^m u\|_{C^\alpha_s(\bar B_r^+(z_0))}
\\
&\quad \leq C\left(\|\zeta  f_{k,m,u}\|_{L^p(B_r^+(z_0),y^{\beta-1})} + \|[A, \zeta]D_x^{k-m}D_y^m u\|_{L^p(B_r^+(z_0),y^{\beta-1})} \right.
\\
&\qquad + \left. \|\zeta D_x^{k-m}D_y^m u\|_{L^2(B_r^+(z_0),y^{\beta-1})}\right).
\end{aligned}
\end{equation}
We observe that
\begin{align*}
{}&\|\zeta f_{k,m,u}\|_{L^p(B_r^+(z_0),y^{\beta-1})}
\\
&\quad \leq \|D_x^{k-m}D_y^m f\|_{L^p(B_r^+(z_0), y^{\beta-1})} + \frac{m}{2}\|D_x^{k+1-m}D_y^{m-1} u\|_{L^p(B_r^+(z_0), y^{\beta-1})}
\\
&\qquad + \frac{m}{2}\|D_x^{k+2-m}D_y^{m-1} u\|_{L^p(B_r^+(z_0), y^{\beta-1})}
\\
&\quad \leq C\left(\|D_x^{k-m}D_y^m f\|_{L^p(B_r^+(z_0), y^{\beta-1})} + m\|D_x^{k+1-m}D_y^{m-1} u\|_{L^\infty(B_r^+(z_0))}\right.
\\
&\qquad\left.
+ m\|D_x^{k+2-m}D_y^{m-1} u\|_{L^\infty(B_r^+(z_0))}\right),
\end{align*}
for $C=C(r,p)$, a positive constant, while Lemma \ref{lem:CommutatorEstimate} implies that
\begin{align*}
{}&\|[A, \zeta]D_x^{k-m}D_y^m u\|_{L^p(B_r^+(z_0),y^{\beta-1})}
\\
&\quad \leq C\left( \|yD_x^{k+1-m}D_y^m u\|_{L^p(B_r^+(z_0), y^{\beta-1})} + \|yD_x^{k-m}D_y^{m+1} u\|_{L^p(B_r^+(z_0), y^{\beta-1})} \right.
\\
&\qquad + \left. \|(1+y)D_x^{k-m}D_y^m u\|_{L^p(B_r^+(z_0), y^{\beta-1})} \right)
\\
&\quad \leq C\left( \|D_x^{k+1-m}D_y^m u\|_{L^\infty(B_r^+(z_0))} + \|D_x^{k-m}D_y^{m+1} u\|_{L^\infty(B_r^+(z_0))}
+ \|D_x^{k-m}D_y^m u\|_{L^\infty(B_r^+(z_0))} \right),
\end{align*}
where $C=C(\Lambda,\nu_0,r)$, a positive constant, and $\zeta$ is chosen such that $\|\zeta\|_{C^2(\bar B_r^+(z_0))} \leq Mr^{-2}$, where $M>0$ is a universal constant.
But $f\in W^{2k+2,p}_{\loc}(\underline\sO)$ by hypothesis, and from Proposition \ref{prop:CkalphasHolderContinuityDomain_Precompact} (applied with $k$ replaced by $k+1$ and $p>\max\{4, 2+(k+1)+\beta\} = \max\{4, 3+k+\beta\}$), we know that
\begin{align*}
D_x^{k+1-m}D_y^{m-1} u, \quad D_x^{k+2-m}D_y^{m-1} u &\in C(\bar B_r^+(z_0)), \quad 1\leq m \leq k,
\\
D_x^{k+1-m}D_y^m u, \quad D_x^{k-m}D_y^{m+1} u, \quad D_x^{k-m}D_y^m u &\in C(\bar B_r^+(z_0)), \quad 0\leq m \leq k,
\end{align*}
since $B_r^+(z_0) \Subset \underline\sO$ and $u \in C^{k+1}(\bar B_r^+(z_0))$. The preceding inequalities and boundedness conditions ensure that the integrability condition \eqref{eq:CutfkmuCommutatorLpIntegrability} holds.

In particular, $\zeta D_x^{k-m}D_y^m u \in C^\alpha_s(\bar B_{R_1}^+(z_0))$ and, because $\zeta=1$ on $B_{r/2}^+(z_0)$, we obtain
that $D_x^{k-m}D_y^m u \in C^\alpha_s(\bar B_{r/2}^+(z_0))$, as desired. This completes the proof of
\eqref{eq:CkalphasHolderContinuityDomain_SufficientRegularityConclusion} (when $\ell=k$) and hence that $u \in C^{k,\alpha}_s(\underline\sO)$.

Finally, we prove the a priori estimate \eqref{eq:CkalphasHolderContinuityDomain}. Since $\sO'\Subset\underline\sO''$ and $\dist(\partial_1\sO',\partial_1\sO'')\geq d_1$ and $\sO''\subset(-\Upsilon,\Upsilon)\times(0,\Upsilon)$ by hypothesis, there are a finite number of balls of radius $r:=\min\{d_1/4, R_1\}$ (the number of balls is determined by $r=r(d_1,R_1)=r(d_1,k)$ and $\Upsilon$) such that
\begin{equation}
\label{eq:OpenCoverOprime}
\sO' \subset \bigcup_{i,j} B_{r/2}^+(z_{0,i})\cup B_{r/2}(z_{1,j}),
\end{equation}
with
$\HH\cap B_{4r}(z_{0,i}) \subset \sO''$ and $B_{4r}(z_{1,j}) \subset \sO''$. For each $i$, we let $\zeta_i\in C^\infty_0(\bar\HH)$ be a cutoff function such that $0\leq\zeta_i\leq 1$ on $\HH$, and $\zeta_i=1$ on $B_{r/2}^+(z_{0,i})$, and $\supp\zeta_i\subset \underline B_r^+(z_{0,i})$, and $\|\zeta_i\|_{C^2(\bar B_r^+(z_{0,i}))} \leq Mr^{-2}$.  The preceding $L^p$ estimate for $\zeta f_{k,m,u}$ yields
$$
\|\zeta_i f_{k,m,u}\|_{L^p(B_r^+(z_{0,i}))}
\leq
C\left(\|f\|_{W^{k,p}(B_r^+(z_{0,i})), y^{\beta-1})} + \|u\|_{C^{k+1}_s(\bar B_r^+(z_{0,i}))}\right), 
$$
for all $0\leq m\leq k$,
with $C=C(k,p,r)$. Thus, Proposition \ref{prop:CkalphasHolderContinuityDomain_Precompact} (noting that $B_r^+(z_0)\Subset\underline B_{2r}^+(z_0)$ and $B_{2r}^+(z_0)\Subset\underline\sO''$ with $\dist(\partial_1 B_{2r}^+(z_0), \partial_1 \sO'')\geq 2r$ since $B_{4r}^+(z_0)\subset\sO''$) and the definition  of
$\fw$
in \eqref{eq:HestonWeight} gives
\begin{equation}
\label{eq:Cutoff_fkmu}
\|\zeta_i f_{k,m,u}\|_{L^p(B_r^+(z_{0,i}))}
\leq
C\left(\|f\|_{W^{2k+2,p}(B_{2r}^+(z_{0,i})), \fw)} + \|u\|_{L^2(B_{2r}^+(z_{0,i})), \fw)}\right), 
\end{equation}
for all $0\leq m\leq k$,
with $C=C(\Lambda,\nu_0,d_1,k,p,r)=C(\Lambda,\nu_0,d_1,k,p)$. Similarly, the preceding $L^p$ estimate for
the commutator $[A, \zeta]D_x^{k-m}D_y^m u$ yields
$$
\|[A, \zeta_i]D_x^{k-m}D_y^m u\|_{L^p(B_r^+(z_{0,i}))}
\leq
C\left(\|f\|_{W^{k,p}(B_r^+(z_{0,i}), y^{\beta-1})} + \|u\|_{C^{k+1}_s(\bar B_r^+(z_{0,i}))}\right), 
$$
for all $0\leq m\leq k$,
with $C=C(\Lambda,\nu_0,p,r)$. Thus, Proposition \ref{prop:CkalphasHolderContinuityDomain_Precompact} now gives, for all $0\leq m\leq k$,
\begin{equation}
\label{eq:CutoffCommutatorDkmu}
\|[A, \zeta_i]D_x^{k-m}D_y^m u\|_{L^p(B_r^+(z_{0,i}))}
\leq
C\left(\|f\|_{W^{2k+2,p}(B_{2r}^+(z_{0,i})), \fw)} + \|u\|_{L^2(B_{2r}^+(z_{0,i})), \fw)}\right),
\end{equation}
with $C=C(\Lambda,\nu_0,d_1,k,p)$, a positive constant. By combining \eqref{eq:CkalphasHolderContinuityDomain_Corner}, \eqref{eq:Cutoff_fkmu}, and \eqref{eq:CutoffCommutatorDkmu}, and recalling
that $\zeta_i=1$ on $B_{r/2}^+(z_{0,i})$ and $\supp\zeta_i\subset\underline B_r^+(z_{0,i})$, we obtain, for $0\leq m\leq k$,
$$
\|D_x^{k-m}D_y^m u\|_{C^\alpha_s(B_{r/2}^+(z_{0,i}))}
\leq
C\left(\|f\|_{W^{2k+2,p}(B_{2r}^+(z_{0,i}), \fw)} + \|u\|_{L^2(B_{2r}^+(z_{0,i}), \fw)}\right),
$$
with $C=C(\Lambda,\nu_0,d_1,k,p)$, a positive constant. Therefore, by the same argument, for any $0\leq m\leq \ell\leq k$, we have
\begin{equation}
\label{eq:CkalphasHolderContinuityDomain_NearCorner}
\|D_x^{\ell-m}D_y^m u\|_{C^\alpha_s(B_{r/2}^+(z_{0,i}))}
\leq
C\left(\|f\|_{W^{2k+2,p}(B_{2r}^+(z_{0,i}), \fw)} + \|u\|_{L^2(B_{2r}^+(z_{0,i}), \fw)}\right),
\end{equation}
with $C=C(\Lambda,\nu_0,d_1,k,p)$, a positive constant.

On the other hand, by applying \eqref{eq:uCk+1lambdaBestimate} to the balls $B_r(z_{1,j})$, we obtain
\begin{equation}
\label{eq:uCk+1lambdaBestimate_FullyInteriorBall}
\|u\|_{C^{k+1}(\bar B_{r/2}(z_{1,j}))}
\leq
C\left(\|f\|_{W^{2k,p}(B_r(z_{1,j}))} + \|u\|_{L^2(B_r(z_{1,j}))}\right),
\end{equation}
with $C=C(\Lambda,\nu_0,k,p,r)=C(\Lambda,\nu_0,d_1,k,p)$, a positive constant. The desired a priori estimate \eqref{eq:CkalphasHolderContinuityDomain} now follows by combining the a priori estimates \eqref{eq:CkalphasHolderContinuityDomain_NearCorner} and \eqref{eq:uCk+1lambdaBestimate_FullyInteriorBall}, noting that
$C^{k+1}(\bar B_{r/2}^+(z_{1,j})) \hookrightarrow C^{k,\alpha}(\bar B_{r/2}^+(z_{1,j}))$ and using \eqref{eq:NontranslationalInvariantWeightAdjustment} to give
$$
\|f\|_{W^{2k,p}(B_r(z_{1,j}))} \leq e^{\gamma\sqrt{1+\Upsilon^2}}\|f\|_{W^{2k,p}(B_r(z_{1,j}),\fw)},
$$
and $\|f\|_{W^{2k,p}(B_r(z_{1,j}),\fw)} \leq \|f\|_{W^{2k,p}(\sO'',\fw)}$. This completes the proof.
\end{proof}

Next, we have the

\begin{proof}[Proof of Corollary \ref{cor:CkalphasHolderContinuityDomainStrip}]
As in the proof of Theorem \ref{thm:HkSobolevRegularityDomain}, it suffices to choose a cover \eqref{eq:OpenCoverOprime} of $\sO'$ by open balls $B_{r/2}(z_{1,j})$ or half-balls $B_{r/2}^+(z_{0,j})$ contained in $\sO''$ which is \emph{uniformly locally finite}. Again using the definition  of
$\fw$
in \eqref{eq:HestonWeight} to replace the integral weights $\fw$ in \eqref{eq:CkalphasHolderContinuityDomain_NearCorner} and $1$ in \eqref{eq:uCk+1lambdaBestimate_FullyInteriorBall}, respectively, by $y^{\beta-1}$ on the right-hand side, and arguing just as in the proof of Proposition \ref{prop:HolderContinuity_ux} to eliminate factors such as $e^{\frac{\gamma}{2}|x_{0,i}|}$ or $e^{\frac{\gamma}{2}|x_{1,j}|}$, we see that
\begin{align*}
\|u\|_{C^{k,\alpha}_s(\bar B_{r/2}^+(z_{0,i}))} &\leq C\left(\|f\|_{W^{2k+2,p}(B_{2r}^+(z_{0,i}), y^{\beta-1})} + \|u\|_{L^2(B_{2r}^+(z_{0,i}), y^{\beta-1})}\right)
\quad\hbox{(by \eqref{eq:CkalphasHolderContinuityDomain_NearCorner})},
\\
\|u\|_{C^{k+1}(\bar B_{r/2}(z_{1,j}))}
&\leq
C\left(\|f\|_{W^{2k,p}(B_r(z_{1,j}), y^{\beta-1})} + \|u\|_{L^2(B_r(z_{1,j}), y^{\beta-1})}\right)
\quad\hbox{(by \eqref{eq:uCk+1lambdaBestimate_FullyInteriorBall})},
\end{align*}
with $C=C(\Lambda,\nu_0,d_1,k,p)$, a positive constant. Therefore, using the fact that $C^{k+1}(\bar B_{r/2}(z_{1,j})) \hookrightarrow C^{k,\alpha}_s(\bar B_{r/2}(z_{1,j}))$, we obtain
$$
\sup_i\|u\|_{C^{k,\alpha}_s(\bar B_{r/2}^+(z_{0,i})} + \sup_j\|u\|_{C^{k,\alpha}_s(\bar B_{r/2}(z_{1,j})}
\leq
C\left(\|f\|_{W^{2k+2,p}(\sO'', y^{\beta-1})} + \|u\|_{L^2(\sO'', y^{\beta-1})}\right),
$$
with $C=C(\Lambda,\nu_0,d_1,k,p)$, a positive constant. Since
$$
\|u\|_{C^k(\bar\sO')} \leq \sup_i \|u\|_{C^k(\bar B_{R_1}^+(z_{0,i}))} + \sup_j \|u\|_{C^k(\bar B_{R_1}(z_{1,j}))},
$$
and, denoting
$$
[D^k u]_{C^\alpha_s(\bar U)} := \max_{0\leq m\leq k}[D_x^{k-m}D_y^m u]_{C^\alpha_s(\bar U)},
$$
for any open set $U\subset\HH$, we see that
\begin{align*}
\|u\|_{C^{k,\alpha}_s(\bar\sO')} &\equiv \|u\|_{C^k(\bar\sO')} + [D^k u]_{C^\alpha_s(\bar\sO')} \quad\hbox{(Definition \ref{defn:DHspaces})}
\\
&\leq C(r)\|u\|_{C^k(\bar\sO')} + \sup_i[D^k u]_{C^\alpha_s(\bar B_{r/2}^+(z_{0,i}))} + \sup_j[D^k u]_{C^\alpha_s(\bar B_{r/2}(z_{1,j}))}.
\end{align*}
Combining the preceding estimates and recalling that $r\equiv\min\{d_1/4,R_1\}$ in \eqref{eq:OpenCoverOprime}, so $r=r(d_1,k)$, yields the desired a priori bound \eqref{eq:CkalphasHolderContinuityDomainStrip} for $\|u\|_{C^{k,\alpha}_s(\bar\sO')}$.
\end{proof}

Lastly, we turn to the

\begin{proof}[Proof of Theorem \ref{thm:ExistUniqueCk2+alphasHolderContinuityDomain}]
By hypothesis, we have $f\in C(\bar\sO)$ and $(1+y)g \in W^{2,\infty}(\sO)$, since $g \in C(\bar\sO)$ with $(1+y)g\in C^2(\bar\sO)$. Therefore, Theorem \ref{thm:ExistUniqueHkSobolevRegularityDomain} (with $k=0$) implies that there exists a unique
solution $u \in H^1(\sO,\fw)$ to the variational equation \eqref{eq:HestonVariationalEquation}, with boundary condition $u-g\in H^1_0(\underline\sO,\fw)$.

By \cite[Corollary 8.28]{GilbargTrudinger}, we must have $u\in C(\sO\cup\partial_1\sO)$, since $u=g$ on $\partial\sO$ in the sense of $H^1(\sO,\fw)$ and $g\in C(\overline{\partial_1\sO})$, so $u=g$ on $\partial_1\sO$, that is, $u$ obeys the boundary condition \eqref{eq:IntroBoundaryValueProblemBC}. Because $f\in C(\bar\sO)$ by hypothesis, the maximum principle \cite[Theorem 8.15]{Feehan_maximumprinciple}
implies that $u$ is bounded, that is, $u\in L^\infty(\sO)$.

By hypothesis, we also have $f\in C^{2k+6,\alpha}_s(\underline\sO)$. For any $1\leq p\leq\infty$, there is a continuous embedding $C^{2k+6,\alpha}_s(\bar U) \hookrightarrow W^{2k+6, p}(U, \fw)$, for any subset $U\Subset\bar\HH$, by Definition \ref{defn:DHspaces} of $C^{\ell,\alpha}_s(\bar U)$, for $\ell\geq 0$, and hence $f\in W^{2k+6,p}_{\loc}(\underline\sO, \fw)=W^{2(k+2)+2,p}_{\loc}(\underline\sO, \fw)$. Choose
$p = \max\{4, 3+(k+2)+\beta\}+1 = 6+k+\beta$ and observe that Theorem \ref{thm:CkalphasHolderContinuityDomain} implies that $u\in C^{k+2,\alpha}_s(\underline\sO)$. From the Definition \ref{defn:DH2spaces} of $C^{k,2+\alpha}_s(\bar U)$, it follows that there is a continuous embedding $C^{k+2,\alpha}_s(\bar U) \hookrightarrow C^{k,2+\alpha}_s(\bar U)$, for any open set $U\Subset\bar\HH$. Hence, it follows that $u\in C^{k,2+\alpha}_s(\underline\sO)$.

Since the equation \eqref{eq:IntroBoundaryValueProblem} and the desired Schauder a priori estimate \eqref{eq:ExistUniqueCk2+alphasHolderContinuityDomain} are invariant under translations with respect to $x$ and $\diam(\sO'')\leq\Upsilon$ by hypothesis, we may assume without loss of generality that $\sO'' \subset(-\Upsilon,\Upsilon)\times(0,\Upsilon)$. Therefore, the desired Schauder estimate \eqref{eq:ExistUniqueCk2+alphasHolderContinuityDomain} follows from Theorem \ref{thm:CkalphasHolderContinuityDomain} and the a priori estimate \eqref{thm:CkalphasHolderContinuityDomain} and the fact that
$$
\|u\|_{C^{k,2+\alpha}_s(\bar\sO')} \leq C\|u\|_{C^{k+2,\alpha}_s(\bar\sO')},
$$
where $C=C(\Upsilon)$ is a positive constant.
\end{proof}

\begin{proof}[Proof of Corollary \ref{cor:ExistUniqueCk2+alphasHolderContinuityDomain}]
Theorem \ref{thm:MainContinuityBoundaryL2RHSu} implies that $u\in C^\alpha_s(\bar B_{R_1}^+(z_0))$, for each corner point $z_0\in\overline{\partial_0\sO}\cap\overline{\partial_1\sO}$. Thus, $u\in C^\alpha_{s,\loc}(\bar\sO)$ and, because $u$ is bounded, we obtain $u \in C(\bar\sO)$. Moreover, because $u$ is uniformly $C^\alpha_s(\bar B^+)$-H\"older continuous, for all balls $B\subset\RR^2$, where $\alpha=\alpha(\Lambda,\nu_0,k,K)\in(0,1)$ is the smallest H\"older exponent in Theorems \ref{thm:CkalphasHolderContinuityDomain} and \ref{thm:MainContinuityBoundaryL2RHSu}
and in \cite[Theorem 8.29]{GilbargTrudinger}, we have $u \in C^\alpha_{\loc,s}(\bar\sO)$.
\end{proof}

\begin{proof}[Proof of Corollary \ref{cor:Ck2+alphasHolderContinuityDomainStrip}]
The desired Schauder estimate \eqref{eq:Ck2+alphasHolderContinuityDomainStrip} follows by replacing the role of the a priori estimate \eqref{thm:CkalphasHolderContinuityDomain} with that of the a priori estimate \eqref{eq:CkalphasHolderContinuityDomainStrip} in the proof of Theorem \ref{thm:ExistUniqueCk2+alphasHolderContinuityDomain}.
\end{proof}

\appendix

\section{Appendix}
\label{sec:Appendix}
For the convenience of the reader, we collect here some useful facts from some our earlier articles for easier reference, together with some technical proofs of results used in the body of this article. In \S \ref{app:ApproximationSmoothFunctions}, we describe approximation results for the weighted Sobolev spaces appearing in this article and which are used, for example, to prove integration-by-parts formulae, as illustrated in \S \ref{app:IntegrationByParts}. Finally, in \S \ref{app:VarEqn_DkxyAuxiliarymZero}, we explain the need for some of the technical hypotheses in Proposition \ref{prop:VarEqn_Dkxy}.

\subsection{Approximation by smooth functions}
\label{app:ApproximationSmoothFunctions}
We begin with the

\begin{defn}[$C^1$-orthogonal curves in the upper half-space]
\label{defn:C1Orthogonal}
We say that a curve $T\subset\HH$ is \emph{uniformly $C^1$-orthogonal to $\partial\HH$} if $T$ is a relatively open $C^1$-curve and there is a positive constant, $\delta$, such that for each point $z_0 = (x_0,0) \in \bar T\cap\partial\HH$ we have
$$
\bar T\cap B_\delta(z_0) \subset \{(x_0,y) \in \RR^2: y \geq 0\}.
$$
\end{defn}

The next approximation result\footnote{While the conclusion holds for weaker hypotheses on $\partial_1\sO$, the result suffices for applications in this article and counterexamples show that some conditions on the regularity of $\partial_1\sO$ and the geometry of its intersection with $\partial\HH$ are required.} follows from \cite[Corollary A.12]{Daskalopoulos_Feehan_statvarineqheston}.

\begin{thm}[Density of smooth functions]
\label{thm:KufnerPowerWeightBoundedDerivatives}
Let $\sO\subset\HH$ be a domain such that $\partial_1\sO$ is uniformly $C^1$-orthogonal to $\partial\HH$. Then $C^\infty_0(\bar\sO)$ is a dense subset of $H^k(\sO,\fw)$, and $\sH^k(\sO,\fw)$, and $W^k(\sO,\fw)$ for all integers $k\geq 0$.
\end{thm}

\begin{proof}
When $k=0,1,2$, the conclusion for $H^k(\sO,\fw) = \sH^k(\sO,\fw)$ is given by \cite[Corollary A.12]{Daskalopoulos_Feehan_statvarineqheston}, which asserts that $C^\infty_0(\bar\sO)$ is a dense subset of $H^k(\sO,\fw)$. The proof of \cite[Corollary A.12]{Daskalopoulos_Feehan_statvarineqheston} extends easily to include the remaining cases.
\end{proof}

\subsection{Integration by parts}
\label{app:IntegrationByParts}
We recall the special case of \cite[Lemma 2.23]{Daskalopoulos_Feehan_statvarineqheston}; no hypothesis on $\partial_1\sO$ is required here because we assume $v\in H^1_0(\underline\sO,\fw)$ rather than allow any $v\in H^1(\sO,\fw)$.

\begin{lem}[Integration by parts for the Heston operator]
\label{lem:HestonIntegrationByParts}
Let $\sO\subseteqq\HH$ be a domain. If $u\in H^2(\sO,\fw)$ and $v\in H^1_0(\underline\sO,\fw)$, then $Au\in L^2(\sO,\fw)$ and
\begin{equation}
\label{eq:HestonIntegrationByPartsFormula}
(Au,v)_{L^2(\sO,\fw)} = \fa(u,v).
\end{equation}
\end{lem}

\begin{proof}
When $\tilde u \in C^\infty_0(\bar\sO)$ and $\tilde v \in C^\infty_0(\underline\sO)$, we obtain
$$
(A\tilde u, \tilde v)_{L^2(\sO,\fw)} = \fa(\tilde u, \tilde v)
$$
by direct calculation, as in the proof of \cite[Lemma 2.23]{Daskalopoulos_Feehan_statvarineqheston}. Because $\supp\tilde v \subset \underline\sO$ is compact, we may choose a subdomain $\sO'\Subset\underline\sO$ such that $\partial_1\sO'$ is uniformly $C^1$-orthogonal to $\partial\HH$ and $\supp\tilde v\subset\underline\sO'$. If $u\in H^2(\sO,\fw)$, Theorem \ref{thm:KufnerPowerWeightBoundedDerivatives} implies that there is a sequence $\{u_n\}_{n\in\NN} \subset C^\infty_0(\bar\sO')$ such that $u_n\to u$ strongly in
$H^2(\sO',\fw)$ as $n\to\infty$ and thus
$$
(Au, \tilde v)_{L^2(\sO,\fw)} = (Au, \tilde v)_{L^2(\sO',\fw)} = \lim_{n\to\infty}(Au_n, \tilde v)_{L^2(\sO',\fw)} = \lim_{n\to\infty} \fa(u_n, \tilde v)
= \fa(u, \tilde v),
$$
because $A:H^2(\sO,\fw) \to L^2(\sO,\fw)$ is a continuous linear operator and $\fa: H^1(\sO,\fw)\times H^1(\sO,\fw) \to \RR$ is a continuous bilinear map. Since $v \in H^1_0(\underline\sO,\fw)$, there is a sequence $\{v_n\}_{n\in\NN} \subset C^\infty_0(\underline\sO)$ such that $v_n\to v$ strongly in $H^1(\sO,\fw)$ as $n\to\infty$ and thus
$$
(Au, v)_{L^2(\sO,\fw)} = \lim_{n\to\infty}(Au, v_n)_{L^2(\sO,\fw)} = \lim_{n\to\infty} \fa(u, v_n)
= \fa(u, v).
$$
This completes the proof.
\end{proof}

\subsection{Need for the auxiliary regularity condition in Proposition \ref{prop:VarEqn_Dkxy}}
\label{app:VarEqn_DkxyAuxiliarymZero}
We explain the role of the hypothesis, $D_x^k u \in H^1(\sO,\fw)$, in the statement of Proposition \ref{prop:VarEqn_Dkxy}.

First, we explain the role of the auxiliary regularity condition when $m=0$ in \eqref{eq:IntroHestonWeakMixedProblemHomogeneous_Dkxy}. If $k=1, m=0$ and $u \in H^2(\sO,\fw)$, then we recall from \eqref{eq:H2WeightedSobolevSpace} that while this ensures $(1+y)^{1/2}u_x$ belongs to $L^2(\sO,\fw)$, it does not imply that $y^{1/2}u_{xx}, \ y^{1/2}u_{xy}$ belong to $L^2(\sO,\fw)$, and so $u \in H^2(\sO,\fw)$ does not imply $u_x \in H^1(\sO,\fw)$. However, when $k=1, m=1$, we have seen that $u \in H^2(\sO,\fw)$ does imply $u_y \in H^1(\sO,\fw_1)$.

If $u \in \sH^{k+1}(\sO,\fw)$ and $k\geq 2$, then we recall from Definition \ref{defn:HkWeightedSobolevSpaceNormPowery} that
\begin{align*}
yD_x^{k+1-m}D_y^m u &\in \begin{cases}L^2(\sO,\fw_{m-2}), &3\leq m \leq k, \\ L^2(\sO,\fw), &m =1,2, \end{cases}
\\
yD_x^{k-m}D_y^{m+1} u &\in \begin{cases}L^2(\sO,\fw_{m-1}), &2\leq m \leq k, \\ L^2(\sO,\fw), &m =1, \end{cases}
\end{align*}
that is,
\begin{align*}
y^{1/2}D_x^{k+1-m}D_y^m u &\in \begin{cases}L^2(\sO,\fw_{m-1}), &3\leq m \leq k, \\ L^2(\sO,\fw_1), &m =1,2, \end{cases}
\\
y^{1/2}D_x^{k-m}D_y^{m+1} u &\in \begin{cases}L^2(\sO,\fw_m), &2\leq m \leq k, \\ L^2(\sO,\fw_1), &m =1. \end{cases}
\end{align*}
Note that $L^2(\sO,\fw_{m-1}) \subset L^2(\sO,\fw_m)$, for all $m\geq 1$. Moreover, $u \in \sH^{k+1}(\sO,\fw)$ with $k\geq 2$ implies $(1+y)D_x^{k-m}D_y^m u \in L^2(\sO,\fw_{m-2}) \subset L^2(\sO,\fw_m)$ when $3\leq m\leq k$, and $(1+y)D_x^{k-m}D_y^m u \in L^2(\sO,\fw) \subset L^2(\sO,\fw_m)$, when $m=0,1,2$. Thus, for $k\geq 2$,
$$
u \in \sH^{k+1}(\sO,\fw) \implies D_x^{k-m}D_y^m u \in H^1(\sO,\fw_m), \quad k\geq 2, \ 1 \leq m \leq k.
$$
However, when $k\geq 2$ and $m=0$, the auxiliary condition $D_x^k u \in H^1(\sO,\fw)$ required for the \emph{left-hand side} of \eqref{eq:IntroHestonWeakMixedProblemHomogeneous_Dkxy} to be well-defined is not implied by the hypothesis $u \in \sH^{k+1}(\sO,\fw)$, since the latter condition implies $yD_x^{k+1} u, \ yD_x^kD_y u \in L^2(\sO,\fw)$ but not $y^{1/2}D_x^{k+1} u, \ y^{1/2}D_x^kD_y u \in L^2(\sO,\fw)$.

Second, we explain the role of the auxiliary regularity condition when $m=1$ in \eqref{eq:IntroHestonWeakMixedProblemHomogeneous_Dkxy}. If $u \in \sH^{k+1}(\sO,\fw)$ and $k\geq 2$, then we recall from Definition \ref{defn:HkWeightedSobolevSpaceNormPowery} that
$$
yD_x^{k+2-m}D_y^{m-1} u \in \begin{cases}L^2(\sO,\fw_{m-3}), &m \geq 4, \\ L^2(\sO,\fw), &m =1,2,3, \end{cases}
$$
that is,
$$
D_x^{k+2-m}D_y^{m-1} u \in \begin{cases}L^2(\sO,\fw_{m-1}), &m \geq 4, \\ L^2(\sO,\fw_2), &m =1,2,3. \end{cases}
$$
Hence, when $k\geq 2$ and $m=1$, the auxiliary condition $D_x^{k+1} u \in L^2(\sO,\fw_1)$ required for the \emph{right-hand side} of \eqref{eq:IntroHestonWeakMixedProblemHomogeneous_Dkxy} to be well-defined is not implied by the hypothesis $u \in \sH^{k+1}(\sO,\fw)$. However, the condition $D_x^k u \in H^1(\sO,\fw)$ ensures, by definition \eqref{eq:H1WeightedSobolevSpace} of $H^1(\sO,\fw)$, that $y^{1/2}D_x^{k+1} u \in L^2(\sO,\fw)$ or, equivalently, $D_x^{k+1} u \in L^2(\sO,\fw_1)$.

%
%

\bibliography{mfpde}
\bibliographystyle{amsplain}

\end{document}